\RequirePackage{fix-cm}

\documentclass[11pt,leqno,a4paper,oneside]{amsart}
\usepackage[T1,T2A]{fontenc}

\usepackage[english]{babel}
\usepackage{indentfirst}

\usepackage{amssymb,amsthm,enumerate}
\usepackage{wasysym}
\usepackage{array}
\usepackage{enumitem}
\usepackage[pdftex,bookmarks=true]{hyperref}                        
\usepackage[capitalise,nameinlink]{cleveref}
\usepackage{textcomp}
\usepackage{array,blkarray}
\usepackage{multirow,bigdelim}
\usepackage{enumitem}
\usepackage[normalem]{ulem}
\usepackage[all]{xy}

\usepackage[pdftex]{graphicx}
\usepackage[pdftex,table]{xcolor}
\usepackage{makecell}

\usepackage{boldline, multirow}

\usepackage{tikz}                                
\usetikzlibrary{angles}
\usetikzlibrary{decorations.markings,calc}
\usetikzlibrary{cd}
\usepackage{longtable}
\usepackage{mathrsfs}
\usepackage{hhline}
\usepackage{multicol}
\usepackage{multirow}
\usepackage{blkarray} 
\usepackage{lscape}
\usepackage{rotating}
\usepackage[font=sc]{caption}
\usepackage{adjustbox}
\usepackage{comment}
\usepackage{booktabs}
\usepackage{tabularx}
\newcolumntype{Y}{>{\centering\arraybackslash}X}
\usepackage{nicematrix}
\usepackage{mathtools}
\usepackage{longtable}

\usepackage{float}

\usepackage{verbatim}

\usepackage{scalerel,stackengine}
\stackMath
\newcommand\reallywidehat[1]{%
\savestack{\tmpbox}{\stretchto{%
  \scaleto{%
    \scalerel*[\widthof{\ensuremath{#1}}]{\kern-.6pt\bigwedge\kern-.6pt}%
    {\rule[-\textheight/2]{1ex}{\textheight}}
  }{\textheight}%
}{0.5ex}}%
\stackon[1pt]{#1}{\tmpbox}%
}

\usepackage{doi} 

\input xy
\xyoption{all}

\entrymodifiers={!!<0pt,0.7ex>+}

\hypersetup{pdffitwindow=true,
colorlinks=true,
citecolor=black,
filecolor=black,
linkcolor=black,
urlcolor=black,
hypertexnames=true}

\newcommand{\IZ}{{\mathbb{Z}}}

\newcommand{\cO}{{\mathcal{O}}}

\newcommand{\propernormalsubgroup}{%
\mathrel{\ooalign{$\lneq$\cr\raise.22ex\hbox{$\lhd$}\cr}}}

\DeclareMathOperator{\Hom}{Hom}             
\DeclareMathOperator{\Syl}{Syl}                  
\DeclareMathOperator{\Res}{Res}               
\DeclareMathOperator{\Ind}{Ind}                  
\DeclareMathOperator{\Inf}{Inf}                    
\DeclareMathOperator{\tr}{tr}			

\DeclareMathOperator{\Irr}{Irr}
\DeclareMathOperator{\IBr}{IBr}
\DeclareMathOperator{\Lin}{Lin}

\DeclareMathOperator{\TS}{TS}

\DeclareMathOperator{\Soc}{Soc}
\DeclareMathOperator{\Rad}{Rad}

\DeclareMathOperator{\Sc}{Sc}

\DeclareMathOperator{\Triv}{Triv}

\newcommand{\TrivialGroup}{\{1\}}

\renewcommand{\C}{{\mathbb C}}

\newcommand{\Hd}{\operatorname{Hd}}

\newcommand{\GAP}{\textsf{GAP}}

\newcommand{\SymmetricGroup}{\mathfrak{S}}

\newcommand{\AlternatingGroup}{\mathfrak{A}}

\newcommand{\xddots}{%
	\raise 14pt \hbox {.}
	\mkern 6mu
	\raise 4pt \hbox {.}
	\mkern 6mu
	\raise -6pt \hbox {.}
}

\newcommand{\xvdots}{%
	\raise 14pt \hbox {.}
	\mkern -5.6mu
	\raise 4pt \hbox {.}
	\mkern -5.1mu
	\raise -6pt \hbox {.}
}

\newcommand{\xhdots}{%
	\raise 5.5pt \hbox {.}
	\mkern 6mu
	\raise 5.5pt \hbox {.}
	\mkern 6mu
	\raise 5.5pt \hbox {.}
}

\newcommand{\Bl}{\textup{Bl}}

\let\lra=\longrightarrow

\newtheoremstyle{thmnew}{3ex}{3ex}{\itshape}{}{\bf}{.}{.5em}{}
\theoremstyle{thmnew}
\newtheorem{thm}{Theorem}[section]
\newtheorem{lem}[thm]{Lemma}

\newtheorem{cor}[thm]{Corollary}
\newtheorem{prop}[thm]{Proposition}

\newtheoremstyle{defnew}{3ex}{3ex}{}{}{\bf}{.}{.5em}{}
\theoremstyle{defnew}

\newtheorem{defn}[thm]{Definition}

\newtheorem{conv}[thm]{Convention}

\newtheorem{nota}[thm]{Notation}
\newtheorem{properties}[thm]{Properties}

\theoremstyle{remark}
\newtheorem{rem}[thm]{Remark}

\makeatletter
\renewcommand*\env@matrix[1][\
arraystretch]{%
	\edef\arraystretch{#1}%
	\hskip -\arraycolsep
	\let\@ifnextchar\new@ifnextchar
	\array{*\c@MaxMatrixCols c}}
\makeatother

\begin{document}
\hyphenrules{english}

	
        \title[Trivial source characters in blocks of domestic representation type]{Trivial source characters in blocks of domestic representation type}

	
        \author{{Bernhard B\"ohmler}}
        \address{{\sc Bernhard B\"ohmler},  Leibniz Universit\"at Hannover, Institut f\"ur Algebra, Zahlentheorie und Diskrete Mathematik, Welfengarten 1, 30167 Hannover, Germany}
        \email{boehmler@math.uni-hannover.de}

        \date{\today}

        \begin{abstract}
        Let $G$ be a finite group of even order, let $k$ be an algebraically closed field of characteristic~$2$, and let $B$ be a block of the group algebra $kG$ which is of domestic representation type. Up to splendid Morita equivalence, precisely three cases can occur:
        $kV_4$, $k\AlternatingGroup_4$ and the principal block of $k\AlternatingGroup_5$. In each case, given the character values of the ordinary irreducible characters of $B$, we determine the ordinary characters of all trivial source~$B$-modules.
	\end{abstract}

\noindent\thanks{This work is part of a doctoral thesis prepared at the RPTU Kaiserslautern-Landau under the supervision of Caroline Lassueur.}
	
	\maketitle

	
\pagestyle{myheadings}
\markboth{B. B\"ohmler}{Trivial source characters in blocks of domestic representation type}

\section{Introduction}
Trivial source modules play a prominent r{\^o}le in modular representation theory of finite groups. They are, by definition, the indecomposable direct summands of the permutation modules and occur as elementary building blocks for the construction of various categorical equivalences between block algebras, such as splendid Rickard equivalences, $p$-permutation equivalences, source-algebra equivalences, or Morita equivalences with endo-permutation source.\par
Any trivial source module can be lifted to characteristic zero and affords a well-defined ordinary character, which contains important pieces of information about its structure. A determination of these ordinary characters is therefore essential.\par
A complete classification of trivial source modules and their characters is not easy to obtain in general. For this reason, it is natural to start with modules belonging to blocks with defect groups of small order. In finite representation type, which occurs precisely when the defect groups are cyclic, such a classification was accomplished  by Hi{\ss}-Lassueur and Koshitani-Lassueur in \cite{HissLassueurCyclicDefect, CharactersTS_Cyclic_CL_SK}. The next aim is to understand blocks of tame representation type. The easiest subclass amongst these algebras is the class of blocks of domestic representation type, which happens if and only if the defect groups of the block are isomorphic to a Klein four-group. In the present article, the ordinary characters of all trivial source modules belonging to blocks with Klein-four defect groups are determined.\par
\indent This work is part of a program aiming at gathering information about trivial source modules of small finite groups and their associated trivial source character tables in a database, see~\cite{DatabaseTSCTs}. We refer the reader to \cite[Appendix]{book:oldbenson}, \cite[§4.10]{LuxPah}, \cite{BBCLNF}, \cite{FarLas23}, \cite{Thesis_Jorge_Rene_Ledesma_Granados}, and \cite{FrobeniusPaper_BB_CL} for previously computed trivial source modules and trivial source character tables. 
\indent Since trivial source modules are preserved by splendid Morita equivalences, it is useful to consider blocks of domestic representation type up to splendid Morita equivalence. If $B$ is a block of domestic representation type, then by the main result of \cite{KleinFourDefectGroupsCravenEatonLinckelmannK} (see \cref{KleinFourDefectGroupsCravenEatonLinckelmannK}) precisely one of the following cases must occur: $B$ is splendidly Morita equivalent to $kV_4$ or to~$k\AlternatingGroup_4$ or to the principal block of $k\AlternatingGroup_5$.\\
\indent The main result of the present article is the determination of the ordinary characters of the trivial source $B$-modules, in the three cases above, given the character values of the ordinary irreducible
characters of $B$:
\begin{itemize}
    \item[1)] the first case is dealt with in \cref{TS_characters_Blocks_Puig_equivalent_to_kV4};
    \item[2)] the second case is established in \cref{Proposition_ts_Puig_A4};
    \item[3)] the third case is settled in \cref{Proposition_ts_Puig_A5}.
\end{itemize}

We remark that the article~\cite{KleinFourDefectGroupsCravenEatonLinckelmannK} implies a bijection between the ordinary irreducible characters of $B$ and those of~$kV_4$ resp. $k\AlternatingGroup_4$ resp. the principal block of $k\AlternatingGroup_5$. However, this information does not yet determine trivial source characters of $B$ if we are only given the character values of the ordinary irreducible characters of $B$ (e.g. via the ordinary character table of a finite group). The present article accomplishes this aim.\\
\indent In order to obtain these results, it is necessary to treat the cases $kV_4$, $k\AlternatingGroup_4$, and $k\AlternatingGroup_5$ first.\\
\indent As a consequence of our three main theorems, in Section $5$, we can describe all trivial source character tables of the infinite family of dihedral groups of type $D_{4v}$, where $v$ is an odd integer, completely.\\
\indent To a large extent, this is achieved by evaluating ordinary characters of these groups, which aligns with a similar approach adopted in \cite{KoshitaniKunugiCharValues}. See also \cref{CharDegreesDeterminePuigClass}. 
Notice that calculating the trivial source character table of a finite group $G$ at a prime $p$ requires more than knowing the splendid Morita equivalence class for each block of $G$, as it provides insights into all possible Brauer constructions, necessitating the consideration of many subgroup characters.\par
\indent The article is structured as follows: in Section $2$, we fix the general notation and review existing
literature on representation theory and properties of trivial source modules and trivial source character tables. In Section $3$, we compute the trivial source modules and the trivial source character tables of 
$kV_4$, 
$k\AlternatingGroup_4$, and $k\AlternatingGroup_5$, respectively, in charateristic two.
In Section $4$, we prove that the character values of $\Irr_K(B)$ uniquely determine the ordinary characters of the trivial source $B$-modules for the aforementioned cases. 
Section $5$ presents two examples applying these results: the first is the infinite family of dihedral groups $D_{4v}$~($v\in\mathbb{Z}_{\geq 3}$ odd), and the second is a computational example of a concrete finite group $G$, where we derive the entire trivial source character table using theoretical arguments and character values of $\Irr_K(G)$, providing extensive details for clarity.

\section{Preliminaries}
\label{sec:notanddefs}

\vspace{2mm}
\subsection{General notation}
Throughout, unless otherwise stated, we adopt the notation and conventions below. We let~$p$ denote a positive prime number and  $G$ denote a finite group of order divisible by~$p$. We let~$(K,\mathcal{O},k)$ be a $p$-modular system, where $\mathcal{O}$ denotes a complete discrete valuation ring of characteristic zero with unique maximal ideal $\mathfrak{p}:=J(\mathcal{O})$, algebraically closed residue field $k=\mathcal{O}/\mathfrak{p}$ of characteristic $p$, and field of fractions $K=\textup{Frac}(\mathcal{O})$, which we assume to be large enough for $G$ and its subgroups in the sense that $K$ contains a root of unity of order~$\exp(G)$, the exponent of $G$. For~$R\in\{\cO,k\}$, $RG$-modules are assumed to be finitely generated left $RG$-lattices, that is, free as~$R$-modules, and  we let  $R$  denote the trivial $RG$-lattice. Given an $\mathcal{O}G$-lattice $M$, we denote the $p$-modular reduction of $M$ by $\overline{M}$.
    \par
    We let $\textup{Syl}_p(G)$ denote the set of all Sylow $p$-subgroups of $G$,  $ccls(G)$ denote a set of representatives for the conjugacy classes of $G$,  $[G]_{p^\prime}$ denote a set of representatives for the $p$-regular conjugacy classes of $G$, and we let $G_{p'}:=\{g\in G\mid p\nmid o(g)\}$. If two elements $a$ and $b$ are conjugate in $G$, we write $a\sim_G b$; otherwise, we write $a\not\sim_G b$. If the group is clear from the context, we omit the symbol $G$ in this notation. The centre of the group $G$ is denoted by $Z(G)$. Moreover, we let $\mathscr{S}_p(G)$ be a set of representatives of the set of conjugacy classes of $p$-subgroups of $G$.
    \par Unless otherwise stated, by a $p$-block we mean a block of the group algebra $kG$. We write~$B_0(G)$ for the principal $p$-block of $G$ and $\textup{Bl}_p(G)$ for the set of $p$-blocks of $G$.

We denote by $\Irr(G)$ (resp. $\Irr(B)$) the set of irreducible
$K$-characters of $G$ (resp. of the~$p$-block~$B$). Since the blocks of $\mathcal{O}G$ are in bijection with the blocks of $kG$ via reduction modulo~$J(\mathcal{O})$, by abuse of notation and terminology, given a $p$-block $B$ of $kG$, we write $\Irr(B)$ for the set of irreducible $K$-characters of the corresponding block of $\mathcal{O}G$ and talk about the ordinary irreducible characters of $B$. Moreover, we let $\Lin(G)$ (resp. $\Lin(B)$) and $\IBr_p(G)$ (resp. $\IBr_p(B)$) denote the set of linear characters of $G$ (resp. of the $p$-block $B$), and the set of all irreducible~$p$-Brauer characters of $G$ (resp. of the $p$-block $B$), respectively. Further, we write~$X(G)$ for the ordinary character table of the group $G$.
     \par
    Given a subgroup $H\leq G$ of $G$, an ordinary character $\psi$ of $H$ and $\chi$ an ordinary character of~$G$, we write 
    $\Ind_H^G(\psi)$ for the induction of $\psi$ from $H$ to $G$, 
    $\Res^G_H(\chi)$ for the restriction of $\chi$ from $G$ to $H$, $\chi^\circ:=\chi|_{G_{p'}}$ for the reduction modulo $p$ of $\chi$, $1_G$ for the trivial character of $G$, and $1_{G_{p^\prime}}$ for the trivial $p$-Brauer character of $G$. Given $N\unlhd G$ and an ordinary character~$\nu$ of~$G/N$, we write $\Inf_{G/N}^G(\nu)$ for the inflation of $\chi$ from $G/N$ to $G$. Similarly, we write~$\Ind_H^G(L)$ for the induction of the~$kH$-module $L$ from $H$ to $G$, $\Res^G_H(M)$ for the restriction of the $kG$-module $M$ from~$G$ to~$H$, and~$\Inf_{G/N}^G(U)$ for the inflation of the $k[G/N]$-module $U$ from $G/N$ to $G$. Moreover, if~$M$ is a~$kG$-module, then we denote by $\varphi_{M}$ the Brauer character afforded by $M$,  and if $Q\leq G$ then the Brauer quotient~(or Brauer construction) of $M$ at $Q$ is the $k$-vector space $M[Q]:=M^{Q}\big/ \sum_{R<Q}\tr_{R}^{Q}(M^{R})$, where~$M^{Q}$ denotes the fixed points of $M$ under~$Q$ and $\tr_{R}^{Q}$ denotes  the relative trace map. This vector space has a natural structure of a $kN_{G}(Q)$-module, but also of a $kN_{G}(Q)/Q$-module, and is equal to zero if $Q$ is not a $p$-subgroup. Moreover, we use the abbreviation PIM to mean a projective indecomposable module. We recall that for any $\chi\in\Irr(G)$ there exist uniquely determined non-negative integers  $d_{\chi\varphi}$ such that $\chi^\circ=\sum_{\varphi\in\IBr_p(G)}d_{\chi\varphi}\varphi$\,.
    Then, for any ${\varphi\in\IBr_p(G)}$,  the projective indecomposable character associated to $\varphi$ is 
     \begin{equation*}
\Phi_\varphi:=\sum_{\chi\in\Irr(G)}d_{\chi\varphi}\chi\,.
     \end{equation*}
    The $p$-decomposition matrix of $G$ is then
     \[
     \mathfrak{D}(kG)  :=\Big(d_{\chi,\varphi}\Big)_{\substack{\chi\in\Irr(G)\\\varphi\in\IBr_p(G)}}\in K^{|\Irr(G)|\times|\IBr_p(G)|}\,.
     \]
The $p$-decomposition matrix $\mathfrak{D}(B)$ of a $p$-block $B$ of $G$ is defined analogously. Given a $kG$-module $M$, we denote the dual module of $M$ by $M^{*}:=\Hom_k(M,k)$, the socle of $M$ by $\textup{Soc}(M)$, and the radical of $M$ by $\textup{Rad}(M)$. We write $[M]=\sum_{1\leq i\leq s}^{} {a_iM_i}$ if $M$ has precisely $\sum_{1\leq i\leq s}^{} {a_i}$ composition factors where, for each $1\leq i\leq s$, the simple composition factor $M_i$ occurs exactly~$a_i$ times. Given an algebra $A$, we denote the regular $A$-module by $A^\textup{reg}$.

     For any positive integer $n>1$, we denote the symmetric (respectively alternating) group acting on $n$ letters by $\SymmetricGroup_n$ (respectively $\AlternatingGroup_n$) and the dihedral group of order $2n$ by $D_{2n}$. Moreover, we denote the Klein four-group by $V_4$.
\par\indent We assume that the reader is familiar with elementary notions of ordinary and modular representation theory of finite groups. We refer to \cite{Linckelmann1, Linckelmann2, Webb, NagaoTsushima, CurtisReinerMethods1, ThevenazBook} for further standard notation and background results. However, we collect frequently used results from the literature in Subsection \ref{subsec:RepTypeOfAnAlgebra} and Subsection \ref{subsec:GeneralBackground}.

\subsection{The representation type of an algebra}\label{subsec:RepTypeOfAnAlgebra}

\indent Throughout this subsection, we denote by $k$ an algebraically closed field of characteristic $p>0$ for a (not necessarily even) prime number $p$. Recall that every finite dimensional $k$-algebra is of finite, tame or wild representation type, and these types are mutually exclusive (see, e.g., {\cite[Theorem 4.4.2]{Benson1}}). A subclass of all tame algebras  is given by the class of all domestic algebras defined as follows.

\begin{defn}
A $k$-algebra $A$ is called \textbf{domestic} if it is of finite growth.
\end{defn}
We refer the reader to \cite[Chapter XIX]{SimsonSkowronskiVol3} for more information on the notion of growth of an algebra, for equivalent definitions of a domestic $k$-algebra, and for some relations between algebras of tame representation type and algebras of domestic representation type. We use the term \emph{domestic} exclusively for representation-infinite domestic algebras. It follows from {\cite[The second Corollary]{KrauseStableEquivOfMoritaPresRepType}} that domestic representation type is preserved by Morita equivalences. If the algebra under consideration is a representation-infinite block $B$ of a group algebra, we can say even more (see, e.g., {\cite[Theorem 4.5]{FarnsteinerDomestic}}): a defect group of $B$ is isomorphic to $V_4$ if and only if $B$ has domestic representation type if and only if $B$ is Morita equivalent to~$kV_4$ or to $k\AlternatingGroup_4$ or to the principal block
of $k\AlternatingGroup_5$.

\subsection{Background from general representation theory}\label{subsec:GeneralBackground}

We collect here some well-known results from the modular representation theory of finite groups.
\begin{lem}[{\cite[page 40]{Alperin}}]\label{lem:Duals_and_Socles} 
Let $M$ be a $kG$-module. Then $\Soc(M^{*})\cong {(M/\Rad(M))}^{*}$ as~$kG$-modules.
\end{lem}

\begin{lem}[{\cite[III.9 Lemma 5]{Alperin}}]\label{LemmaAllModulesAreInducedFromSylow}
	Let $M$ be an indecomposable $kG$-module. Then, there exists a Sylow $p$-subgroup~$P$ of $G$ and an indecomposable $kP$-module $L$ such that $M$ is isomorphic to a direct summand of~$\Ind_{P}^{G}(L)$.
\end{lem}

\begin{lem}[{\cite[Theorem 3]{Robinson}}]\label{Lemma_Robinson}
	Let $G$ be a finite group. Let $\widetilde{H}\leq G$, and let $S$ and~$\tilde{S}$ respectively be a simple $kG$-module and a simple $k\widetilde{H}$-module. Then, the multiplicity of the projective cover $P(S)$ of $S$ as a direct summand of $\Ind_{\widetilde{H}}^{G}({\tilde{S}})$ is equal to the multiplicity of the projective cover $P(\tilde{S})$ of $\tilde{S}$ as a direct summand of $\Res_{\widetilde{H}}^{G}({S})$.
\end{lem}

\subsection{Trivial source modules and trivial source character tables}
\noindent In this subsection, we introduce trivial source modules and trivial source character tables. We refer the reader to {\cite{ThevenazBook}}, {\cite{Linckelmann1}}, {\cite{Linckelmann2}}, {\cite{LASpPermSuryey}}, {\cite{LuxPah}}, and {\cite{BoucThevenazPrimitiveIdempotents}} for more details.\par
\indent Given $R\in\{\mathcal{O},k\}$, an $RG$-lattice $M$ is called a \emph{trivial source} $RG$-lattice if it is isomorphic to an indecomposable  direct summand of an induced lattice $\Ind_{Q}^{G}(R)$, where $Q\leq G$ is a $p$-subgroup. In addition, if $Q$ is of minimal order subject to this property, then $Q$ is a vertex of~$M$. Recall that an $RG$-module $U$ is called a \emph{$p$-permutation $RG$-module} if $\Res_Q^G(U)$ is a permutation module for each $p$-subgroup $Q$ of $G$. Hence, the trivial source $kG$-modules are precisely the indecomposable~$p$-permutation $kG$-modules. In general, trivial source $kG$-modules afford several lifts to $\mathcal{O}G$-modules, but there is a unique one amongst these which is a trivial source $\mathcal{O} G$-module. We denote this trivial source lift by $\widehat{M}$ and we let~$\chi_{\widehat{M}}$ be the ordinary character afforded by~$K\otimes_{\mathcal{O}}\widehat{M}$.
Character values of trivial source modules have the following properties.

\begin{prop}\label{Landrock_and_Landrock-Scott}\label{Lifting_TS_Commutes_With_Ind}\noindent Let $M$ and $N$ be indecomposable trivial source $kG$-modules and let $x\in G$ be a $p$-element. Moreover, let $H\leq G$ be a subgroup of the group $G$. Then the following holds:
	\begin{enumerate}[label=\textup{(\alph*)}]
		\item the algebraic integer $\chi_{\widehat{M}}(x)$ is a non-negative integer equal to the multiplicity of the trivial $k\langle x\rangle$-module as a direct summand of $\Res_{\langle x \rangle}^G(M)$;
		\item $\chi_{\widehat{M}}(x)\neq 0$ if and only if $x$ belongs to a vertex of $M$;
        \item
        $\dim_k(\Hom_{kG}(M,N))= \langle \chi_{\widehat{M}}, \chi_{\widehat{N}}\rangle$,
	\noindent where we denote by $\langle -,-\rangle$ the usual scalar product of ordinary characters;
 \item if $L$ is an $\mathcal{O}H$-module then $\overline{\Ind_{H}^{G}(L)}\cong \Ind_H^G(\overline{L})$ as $kG$-modules;
 \item if $Z$ be a trivial source $kH$-module then $\chi_{\reallywidehat{\Ind_H^G(Z)}} = \Ind_H^G(\chi_{\widehat{Z}})$.
	\end{enumerate}
\end{prop}

\begin{proof}
    For Parts (a) and (b) see  {\cite[Lemma II. 12.6]{Landrock}}. For Part (c) combine {\cite[Theorem II.12.4 iii)]{Landrock}} and {\cite[Proposition I.14.8]{Landrock}}. For Part (d) see \cite[I.3 Proposition 4.11]{KarpilovskyVol4}. As the $kH$-module $Z$ lifts to an $\mathcal{O}H$-module $Y$ which is unique up to isomorphism, the $kH$-module~$\Ind_H^G(Z)$ lifts by (d) to the $\mathcal{O}G$-module~$\Ind_H^G(Y)$. The assertion of (e) follows.
\end{proof}

We will study trivial source modules vertex by vertex.  Hence, we denote by $\TS(G;Q)$ (resp. $\TS(B;Q)$) the set of isomorphism classes of indecomposable trivial source $kG$-modules (resp. trivial source $B$-modules) with vertex $Q$, when $B$ is a $p$-block of $G$. We notice that  $\TS(G;\{1\})$ is precisely  the set of isomorphism classes of PIMs of $kG$ and if $M$ is a PIM of $kG$, then $\chi_{\widehat{M}}=\Phi_{\varphi}$ where~$\varphi$ is the Brauer character afforded by the unique simple $kG$-module in the socle of $M$.

\begin{prop}\label{Omnibus_properties}
	\begin{enumerate}[label=\textup{(\alph*)}]
		\item The $p$-permutation modules are preserved under direct sums, tensor products, inflation, restriction and induction.
		\item If $Q\leq G$ is an $p$-subgroup of $G$ and $M$ is a $p$-permutation $kG$-module, then $M[Q]$ is a $p$-permutation $k\overline{N}_{G}(Q)$-module. 
		\item The vertices of a trivial source $kG$-module $M$ are the maximal $p$-subgroups $Q$ of $G$ such that $M[Q]\neq\{0\}$.
		\item A trivial source $kG$-module $M$ has vertex $Q$ if and only if $M[Q]$ is a non-zero projective $k\overline{N}_{G}(Q)$-module.
		Moreover, if this is the case, then the $kN_{G}(Q)$-Green correspondent $f(M)$ of $M$ is $M[Q]$ (viewed as a $kN_{G}(Q)$-module). Thus, there are \smallskip bijections:
		\begin{center}
			\begin{tabular}{ccccc}
				$\TS(G;Q)$      & $\longleftrightarrow$ &    $\TS(N_{G}(Q);Q)$   &  $\longleftrightarrow$ & $\TS(\overline{N}_G(Q);\TrivialGroup)$  \\
				$M$      & $\mapsto$     &   $f(M)$      &  $\mapsto$     & $M[Q].$
			\end{tabular}
		\end{center}
		These sets are also in bijection with the set of $p'\textup{-conjugacy classes of }\overline{N}_G(Q)$. 
		\item Let $H\leq G$. Then the Scott module $\Sc(G,H)$ is a trivial source $kG$-module lying in~$B_{0}(G)$.  If $Q\in\textup{Syl}_{p}(H)$, then $Q$ is a vertex of $\Sc(G,H)$ and $\Sc(G,H)\cong\Sc(G,Q)$. 
		\item The trivial source $kG$-modules, together with their vertices, are preserved under splendid Morita equivalences. 
		\item The trivial source $kG$-modules are preserved under taking duals.
		\item Let $H\leq G$. Then a $kH$-module $M$ is a trivial source module if and only if ${}^{g}(M)$ is a trivial source $k[{}^gH]$-module.
	\end{enumerate}
\end{prop}

\begin{proof}
\noindent Statements (a) to (e) are proved in~\cite{Bro}. Statement (f) follows from {\cite[Theorem 6.4.10 \& Theorem 9.7.4]{Linckelmann2}}. Part (g) follows from \cite[Lemma III.5]{Alperin}. Part (h) follows because of the following argument. For all $g\in G$ and for all subgroups $J$ of~$H$ and for all $kJ$-modules $L$ we have the $k[{}^gH]$-module isomorphism ${}^g(\Ind_{J}^{H}(L)) \cong \Ind_{{}^gJ}^{{}^gH}({}^gL)$. The claim follows now from specialising to the case in which $L$ is the trivial $kJ$-module.
\end{proof}

Next, we let $a(kG, \Triv)$ denote the trivial source ring of $kG$, which is defined to be the subring
of the Green ring of $kG$ generated by the set of all isomorphism classes of trivial source $kG$-modules. Notice that this ring is finitely generated. By definition, the trivial source character
table of the group $G$ at the prime $p$, denoted $\Triv_p(G)$, is the species table of the trivial source
ring of $kG$. See e.g. \cite{BenPar}. However, we follow \cite[Section 4.10]{LuxPah} and consider $\Triv_p(G)$ as
the block square matrix defined according to the following convention.

\begin{conv}\label{conv:tsctbl}%
First, fix  a set of representatives $Q_1,\ldots, Q_r$ ($r\in\IZ_{\geq 1}$) for the conjugacy classes of~$p$-subgroups of $G$ where $Q_{1}:=\{1\}$, $Q_{r}\in\Syl_{p}(G)$ and $|Q_1|\leq \ldots \leq |Q_r|$. For each $1\leq v\leq r$ set~$\overline{N}_{G}(Q_v):=N_{G}(Q_{v})/Q_{v}$.  
	For each pair $(Q_{v},s)$ with $1\leq v\leq r$ and $s\in [\overline{N}_{G}(Q_v)]_{p^\prime}$ there is a ring homomorphism  
	\begin{center}
		\begin{tabular}{cccl}
			$\tau_{Q_{v},s}^{G}$\,:            &   $a(kG,\mbox{Triv})$      & $\lra$ &    $K$     \\
			&   $[M]$      & $\mapsto$     &   $\varphi^{}_{M[Q_{v}]}(s)$   
		\end{tabular}
	\end{center}
	mapping the class of a trivial source $kG$-module $M$ to the value at~$s$ of the Brauer character~$\varphi^{}_{M[Q_{v}]}$ of the Brauer quotient $M[Q_{v}]$. 
	(Note  that the group $G$ acts by conjugation on the pairs $(Q_{v},s)$ and the values of $\tau_{Q_{v},s}^{G}$ do not depend on the choice of  $(Q_{v},s)$ in its $G$-orbit.)
	Then, for each $1\leq i,v\leq r$ define a matrix
	$$T_{i,v}:=\Big( \tau_{Q_{v},s}^{G}([M])\Big)_{\substack{M\in \TS(G;Q_{i})\\s\in [\overline{N}_{G}(Q_v)]_{p^\prime}}}\,.$$
	The \emph{trivial source character table of  $G$ at the prime $p$} is then the block matrix 
	$$\Triv_{p}(G):=\Big[T_{i,v}\Big]_{\substack{1\leq i\leq r\\1\leq v\leq r}}\,.$$
    For convenience, we will label the columns of $\Triv_p(G)$ by representatives of the $p^\prime$-elements of $\overline{N}_G(Q_v)$ in $N_G(Q_v)$ ($1\leq v\leq r$). This is possible e.g. by \cite[Lemma 3.1.12]{BBthesis}. Moreover, we label the rows of $\Triv_{p}(G)$  with the ordinary characters $\chi^{}_{\widehat{M}}$ instead of the isomorphism classes of  trivial source $kG$-modules $M$ themselves. Whenever we give a concrete example of a trivial source character table with entries in $\C$, we follow \cite{LuxPah} with the choice of the $p$-modular system. For brevity, we frequently write $N_G(Q_i)$ as $N_i$ and $\overline{N}_G(Q_i)$ as $\overline{N}_i$.
\end{conv}

\begin{rem}\label{rem:tsctbl}
	\begin{enumerate}[label=\textup{(\alph*)}]
		\item The block $T_{i,i}$ consists of the values of projective indecomposable characters of $\overline{N}_{i}$ at the $p'$-conjugacy classes of $\overline{N}_{i}$. In particular, $T^{}_{1,1}$ consists of the values of projective indecomposable characters of $G$ at the $p'$-conjugacy classes of $G$. 
		\item The trivial $kG$-module $k$ is a trivial source module with vertex $Q_{r}$ and  $k=M_{(Q_{r},k)}$.  For every $1\leq v\leq r$ and every $s\in[\overline{N}_{v}]_{p'}$ we have $\tau_{Q_{v},s}^{G}(k)=1$ since  $k[Q_{v}]=k$. 
		\item For $p$-subgroups $Q_{i},Q_{v}$ of $G$, it follows immediately from the definition and \cref{Omnibus_properties}(c) that $\tau^{G}_{Q_{v},s}([M_{(Q_{i},E)}])=0$ unless $Q_{v}\leq_{G} Q_{i}$. In other words $T_{i,v}=\mathbf{0}$ if $Q_{v}\not\leq_{G} Q_{i}$. 
		\item If $v=1$ and $1\leq i\leq r$, then $M[\TrivialGroup]=M$ and so 
		$$\tau_{\TrivialGroup,s}^{G}([M])=\chi^{}_{\widehat{M}}(s)$$
		for every $M\in \TS(G;Q_{i})$ and every $s\in [G]_{p'}$. In particular, for $M\in \TS(G;Q_{i})$ we have
		$$\tau_{{\TrivialGroup},1}^{G}([M])=\dim_{k}M\,.$$
		This explains the terminology \emph{trivial source character table} and our labelling of the rows by the characters. 
		\item More generally, if $(Q,s)\in \mathcal{Q}_{G,p}$ with $s=1$ and $M$ is any $p$-permutation $kG$-module, then  $\tau^{G}_{Q,1}([M])=\dim_{k} M[Q]$.
	\end{enumerate}
\end{rem}

\begin{lem}[{\cite[Lemma 6.2]{Ric} or \cite[10.13. LEMMA]{BenPar}}]\label{CorollaryRickard}\label{LemmaRickard}
Let $M$ be an indecomposable $kG$-module with a trivial source. Let $g_p$ and $g_{p^\prime}$ be commuting elements of $G$ with $g_p$ a $p$-element and $g_{p'}$ a $p'$-element. Write $g=g_p\cdot g_{p'}$. Then
        $$
        \chi_{\widehat{M}}(g) 
            = \tau_{\langle g_p\rangle, g_{p^{\prime}}}^G([M])\,.
        $$    
\end{lem}

Finally, we will need the following basic and well-known lemma about trivial source characters in blocks with defect groups of order two which easily follows from the theory of Brauer tree algebras (see \cite[Chapters 1-4]{HissLux} or \cite[Chapter V]{Alperin}).

\begin{lem}\label{BrauerTreeTrivialSourceModules}
	Let $G$ be a finite group and let $B\in\Bl(kG)$ be a block of the group algebra $kG$ with a cyclic defect group $D=\langle x \rangle$ of order two. Let $S$ be the unique simple $B$-module. Then, the following assertions hold:
	\begin{enumerate}[label=\textup{(\alph*)}]
        \item up to isomorphism, there are exactly two non-isomorphic indecomposable $kG$-modules that belong to $B$, namely the simple module $S$ and its projective cover $P(S)$;
		\item the $kG$-module $S$ has a trivial source and vertex $D$;
  \item the character $\chi_{\widehat{S}}$ is the unique element of $\Irr_K(B)$ which takes a positive value at $x$.
	\end{enumerate} 
\end{lem}

\section{The trivial source character tables of the Klein four-group \texorpdfstring{$V_4$}{V4}, the alternating group \texorpdfstring{$\AlternatingGroup_4$}{A4}, and the alternating group \texorpdfstring{$\AlternatingGroup_5$}{A5}}
\noindent From now on, we assume that $p=2$ and that $k$ is an algebraically closed field of characteristic~$2$. \indent The three cases of $V_4, \AlternatingGroup_4$, and $\AlternatingGroup_5$ are well-known and easy to prove. However, we give the results in order to set up the notation for the splendid Morita equivalences which are considered in the sequel.

\subsection{The Klein four-group \texorpdfstring{$V_4$}{V4}}\label{Notation_V4}
\noindent In this subsection we assume that $G=V_4=\langle a , b \rangle\leq \SymmetricGroup_4$, where $a:=(1,2)(3,4)$ and $b:=(1,3)(2,4)$. The ordinary character table of $V_4$ is as given in \cref{table:ct_C_2_times_C_2}.

\begin{table}[H]
	\centering
	\begin{tabular}{@{}l@{}l@{}l@{}}
		\hline
		\(\begin{array}{|l|cccc|}
			& 1 & a & b & ab\\ \hline
			\chi_{1} & 1 & 1 & 1 & 1\\
			\chi_{2} & 1 & 1 & -1 & -1\\
			\chi_{3} & 1 & -1 & 1 & -1\\
			\chi_{4} & 1 & -1 & -1 & 1\\
			\hline
		\end{array}\)\\
	\end{tabular}
	\caption{Ordinary character table of $V_4$}
	\label{table:ct_C_2_times_C_2}
\end{table}

\noindent We set
$$Q_1:=\TrivialGroup,\quad Q_2:=\langle a\rangle,\quad  Q_3:=\langle b\rangle,\quad  Q_4:=\langle ab\rangle,\quad  \textup{and}\quad  Q_5:=V_4.$$
Furthermore, $\mathscr{S}_2(V_4) = \{Q_1, Q_2, Q_3, Q_4, Q_5 \}$. The lattice of subgroups in~$\mathscr{S}_2(V_4)$ is given as follows:
$$
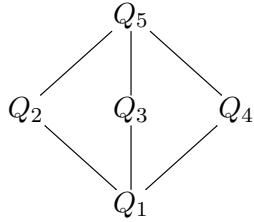
\resizebox{0.2\textwidth}{!}{
\xymatrix{ & Q_5 & \\
	Q_2 \ar@{-}[ur] & Q_3 \ar@{-}[u] & Q_4 \ar@{-}[ul] \\
	& Q_1 \ar@{-}[ul] \ar@{-}[u] \ar@{-}[ur]	& }
    }
    $$

\begin{lem}\label{lem:Normalisers_and_Quotients_V4}
The following assertions hold:
\begin{table}[H]
	\centering
	{\resizebox{\textwidth}{!}{
	\begin{tabular}{@{}l@{}l@{}l@{}l@{}l@{}l@{}l@{}l@{}l@{}l@{}l@{}l@{}l@{}l@{}}
		\(\begin{array}{ccccc}
   N_{V_4}(Q_1) = {V_4};\ & N_{V_4}(Q_2) = {V_4};\ & N_{V_4}(Q_3) = {V_4};\ & N_{V_4}(Q_4) = {V_4};\ & N_{V_4}(Q_5) = {V_4};\\ [0.5ex]
   \overline{N}_{V_4}(Q_1) \cong {V_4}; & \overline{N}_{V_4}(Q_2) \cong Q_3 \cong C_2; & \overline{N}_{V_4}(Q_3) \cong Q_2 \cong C_2; & \overline{N}_{V_4}(Q_4) \cong C_2; & \overline{N}_{V_4}(Q_5) \cong Q_1.\\
		\end{array}\)
	\end{tabular}
}}
\end{table}
\end{lem}

\begin{proof}
    Direct computation.
\end{proof}

\begin{lem}\label{PropertiesTS_V4_modules}
	Setting $M_i:=\Ind_{Q_i}^{V_4}(k)$ for $1\leq i\leq 5$, the following assertions about the trivial source $kV_4$-modules hold:
$\TS(V_{4};Q_i) = \{ \Ind_{Q_i}^{V_4}(k) \}$ for every $1\leq i\leq 5$.
\end{lem}

\begin{proof}
By \cref{lem:Normalisers_and_Quotients_V4}, we have $|[\overline{N}_{V_{4}}(Q_i)]_{2^\prime}|=|\{1\}|=1$ for each $1\leq i \leq 5$. Hence, it follows from  \cref{Omnibus_properties} that
$|\TS(V_{4};Q_i)|=1$ for every $1\leq i\leq 5$. By Green's indecomposability criterion, the $kV_4$-modules $M_i$ are indecomposable trivial source modules for $1\leq i\leq 5$. They are pairwise non-isomorphic as their vertices are not conjugate in $V_4$.
\end{proof}

\begin{prop}\label{Triv_2_V4}
Labelling the ordinary characters of $V_4$ as in \cref{table:ct_C_2_times_C_2}, the trivial source character table $\Triv_2(V_4)$ is as given in \cref{table:TS_ct_C_2_times_C_2_p_is_2}.

\begin{table}[H]
	\centering
	{\scalebox{0.9}{
	\begin{tabular}{@{}l@{}l@{}l@{}l@{}l@{}l@{}l@{}l@{}l@{}l@{}l@{}l@{}l@{}l@{}}
		\(\begin{array}{|lV{4}c|c|c|c|c|}
			\hline
			Q_v\ (1\leq v\leq 5) & \multicolumn{1}{c|}{Q_1\cong C_1} & \multicolumn{1}{c|}{Q_2\cong C_2} & \multicolumn{1}{c|}{Q_3\cong C_2} & \multicolumn{1}{c|}{Q_4\cong C_2} & \multicolumn{1}{c|}{Q_5\cong V_4}\\ \hline
			N_v\ (1\leq v\leq 5) & \multicolumn{1}{c|}{N_1\cong V_4} & \multicolumn{1}{c|}{N_2\cong V_4} & \multicolumn{1}{c|}{N_3\cong V_4} & \multicolumn{1}{c|}{N_4\cong V_4} & \multicolumn{1}{c|}{N_5\cong V_4}\\ \hline
			n_j\ \in\ N_v & 1 & 1 & 1 & 1 & 1\\ \Xhline{4\arrayrulewidth}
			\chi_{1} + \chi_{2}+ \chi_{3}+ \chi_{4} & 4 & 0 & 0 & 0 & 0\\
			\hline
			\chi_{1}+ \chi_{2} & 2 & 2 & 0 & 0 & 0\\
			\hline
			\chi_{1}+ \chi_{3} & 2 & 0 & 2 & 0 & 0\\
			\hline
			\chi_{1}+ \chi_{4} & 2 & 0 & 0 & 2 & 0\\
			\hline
			\chi_{1} & 1 & 1 & 1 & 1 & 1\\
			\hline
		\end{array}\)
	\end{tabular}
}}
	\caption{Trivial source character table of $V_4$ at $p=2$}
	\label{table:TS_ct_C_2_times_C_2_p_is_2}
\end{table}
\end{prop}

\begin{proof}
As~$a\not\sim b\not\sim ab\not\sim a$ in $V_4$, it follows from \cref{rem:tsctbl}(c) that $T_{i,v}=\mathbf{0}$ for $1\leq i\neq v\leq 5$. Labelling the ordinary characters as in \cref{table:ct_C_2_times_C_2} and keeping the notation from \cref{PropertiesTS_V4_modules}, it follows from 
\cref{Lifting_TS_Commutes_With_Ind} that
$\chi_{\widehat{M_i}} = \Ind_{Q_i}^{G}(1_{Q_i})$ for $1\leq i\leq 5$. Hence, all remaining entries of $\Triv_2(V_4)$ except for the entry of $T_{5,5}$ are now immediate from
\cref{LemmaRickard}. The entry of~$T_{5,5}$ follows from \cref{rem:tsctbl}(b).
\end{proof}

\indent We mention here that $\Triv_2({V_4})$ is isomorphic to the table of marks of ${V_4}$, since $V_4$ is a $2$-group. See, e.g., {\cite[Proposition 3.2.17]{BBthesis}}. Nonetheless, we use the theory of characters here because it turns out to be more useful for further computations in the sequel.

\subsection{The alternating group
\texorpdfstring{$\AlternatingGroup_4$}{A4}}\label{Notation_A4}
\noindent In this subsection we assume that $G=\AlternatingGroup_4 =\langle a, c \rangle\leq \SymmetricGroup_4$, where $a:=(1,2)(3,4)$ and $c:=(1,2,3)$. Moreover, we let~$b:=~(1,3)(2,4)\in~{\AlternatingGroup_4}$ and define $\omega:=\exp(\frac{2\pi i}{3})
$. Then the ordinary character table of $\AlternatingGroup_{4}$ is as given in \cref{table:Ordinary_ct_A_4}.

\begin{table}[ht]
	\centering
	\begin{tabular}{| c | c c c c |}
		\hline
		& $1$ & $a$ & $c$ & $bc^2$\\
		\hline
		$\chi_1$ & $1$ & $1$ & $1$ & $1$\\
		$\chi_2$ & $1$ & $1$ & $\omega$ & $\omega^2$\\
		$\chi_3$ & $1$ & $1$ & $\omega^2$ & $\omega$\\
		$\chi_4$ & $3$ & $-1$ & $0$ & $0$\\
		\hline
	\end{tabular}
	\caption{Ordinary character table of $\AlternatingGroup_4$}
	\label{table:Ordinary_ct_A_4}
\end{table}

\noindent We set
$$Q_1:=\TrivialGroup,\quad Q_2:=\langle a\rangle,\quad  \textup{and}\quad  Q_3:=\langle a,b\rangle\in \textup{Syl}_2({\AlternatingGroup_4}).$$
\noindent Furthermore, we fix $\mathscr{S}_2(\AlternatingGroup_4) = \{Q_1, Q_2, Q_3\}$. Then, the chain of subgroups $Q_1\leq Q_2\leq Q_3$ is the lattice of subgroups in $\mathscr{S}_2(\AlternatingGroup_4)$.

\begin{lem}
The following assertions hold:
\begin{table}[H]
	\centering
	{\scalebox{1}{
	\begin{tabular}{@{}l@{}l@{}l@{}l@{}l@{}l@{}l@{}l@{}l@{}l@{}l@{}l@{}l@{}l@{}}
		\(\begin{array}{ccc}
   N_{\AlternatingGroup_4}(Q_1) = {\AlternatingGroup_4};\ & N_{\AlternatingGroup_4}(Q_2) = {Q_3};\ & N_{\AlternatingGroup_4}(Q_3) = {\AlternatingGroup_4};\\ [0,5ex]
   \overline{N}_{\AlternatingGroup_4}(Q_1) \cong {\AlternatingGroup_4}; & \overline{N}_{\AlternatingGroup_4}(Q_2) \cong \langle b \rangle \cong C_2; & \overline{N}_{\AlternatingGroup_4}(Q_3) \cong \langle c \rangle \cong C_3.\\
		\end{array}\)
	\end{tabular}
}}
\end{table}
\end{lem}

\begin{proof}
    Direct computation.
\end{proof}

\begin{nota}
For the $p^\prime$-conjugacy classes of the normaliser quotients we make the following choices:
    $$[\overline{N}_{\AlternatingGroup_{4}}(Q_1)]_{p^\prime}=\{1Q_1, cQ_1, bc^2Q_1\},\ [\overline{N}_{\AlternatingGroup_{4}}(Q_2)]_{p^\prime}=\{1Q_2\},\ \textup{and}\ [\overline{N}_{\AlternatingGroup_{4}}(Q_3)]_{p^\prime}=\{1Q_3,cQ_3,bc^2Q_3\}.$$
    Moreover, we denote the three one-dimensional simple~$k\AlternatingGroup_{4}$-modules by $k, S_{\omega}$, and $S_{\overline{\omega}}$ with $\varphi_{S_\omega} = \chi_2^\circ$ and $\varphi_{S_{\overline{\omega}}} = \chi_3^\circ$. 
\end{nota}

\indent The classification of the indecomposable modules for the alternating group~$\AlternatingGroup_{4}$ in characteristic~$2$ is well-known. See, e.g., the remark after Theorem 4.3.3 in \cite{Benson1}. Every such module is (e.g. by \cref{LemmaAllModulesAreInducedFromSylow}) a summand of a module induced from the group $Q_3\cong V_{4}$ up to the group~${\AlternatingGroup_4}$ and, in particular, the $k\AlternatingGroup_{4}$-module $M:=\Ind_{Q_2}^{\AlternatingGroup_{4}}(k)$ is indecomposable.

\begin{prop}\label{PropertiesTS_A4_modules}
	The following assertions about the trivial source $k\AlternatingGroup_4$-modules hold:
	\begin{enumerate}[label=\textup{(\alph*)}]
		\item $\TS(\AlternatingGroup_{4};Q_1) = \{ P(k), P(S_{\omega}), P(S_{\overline{\omega}})\}$;
		\item $\TS(\AlternatingGroup_{4};Q_2) = \{M\}$;
		\item $\TS(\AlternatingGroup_{4};Q_3)=\{k, S_{\omega}, S_{\overline{\omega}}\}$.
	\end{enumerate}
\end{prop}

\begin{proof}
\noindent Since $|[\overline{N}_{\AlternatingGroup_{4}}(Q_1)]_{p^\prime}|=3$, $|[\overline{N}_{\AlternatingGroup_{4}}(Q_2)]_{p^\prime}|=1$, and $|[\overline{N}_{\AlternatingGroup_{4}}(Q_3)]_{p^\prime}|=3$, it follows from  \cref{Omnibus_properties} that
$$|\TS(\AlternatingGroup_{4};Q_1)|=3=|\TS(\AlternatingGroup_{4};Q_3)|\ \textup{and}\ |\TS(\AlternatingGroup_{4};Q_2)|=1.$$
\noindent The~$k{\AlternatingGroup_{4}}$-modules $P(k), P(S_{\omega})$, and~$P(S_{\overline{\omega}})$ have trivial vertices and are trivial source modules, since they are projective and indecomposable. This proves (a). Every trivial source $k{\AlternatingGroup_{4}}$-module which has $Q_3$ as a vertex is isomorphic to a direct summand of $\Ind_{Q_3}^{\AlternatingGroup_{4}}(k)$. As $|\TS({\AlternatingGroup_{4}};Q_3)|=3$ and~$[\AlternatingGroup_{4}:Q_3]=3$, it follows that the $k\AlternatingGroup_4$-module $\Ind_{Q_3}^{\AlternatingGroup_{4}}(k)$ decomposes into a direct sum of three non-isomorphic~$1$-dimensional~$k{\AlternatingGroup_{4}}$-modules. Therefore,~$\Ind_{Q_3}^{\AlternatingGroup_{4}}(k)\cong k\oplus S_{\omega}\oplus S_{\overline{\omega}}$ and all simple $k{\AlternatingGroup_{4}}$-modules are trivial source modules with~$Q_3$ as a vertex. This proves (c). The~$k{\AlternatingGroup_{4}}$-module~$M$ is indecomposable and by \cref{Omnibus_properties} it is a trivial source module. The claim in~(b) follows, as~$|\TS(\AlternatingGroup_{4};Q_2)|=1$.
\end{proof}

\begin{prop}\label{Triv_2_A4}
Labelling the ordinary characters of $\AlternatingGroup_4$ as in \cref{table:Ordinary_ct_A_4}, the trivial source character table $\Triv_2(\AlternatingGroup_4)$ is as given in \cref{table:TS_ct_A4_p_is_2}.
\begin{table}[ht]
	\centering
	{\scalebox{0.8}{
			\begin{tabular}{@{}l@{}l@{}l@{}l@{}l@{}l@{}l@{}l@{}l@{}l@{}}
				\(\begin{array}{|lV{4}ccc|c|ccc|}
					\hline
					Q_v\ (1\leq v\leq 3) & \multicolumn{3}{c|}{Q_{1}\cong C_1} & \multicolumn{1}{c|}{Q_{2}\cong C_2} & \multicolumn{3}{c|}{Q_{3}\cong V_4}\\ \hline
					N_v\ (1\leq v\leq 3) & \multicolumn{3}{c|}{N_{1}\cong \AlternatingGroup_4} & \multicolumn{1}{c|}{N_{2}\cong V_4} & \multicolumn{3}{c|}{N_{3}\cong \AlternatingGroup_4}\\ \hline
					n_j\ \in\ N_v & 1 & c & bc^2 & 1 & 1 & c & bc^2\\
					\Xhline{4\arrayrulewidth}
					\chi_{1} + \chi_{4} & 4 & 1 & 1 & 0 & 0 & 0 & 0\\
					\chi_{2} + \chi_{4} & 4 & \omega & \omega^{2} & 0 & 0 & 0 & 0\\
					\chi_{3} + \chi_{4} & 4 & \omega^{2} & \omega & 0 & 0 & 0 & 0\\
					\hline
					\chi_{1}+ \chi_{2}+ \chi_{3}+ \chi_{4} & 6 & 0 & 0 & 2 & 0 & 0 & 0\\
					\hline
					\chi_{1} & 1 & 1 & 1 & 1 & 1 & 1 & 1\\
					\chi_{2} & 1 & \omega & \omega^{2} & 1 & 1 & \omega & \omega^{2}\\
					\chi_{3} & 1 & \omega^{2} & \omega & 1 & 1 & \omega^{2} & \omega\\
					\hline
				\end{array}\)
			\end{tabular}
	}}
	\caption{Trivial source character table of $\AlternatingGroup_4$ at $p=2$}
	\label{table:TS_ct_A4_p_is_2}
\end{table}	
\end{prop}

\begin{proof}
See \cite[Theorem 5.6]{BBCLNF} (via the isomorphism $\operatorname{PSL}_2(3)\cong \AlternatingGroup_4$) or \cite[Proposition 4.1.5]{BBthesis}.
\end{proof}

\subsection{The alternating group \texorpdfstring{$\AlternatingGroup_5$}{A5}}
\noindent In this subsection we assume that $G=\AlternatingGroup_5=\langle a,d\rangle\leq \SymmetricGroup_5$, where $a:=(1,2)(3,4)$ and $d:=(1,3,5)$. Moreover, we set~$b:=(1,3)(2,4)\in {\AlternatingGroup_5}$ and $c:=(1,2,3)\in {\AlternatingGroup_5}$. Defining $A:=\frac{1-\sqrt{5}}{2}\ \textup{and}\ ^{*}A:=\frac{1+\sqrt{5}}{2}$, the ordinary character table of ${\AlternatingGroup_{5}}$ is as given in \cref{table:Ordinary_ct_A_5}.\label{Notation_A5}
\begin{table}[H]
	\setlength\extrarowheight{1.5pt}
	\centering
		\begin{tabular}{|r|ccccc|}
			\hline
			& $\textup{1}$ & $a$ & $d$ & $ad$ & ${(ad)}^2$\\
			\hline
			$\chi_1$ &	$1$ & $1$ & $1$ & $1$ & $1$\\
			$\chi_2$ &	$3$ & $-1$ & $0$ & $A$ & $^{*}A$\\
			$\chi_3$ &	$3$ & $-1$ & $0$ & $^{*}A$ & $A$\\
			$\chi_4$ &	$4$ & $0$ & $1$ & $-1$ & $-1$\\
			$\chi_5$ &	$5$ & $1$ & $-1$ & $0$ & $0$\\
			\hline
		\end{tabular}
		\caption{Ordinary character table of $\AlternatingGroup_5$}
		\label{table:Ordinary_ct_A_5}

\end{table}
\noindent We set
$$Q_1:=\TrivialGroup,\ Q_2:=\langle a\rangle,\ \textup{and}\ Q_3:=\langle a,b\rangle\in \textup{Syl}_2({\AlternatingGroup_{5}}).$$
\noindent Furthermore, we set $\omega:=\exp(\frac{2\pi i}{3})
$, $\eta:=\exp(\frac{2\pi i}{5})
$, and choose~$\mathscr{S}_2(\AlternatingGroup_5) = \{Q_1, Q_2, Q_3\}$. Then, the chain of subgroups $Q_1\leq Q_2\leq Q_3$ is the lattice of subgroups in $\mathscr{S}_2(\AlternatingGroup_5)$.

\begin{lem}
The following assertions hold:
\begin{table}[H]
	\centering
	{\scalebox{1}{
	\begin{tabular}{@{}l@{}l@{}l@{}l@{}l@{}l@{}l@{}l@{}l@{}l@{}l@{}l@{}l@{}l@{}}
		\(\begin{array}{ccc}
   N_{\AlternatingGroup_5}(Q_1) = {\AlternatingGroup_5};\ & N_{\AlternatingGroup_5}(Q_2) = {Q_3};\ & N_{\AlternatingGroup_5}(Q_3) = {\langle a, c \rangle \cong \AlternatingGroup_4};\\ [0,5ex]
   \overline{N}_{\AlternatingGroup_5}(Q_1) \cong {\AlternatingGroup_5}; & \overline{N}_{\AlternatingGroup_5}(Q_2) \cong \langle b \rangle \cong C_2; & \overline{N}_{\AlternatingGroup_5}(Q_3) \cong \langle c \rangle \cong C_3.\\
		\end{array}\)
	\end{tabular}
}}
\end{table}

\end{lem}

\begin{proof}
    Direct computation.
\end{proof}

\begin{nota}\label{nota:A5}
For the $p^\prime$-conjugacy classes of the normaliser quotients we make the following choices:
$$[\overline{N}_{\AlternatingGroup_{5}}(Q_1)]_{p^\prime} = \{1, d, ad, {ad}^2\},\ [\overline{N}_{\AlternatingGroup_{5}}(Q_2)]_{p^\prime} = \{1Q_2\},\ \textup{and}\ [\overline{N}_{\AlternatingGroup_{5}}(Q_3)]_{p^\prime} = \{1Q_3,cQ_3,c^2Q_3\}.$$
As $|[\AlternatingGroup_{5}]_{p^\prime}|=4$, it follows from  \cref{Omnibus_properties} that, up to isomorphism, there exist exactly four simple~$k{\AlternatingGroup_{5}}$-modules which we denote by $S_1 := k$, $S_2$, $S_3$, and~$S_4$. Here,
$$\chi_1^\circ = \varphi_k,\ \chi_2^\circ = \varphi_k + \varphi_{S_3},\ \chi_3^\circ = \varphi_k + \varphi_{S_2},\ \textup{and}\ \chi_4^\circ = \varphi_k + \varphi_{S_2} + \varphi_{S_3}.$$ Moreover, we set
$$H:=\langle (1,2)(3,4), (1,4)(2,5) \rangle \leq {\AlternatingGroup_{5}}\ \textup{and}\ M_5:={\Ind}_{H}^{{\AlternatingGroup_{5}}}(k).$$
Note that $H\cong D_{10}$, the dihedral group of order $10$. 
\end{nota}

\begin{rem}
\noindent We identify $N_{\AlternatingGroup_{5}}(Q_3)$ with the group $\AlternatingGroup_4$ from the previous subsection. Note that these two permutation groups are verbatim the same subgroup of~$\SymmetricGroup_4$. 
\end{rem}

\begin{prop}\label{PropertiesTS_A5_modules}
	\begin{enumerate}[label=\textup{(\alph*)}]
		\item We have $\TS({\AlternatingGroup_{5}};Q_1) = \{ P(k), P(S_2), P(S_3), P(S_4)\}$.
		\item We have $\TS({\AlternatingGroup_{5}};Q_2) = \{M_5\}$ and $\chi_{\widehat{M_5}} = \chi_1 + \chi_5$.
		\item Let the two $5$-dimensional indecomposable summands of ${\Ind}_{Q_{3}}^{\AlternatingGroup_5}(k)$ be denoted by $M_7$ and~$M_8$, respectively. Then we have $\TS({\AlternatingGroup_{5}};Q_3)=\{S_1, M_7, M_8\}$, $\chi_{\widehat{M_7}} = \chi_{\widehat{M_8}} = \chi_5$, and $\chi_{\widehat{S_1}} = 1_{\AlternatingGroup_5}$.
	\end{enumerate}
\end{prop}

\begin{proof}
Since
\begin{align*}
|[\overline{N}_{\AlternatingGroup_{5}}(Q_1)]_{2^\prime}|&=|\{1, d, ad, {ad}^2\}|=4,\ |[\overline{N}_{\AlternatingGroup_{5}}(Q_2)]_{2^\prime}|=|\{1Q_2\}|=1,\ \textup{and}\\
|[\overline{N}_{\AlternatingGroup_{5}}(Q_3)]_{2^\prime}|&=|\{1Q_3,cQ_3,c^2Q_3\}|=3,
\end{align*}
it follows from  \cref{Omnibus_properties} that
$$|\TS(\AlternatingGroup_{5};Q_1)|=4,\ |\TS(\AlternatingGroup_{5};Q_2)|=1,\ \textup{and}\ |\TS(\AlternatingGroup_{5};Q_3)|=3.$$
The~$k{\AlternatingGroup_{5}}$-modules $P(S_1)$, $P(S_2)$, $P(S_3)$ and $P(S_4)$ are trivial source modules by definition. This proves (a). We have $Q_2\in \textup{Syl}_2(H)$. As there is only one $p'$-conjugacy class in $\overline{N}_{H}(Q_2)\cong\TrivialGroup$, we deduce that $\TS(H;Q_2)=\{k_H\}$. All other trvial source $kH$-modules are projective and their dimensions are therefore divisible by $2$. Since the $kH$-module ${\Ind}_{Q_2}^{H}(k)$ is $5$-dimensional, we obtain:
$${\Ind}_{Q_2}^{H}(k) \cong k\oplus X$$ 
\noindent for some projective $kH$-module $X$. Therefore,
$${\Ind}_{Q_2}^{{\AlternatingGroup_{5}}}(k) = {\Ind}_{H}^{{\AlternatingGroup_{5}}} {\Ind}_{Q_2}^{H}(k)
={\Ind}_{H}^{{\AlternatingGroup_{5}}} (k\oplus X)
={\Ind}_{H}^{{\AlternatingGroup_{5}}}(k)\oplus {\Ind}_{H}^{{\AlternatingGroup_{5}}}(X).$$
\noindent Note that the $k{\AlternatingGroup_{5}}$-module $M_5={\Ind}_{H}^{\AlternatingGroup_{5}}(k)$ is $6$-dimensional. By transitivity of induction, we have 
$${\Ind}_{Q_2}^{{\AlternatingGroup_{5}}}(k) = 
\Ind_{N_{\AlternatingGroup_{5}}(Q_3)}^{{\AlternatingGroup_{5}}} {{\Ind}_{Q_2}^{N_{\AlternatingGroup_{5}}(Q_3)}}(k).$$
We know from \cref{PropertiesTS_A4_modules} that the $kN_{\AlternatingGroup_{5}}(Q_3)$-module  $\Ind_{Q_2}^{N_{\AlternatingGroup_{5}}(Q_3)}(k)$ is indecomposable. Since~$\dim_k({\Ind}_{Q_2}^{N_{\AlternatingGroup_{5}}(Q_3)}(k))=6$ and~$N_{\AlternatingGroup_{5}}(Q_2)\leq N_{{\AlternatingGroup_{5}}}(Q_3)$, we deduce from the Green correspondence~(see, e.g., {\cite[(20.6) Theorem]{CurtisReinerMethods1}}) that $\dim_k(g({\Ind}_{Q_2}^{N_{\AlternatingGroup_{5}}(Q_3)}(k)))\geq 6$, as restriction does not change the dimension. Hence, the $k\AlternatingGroup_5$-module $M_5 = {\Ind}_{H}^{{\AlternatingGroup_{5}}}(k)$ is indecomposable. It follows from \cref{Lifting_TS_Commutes_With_Ind} that $\chi_{\widehat{M_5}} = \Ind_{H}^{\AlternatingGroup_5}(1_H) = \chi_1 + \chi_5$. This proves (b). Next, consider the~$k{\AlternatingGroup_{5}}$-module $L:={\Ind}_{Q_3}^{{\AlternatingGroup_{5}}}(k)$. By \cref{Lifting_TS_Commutes_With_Ind}, the ordinary character~$\chi_{\widehat{L}}$ of~$L$ is equal to
$$\chi_{\widehat{L}}={\Ind}_{Q_3}^{\AlternatingGroup_5}(1_{Q_3}) = \chi_1 + \chi_4 + 2\cdot \chi_5.$$
\noindent Since $\chi_4$ does not belong to $\Irr_K(B_0(k{\AlternatingGroup_{5}}))$, the set $\TS({\AlternatingGroup_{5}};Q_3)$, as well as the trivial source characters, are as asserted.
\end{proof}	


Finally, we compute the trivial source character table of the group $\AlternatingGroup_5$ at the prime $p=2$. We remark that $\Triv_2(\AlternatingGroup_5)$ already appears (without proof) in \cite[Appendix, page 199]
{book:oldbenson}.

\begin{prop}\label{Triv_2_A5}
	Labelling the ordinary characters of $\AlternatingGroup_5$ as in \cref{table:Ordinary_ct_A_5}, the trivial source character table $\Triv_2(\AlternatingGroup_5)$ is as given in \cref{table:TS_ct_A5_p_is_2}.
\end{prop}

	\begin{table}[ht]
		\setlength\extrarowheight{1.5pt}
		\centering
		{\scalebox{0.75}{
				\begin{tabular}{@{}l@{}l@{}l@{}l@{}l@{}l@{}l@{}l@{}l@{}l@{}}
					\(\begin{array}{|lV{4}cccc|c|ccc|}
						\hline
						Q_v\ (1\leq v\leq 3) & \multicolumn{4}{c|}{Q_{1}\cong C_1} & \multicolumn{1}{c|}{Q_{2}\cong C_2} & \multicolumn{3}{c|}{Q_{3}\cong V_4}\\ \hline
						N_v\ (1\leq v\leq 3) & \multicolumn{4}{c|}{N_{1}\cong \AlternatingGroup_5} & \multicolumn{1}{c|}{N_{2}\cong V_4} & \multicolumn{3}{c|}{N_{3}\cong \AlternatingGroup_4}\\ \hline
						n_j\ \in\ N_v & 1 & d & ad & {(ad)}^2 & 1 & 1 & c & c^2\\ \Xhline{4\arrayrulewidth}
						\chi_{1}+\chi_{2}+\chi_{3}+ \chi_{5} & 12 & 0 & 2 & 2 & 0 & 0 & 0 & 0\\
						\chi_{3}+ \chi_{5} & 8 & -1 & -\eta^{2}-\eta^{3} & -\eta-\eta^{4} & 0 & 0 & 0 & 0\\
						\chi_{2} + \chi_{5} & 8 & -1 & -\eta-\eta^{4} & -\eta^{2}-\eta^{3} & 0 & 0 & 0 & 0\\
						\chi_{4} & 4 & 1 & -1 & -1 & 0 & 0 & 0 & 0\\
						\hline
						\chi_{1} + \chi_{5} & 6 & 0 & 1 & 1 & 2 & 0 & 0 & 0\\
						\hline
						\chi_{1} & 1 & 1 & 1 & 1 & 1 & 1 & 1 & 1\\
						\chi_{5} & 5 & -1 & 0 & 0 & 1 & 1 & \omega & \omega^{2}\\
						\chi_{5} & 5 & -1 & 0 & 0 & 1 & 1 & \omega^{2} & \omega\\
						\hline	
					\end{array}\)\\
				\end{tabular}	
		}}
		\caption{Trivial source character table of $\AlternatingGroup_5$ at $p=2$}
		\label{table:TS_ct_A5_p_is_2}
	\end{table}

\begin{proof}
Most entries of $\Triv_2(\AlternatingGroup_5)$ follow from \cref{PropertiesTS_A5_modules} in combination with \cref{LemmaRickard} and \cref{rem:tsctbl}(c). For the remaining entries see \cite[Theorem 5.6]{BBCLNF} (via the isomorphism $\operatorname{PSL}_2(5)\cong \AlternatingGroup_5$) or \cite[Proposition 4.1.8]{BBthesis}.
\end{proof}

\section{Trivial source characters in blocks with Klein four defect groups}\label{sec:TS_characters_in_blocks_with_Klein_four_defect_groups}
\begin{nota}
\noindent Throughout this section, we agree on the following assumptions:
\begin{center}
\fbox{\begin{minipage}{13cm}{\begin{enumerate}
				\item $H$ is a finite group;
				\item $\textup{char}(k) = p = 2$;
				\item $B^\prime \in \Bl(kH)$;
				\item $D(B^\prime)$ is an arbitrary but fixed defect group of $B^\prime$.
		\end{enumerate}}
\end{minipage}}
\end{center}
\end{nota}

\indent It has been proved recently that the Morita equivalence classes of blocks with Klein-four defect groups coincide with the splendid Morita equivalence classes of these blocks:

\begin{thm}[{\cite[Theorem 1.1]{KleinFourDefectGroupsCravenEatonLinckelmannK}}]\label{KleinFourDefectGroupsCravenEatonLinckelmannK}
Let $B^\prime\in\Bl(kH)$ such that $D(B^\prime)\cong V_4$. Then there exists a splendid Morita equivalence between $B^\prime$ and either $kV_4$ or $k\AlternatingGroup_4$ or $B_0(k\AlternatingGroup_5)$.
\end{thm}

As a consequence of \cref{KleinFourDefectGroupsCravenEatonLinckelmannK}, in the sequel we are able to express the ordinary characters of the trivial source modules belonging to $B^\prime$ if we are given the character values of $\Irr_K(B^\prime)$.

\subsection{There exists a splendid Morita equivalence between \texorpdfstring{$B^\prime$}{B'} and \texorpdfstring{$kV_4$}{kV4}}

As a starting point, we begin this subsection by quoting the following interesting result due to Richard Brauer:

\begin{prop}[{\cite[Proposition 7B]{BrauerIV}}] \label{Proposition_BrauerIV}
Let $B^\prime\in\Bl(kH)$. Suppose that there exists a splendid Morita equivalence between $B^\prime$ and $kV_4$. Denote the elements of the defect group $D(B^\prime)$ of $B'$ by $1, u, v, uv$. Denote the ordinary irreducible characters in $B^\prime$ by~$\chi_\alpha, \chi_\beta, \chi_\gamma, \chi_\delta$. Then, after a suitable relabelling of the characters $\chi_\iota$, $\iota\in \{\alpha, \beta,\gamma, \delta\}$, the following holds for the $i$-th character in $\Irr_K(B^\prime)$, for $1\leq i \leq 4$.
\begin{enumerate}[label=\textup{(\Roman*)}]
	\item If $u\sim v\sim uv$ in $H$ then there are positive integers $n_1, n_2, n_3$ and signs $\varepsilon_{i,1}, \varepsilon_{i,2}, \varepsilon_{i,3}\in\{\pm 1\}$ such that
	$$\chi_i(u) = \varepsilon_{i,1} n_1 + \varepsilon_{i,2} n_2 + \varepsilon_{i,3} n_3.$$ 
	\item If $u\not\sim v\sim uv$ in $H$ then there are positive integers $n_1, n_2, n_3$ and signs $\varepsilon_{i,1}, \varepsilon_{i,2}, \varepsilon_{i,3}\in\{\pm 1\}$ such that
	$$\chi_i(u) = \varepsilon_{i,1} n_1;\ \chi_i(v) = \varepsilon_{i,2} n_2 + \varepsilon_{i,3} n_3.$$ 
	\item If $u\not\sim v\not\sim uv\not\sim u$ in $H$ then there are positive integers $n_1, n_2, n_3$ and signs $\varepsilon_{i,1}, \varepsilon_{i,2}, \varepsilon_{i,3}\in\{\pm 1\}$ such that
	$$\chi_i(u) = \varepsilon_{i,1} n_1;\ \chi_i(v) = \varepsilon_{i,2} n_2;\ \chi_i(uv) = \varepsilon_{i,3} n_3.$$  
\end{enumerate}
In each of the three formul\textit{\ae} \textup{(I)}, \textup{(II)}, and \textup{(III)} above the integers $\varepsilon_{i,1}, \varepsilon_{i,2}, \varepsilon_{i,3}\in\{\pm 1\}$  in front of $n_1, n_2, n_3$ are given by
$$\begin{pmatrix}[1] \varepsilon_{1,1}  & \varepsilon_{1,2}  & \varepsilon_{1,3} \\ \varepsilon_{2,1}  & \varepsilon_{2,2}  & \varepsilon_{2,3} \\ \varepsilon_{3,1}  & \varepsilon_{3,2}  & \varepsilon_{3,3} \\ \varepsilon_{4,1}  & \varepsilon_{4,2}  & \varepsilon_{4,3} \end{pmatrix}=\begin{pmatrix}[1] +1 & +1 & +1\\ +1 & -1 & -1\\ -1 & +1 & -1\\ -1 & -1 & +1\end{pmatrix}.$$
\end{prop}

\indent Let $G:=V_4$ and let $B:=B_0(kG)$. Assume there exists a splendid Morita equivalence between~$B$ and $B'$ defined by a functor~$F$. We keep the notation from \cref{Notation_V4}. In particular, we set $\Irr_K(B):=\{\chi_1,\chi_2,\chi_3,\chi_4\}$ and $D(B):= \langle a, b\rangle$. Moreover, we set $\Irr_K(B'):=\{\chi_\alpha, \chi_\beta, \chi_\gamma, \chi_\delta\}$ and $D(B^\prime):= \langle a^\prime, b^\prime\rangle$. Furthermore, we endow the images of the $B$-modules under $F$ with the symbol ${}^\prime$.

\begin{rem}
Here we assume that only the existence of $F$ is known, but not an explicit bimodule inducing $F$.
\end{rem}

\begin{thm} \label{TS_characters_Blocks_Puig_equivalent_to_kV4}
The following assertions about the trivial source $B^\prime$-modules hold.
\begin{enumerate}[label=\textup{(\alph*)}]
	\item The trivial source $B^\prime$-modules are given as follows.
		\begin{enumerate}[label=\textup{(\roman*)}]
			\item We have~$\TS(B^\prime;D(B^\prime))=\{S^\prime\}$, where $S^\prime$ denotes the unique simple $B^\prime$-module (up to isomorphism).
			\item We have $\TS(B^\prime;\TrivialGroup)=\{P(S^\prime)\}$.
			\item Up to isomorphism, there are exactly three non-isomorphic trivial source $B'$-modules $M_1^\prime$, $M_2^\prime$ and $M_3^\prime$, respectively, with non-trivial cyclic vertices isomorphic to $C_2$.			
		\end{enumerate}
	\item After a suitable relabelling of the elements of $\Irr_K(B^\prime)$ we have
		$$\chi_{\widehat{\, S^\prime\, }}=\chi_\alpha,\ \ \chi_{\widehat{P(S^\prime)}}=\chi_\alpha + \chi_\beta + \chi_\gamma + \chi_\delta,\ \ \chi_{\widehat{M_1^\prime}} = \chi_\alpha + \chi_\beta,\ \ \chi_{\widehat{M_2^\prime}} = \chi_\alpha + \chi_\gamma,\ \ \chi_{\widehat{M_3^\prime}} = \chi_\alpha + \chi_\delta.$$
	\item The character values of $\Irr_K(B^\prime)$ determine $\chi_{\widehat{S^\prime}}$ uniquely.
\end{enumerate}
\end{thm}

\begin{proof}
By \cref{Omnibus_properties}(f) both all trivial source modules and their respective vertices are preserved under splendid Morita equivalences. Moreover, every Morita equivalence preserves projective indecomposable modules and simple modules. This proves (a). Decomposition matrices are preserved by Morita equivalences. Hence, the decomposition matrix of $B^\prime$ is
\begin{table}[H]
	\centering
    	{\scalebox{0.8}{
	\(\begin{array}{c|c}
		\hline
		& \varphi_{S^\prime}\\
		\hline
		\chi_\alpha & 1\\
		\chi_\beta & 1\\
		\chi_\gamma & 1\\
		\chi_\delta & 1\\
		\hline
	\end{array}\)
    }}
\end{table}
\noindent and, therefore, $\chi_{\widehat{P(S^\prime)}}=\chi_\alpha + \chi_\beta + \chi_\gamma + \chi_\delta$. Since composition factors are preserved by Morita equivalences, we deduce from the trivial source character table of $kV_4$ that each $M_i^\prime$ has the composition series

\[
	{\scalebox{0.8}{
\boxed{
	\!\begin{aligned}
	&S^\prime\\
	&S^\prime
	\end{aligned}
}
}}
\]
\noindent for $1\leq i\leq 3$. We claim that $\dim_k(\Hom_{kH}(S^\prime,M_i^\prime))=1$ for each $i\in\{1,2,3\}$.\newline
\noindent  Assume we are given $f\in \Hom_{kH}(S^\prime,M_i^\prime)$. Then, the morphism $f$ cannot be surjective. Hence, the induced morphism 
$$\widetilde{f}: S^\prime/ \Rad(S^\prime) \rightarrow  M_i^\prime/ \Rad(M_i^\prime),\quad s^\prime + \Rad(S^\prime) \mapsto f(s^\prime) + \Rad(M_i^\prime)$$
is also not surjective. Thus, by Schur's lemma, $\widetilde{f}(\Hd(S^\prime))=\{0\}$. Therefore, every element of $S^\prime\cong \Hd(S^\prime)$ is mapped to $\textup{Rad}(M_i^\prime)\cong S^\prime$ under the homomorphism $f$. Hence, for each $i\in\{1,2,3\}$, Schur's lemma implies that
$$\dim_k(\Hom_{kH}(S^\prime,M_i^\prime)) =\dim_k(\Hom_{kH}(S^\prime,S^\prime))=1.$$
\noindent The multiplicity of $\chi_\alpha$ as a constituent in $\chi_{\widehat{M_i^\prime}}$ is therefore equal to $1$ by \cref{Landrock_and_Landrock-Scott}(c). The decomposition matrix of $B'$ implies $\chi_\iota(1) = \dim_k(S^\prime)$ for all $\iota\in\{\alpha, \beta, \gamma, \delta \}$. Hence, for each~$i\in\{1,2,3\}$, $\chi_{\widehat{M_i^\prime}}=\chi_\alpha + \chi_j$ for some $j\in \{\beta, \gamma, \delta\}$ due to the shape of the composition series of the $M_i^\prime$. After a suitable relabelling the assertion in (b) follows now from the fact that
$$\dim_k(\Hom_{kH}(M_i^\prime,M_j^\prime)) =1\ \textup{for all}\ 1\leq i\neq j\leq 3.$$
\noindent In order to see this, set $M^\prime:=M_i^\prime$ and $N^\prime:=M_j^\prime$ for some arbitrary $i,j\in\{1,2,3\}\ \textup{with}\ i\neq~j$. Let $f\in \Hom_{kH}(M^\prime,N^\prime)$. Since $\dim_k(M^\prime)=\dim_{k}(N^\prime)$ but $M^\prime\not\cong N^\prime$, the shape of the composition series of $M$ and $N$ respectively implies the following assertion.
$$\ \ (*)\ \ \textup{The morphism}\ f\ \textup{maps every element of}\ \Hd(M^\prime)\ \textup{to}\ \Rad(N^\prime).$$
\noindent Indeed, this follows from $\Hd(M^\prime)\cong S^\prime\cong \Hd(N^\prime)$ and Schur's lemma. Moreover, we claim now that in the present case not only $f(\Rad(M^\prime))\leq \Rad(N^\prime)$ but also $f(\Rad(M^\prime))=\{0\}$ holds.\newline
\noindent In order to see this, let $\widetilde{m}\in\Rad(M^\prime)=\Rad(kH)\cdot M^\prime$. Then $\widetilde{m}=j\cdot m$ for some $m\in M^\prime$ and some~$j\in\Rad(kH)$. Consequently, $f(\widetilde{m})=f(j\cdot m)=j\cdot f(m)$. Due to $(*)$ we have~$f(m) \in \Rad(N^\prime)$. Hence,
$$f(\widetilde{m}) \in \Rad(kH)\cdot \Rad(N^\prime)=\Rad(\Rad(N^\prime))=\Rad(S^\prime)=\{0\}.$$
\noindent Therefore, $\dim_k(\Hom_{kH}(M^\prime,N^\prime)) = \dim_k(\Hom_{kH}(S^\prime,S^\prime))=1$. The assertions in (b) follow now from \cref{Landrock_and_Landrock-Scott}(c).\par
\noindent The ordinary character $\chi_{\widehat{S^\prime}}$ is the unique ordinary irreducible character $\chi$ lying in $B'$ with the property that the integer $\chi(x)$ is as large as possible for each $x\in D(B')$. Indeed, since~$S^\prime$ has maximal vertex $D(B')$, this follows from \cref{Proposition_BrauerIV} and \cref{Landrock_and_Landrock-Scott}(a)\&(b).
\end{proof}

\subsection{There exists a splendid Morita equivalence between \texorpdfstring{$B^\prime$}{B'} and \texorpdfstring{$k\AlternatingGroup_4$}{kA4}}
\noindent Let $G:=\AlternatingGroup_4$. We begin with the following auxiliary lemma where we keep the notation from \cref{PropertiesTS_A4_modules}. Moreover we set $S_1:=k, S_2:=S_\omega$, and $S_3 := S_{\overline{\omega}}$.
\begin{lem}
	We have~$\Soc(M)\cong S_1\oplus S_2\oplus S_3$.
\end{lem}

\begin{proof}
	Set $\widetilde{H}:=\langle (1,2) \rangle$. The $kG$-module $M:=\Ind_{\widetilde{H}}^{\AlternatingGroup_4}(k)$ is, up to isomorphism, the only trivial source $k\AlternatingGroup_4$-module with a non-trivial cyclic vertex. Hence, $M^{*}\cong M$, since dual modules have the same vertices. It follows from \cref{lem:Duals_and_Socles} that $\Soc(M)\cong{(\Hd(M))}^{*}$. By the Nakayama relations (see {\cite[Corollary 4.3.8]{Webb}}), we obtain $\Hom_{kG}(M,S_i)\cong \Hom_{k\widetilde{H}}(k,\Res_{\widetilde{H}}^{\AlternatingGroup_4}(S_i))$ for each $1\leq i \leq 3$. But $\Res_{\widetilde{H}}^{\AlternatingGroup_4}(S_i)$ is a one-dimensional trivial source $k\widetilde{H}$-module, hence isomorphic to $k$ for all $i\in \{1,2,3\}$. Consequently, $\dim_k\Hom_{kG}(M,S_i)=\dim_k\Hom_{k\widetilde{H}}(k,k)=1$ for all~$1\leq i\leq 3$. Since $\dim_k\Hom_{kG}(M,S_i)$ is equal to the multiplicity of $S_i$ as a direct summand in $\Hd(M)$ for all~$i\in\{1,2,3\}$, it follows that $\Hd(M)\cong S_1\oplus S_2\oplus S_3$. Hence, as $\Soc(M)\cong{(\Hd(M))}^{*}$, we deduce that $\Soc(M)\cong S_1\oplus S_2\oplus S_3$.
\end{proof}
\noindent Now, let $B:=B_0(kG)$ and let $B^\prime\in\Bl(kH)$ be a block of another group algebra $kH$. Assume there exists a splendid Morita equivalence between $B$ and $B'$ defined by a functor~$F$. We keep the notation from \cref{Notation_A4}. In particular, we set $\Irr_K(B):=\{\chi_1,\chi_2,\chi_3,\chi_4\}$ and $D(B):= \langle a, b\rangle$. Moreover, we set $\Irr_K(B'):=\{\chi_\alpha, \chi_\beta, \chi_\gamma, \chi_\delta\}$ and $D(B^\prime):= \langle a^\prime, b^\prime\rangle$. Furthermore, we denote the image of a $B$-module $Y$ under $F$ by $Y^\prime$.

\begin{thm}\label{Proposition_ts_Puig_A4}
The following assertions about the trivial source $B^\prime$-modules hold.
	\begin{enumerate}[label=\textup{(\alph*)}]
		\item Up to isomorphism there are exactly $7$ trivial source $B'$-modules. They are given as follows.
		\begin{enumerate}[label=\textup{(\roman*)}]
			\item We have $\TS(B^\prime;D(B^\prime))=\{S_1^\prime, S_2^\prime, S_3^\prime\}$.
			\item We have $\TS(B^\prime;\TrivialGroup)=\{P(S_1^\prime), P(S_2^\prime), P(S_3^\prime)\}$.
			\item Up to isomorphism, there is exactly one trivial source $B'$-module $M^\prime$ with non-trivial cyclic vertices isomorphic to $C_2$.
		\end{enumerate}
		\item After a suitable relabelling of the elements of $\Irr_K(B^\prime)$ we have:
		\begin{align*}
			\chi_{\widehat{S_1^\prime}}&=\chi_\alpha,\quad \chi_{\widehat{S_2^\prime}}=\chi_\beta,\quad  \chi_{\widehat{S_3^\prime}}=\chi_\gamma,\quad \chi_{\widehat{M^\prime}}=\chi_\alpha + \chi_\beta + \chi_\gamma + \chi_\delta,\\
			\chi_{\widehat{P(S_1^\prime)}}&=\chi_\alpha + \chi_\delta,\quad \chi_{\widehat{P(S_2^\prime)}}=\chi_\beta + \chi_\delta,\quad \chi_{\widehat{P(S_3^\prime)}}=\chi_\gamma + \chi_\delta.
		\end{align*}
		\item The set $\Irr_K(B^\prime)$ determines $\chi_\delta$ uniquely and in a purely character-theoretic way.
	\end{enumerate}
\end{thm}

\begin{proof}
By \cref{Omnibus_properties}(f) both all trivial source modules and their respective vertices are preserved under splendid Morita equivalences. Moreover, every Morita equivalence preserves projective indecomposable modules and simple modules. This proves (a). Decomposition matrices are preserved by Morita equivalences. Hence, the matrix
\begin{table}[H]
	\centering
    	{\scalebox{0.8}{
	\(\begin{array}{c|cccc}
		\hline
		&  \varphi_{S_1^\prime} & \varphi_{S_2^\prime} & \varphi_{S_3^\prime}\\
		\hline
		\chi_\alpha & 1 & 0 & 0\\
		\chi_\beta & 0 & 1 & 0\\
		\chi_\gamma & 0 & 0 & 1\\
		\chi_\delta & 1 & 1 & 1\\
		\hline
	\end{array}\)
    }}
\end{table}
\noindent is the decomposition matrix of $B'$. As the composition factors of the trivial source $k\AlternatingGroup_4$-module~$M$ are given by
$$[M] = 2\cdot S_1 + 2\cdot S_2 + 2\cdot S_3,$$
\noindent we deduce that the composition factors of $M^\prime$ are given by
$$[M^\prime] = 2\cdot S_1^\prime + 2\cdot S_2^\prime + 2\cdot S_3^\prime.$$
\noindent Since $\Soc(M)\cong S_1\oplus S_2\oplus S_3$, we deduce that $\Soc(M^\prime)\cong S_1^\prime\oplus S_2^\prime\oplus S_3^\prime$. Investigating the decomposition matrix of $B^\prime$ we see that
$$\chi_\delta(1)=\dim_k(S_1^\prime) + \dim_k(S_2^\prime) + \dim_k(S_3^\prime).$$
\noindent Hence, $\chi_\delta$ can be distinguished from $\chi_\alpha, \chi_\beta$ and $\chi_\gamma$ by its degree. Let $f\in\Hom_{kH}(S_1^\prime,M^\prime)$ be non-trivial. Since $$f(S_1^\prime)=f(\Soc(S_1^\prime))\leq \Soc(M^\prime)\cong S_1^\prime\oplus S_2^\prime\oplus S_3^\prime,$$
we deduce that $\dim_k(\Hom_{kH}(S_1^\prime,M^\prime))=1$. Analogously it follows that $\dim_k(\Hom_{kH}(S_2^\prime,M^\prime))=1$ and that $\dim_k(\Hom_{kH}(S_3^\prime,M^\prime))=1$. Consequently,~$\chi_{\widehat{M^\prime}}=\chi_\alpha + \chi_\beta + \chi_\gamma + \chi_\delta$. This proves Part (b) and Part (c).
\end{proof}

\subsection{There exists a splendid Morita equivalence between \texorpdfstring{$B^\prime$}{B'} and \texorpdfstring{$B_0(k\AlternatingGroup_5)$}{B0(kA5)}}
\noindent Let $G:=\AlternatingGroup_5$. We begin with the following auxiliary lemma where we keep the notation from \cref{PropertiesTS_A5_modules}.

\begin{lem} \label{Lemma_A5_M5_Trivial_Socle}
	The $k\AlternatingGroup_5$-module $M_5\cong\Sc(\AlternatingGroup_{5},Q_2)$ has trivial socle $\textup{Soc}(M_5)\cong k$.
\end{lem}

\begin{proof}
	We have $\Ind_{Q_2}^{\AlternatingGroup_5}(k)\cong M_5\oplus 2\cdot P(S_4)\oplus P(S_2)\oplus P(S_3)$. Thus, it follows from \cref{Lemma_Robinson} that
 $$\Res_{Q_2}^{\AlternatingGroup_5}(S_2)\cong {kQ_2}^{\textup{reg}}\quad\textup{and}\quad \Res_{Q_2}^{\AlternatingGroup_5}(S_3)\cong {kQ_2}^{\textup{reg}}.$$
We observe that for all $i\in \{2,3\}$ we have
	\begin{align*}
		\Hom_{kQ_2}(k,k) &\cong \Hom_{kQ_2}(k,\Soc({kQ_2}^{\textup{reg}}))\cong \Hom_{kQ_2}(k,{kQ_2}^{\textup{reg}})\cong \Hom_{kQ_2}(k,\Res_{Q_2}^{{\AlternatingGroup_{5}}}(S_i))\\
		&\cong \Hom_{k{\AlternatingGroup_{5}}}(\Ind_{Q_2}^{{\AlternatingGroup_{5}}}(k),S_i)\cong \Hom_{k{\AlternatingGroup_{5}}}(M_5\oplus 2\cdot P(S_4)\oplus P(S_2)\oplus P(S_3), S_i)\\
		&\cong \Hom_{k{\AlternatingGroup_{5}}}(\Hd(M_5\oplus 2\cdot P(S_4)\oplus P(S_2)\oplus P(S_3)), S_i)\\
		&\cong \Hom_{k{\AlternatingGroup_{5}}}(\Hd(M_5)\oplus 2\cdot S_4\oplus S_2\oplus S_3, S_i).
	\end{align*}
It follows from $\dim_k(\Hom_{kQ_2}(k,k))=1$ that neither $S_2$ nor $S_3$ is a direct summand of~$\Hd(M_5)$. Using the Nakayama relations (see {\cite[Corollary 4.3.8]{Webb}}), we compute for~$S_1=k$:
	\begin{align*}
		\Hom_{kQ_2}(k,k) &\cong\Hom_{kQ_2}(k,\Res_{Q_2}^{{\AlternatingGroup_{5}}}(S_1))\\
		&\cong \Hom_{k{\AlternatingGroup_{5}}}(\Ind_{Q_2}^{{\AlternatingGroup_{5}}}(k),S_1)\cong \Hom_{k{\AlternatingGroup_{5}}}(M_5\oplus 2\cdot P(S_4)\oplus P(S_2)\oplus P(S_3), S_1)\\
		&\cong \Hom_{k{\AlternatingGroup_{5}}}(\Hd(M_5\oplus 2\cdot P(S_4)\oplus P(S_2)\oplus P(S_3)), S_1)\\
		&\cong \Hom_{k{\AlternatingGroup_{5}}}(\Hd(M_5)\oplus 2\cdot S_4\oplus S_2\oplus S_3, S_1).
	\end{align*}
	\noindent Since $\dim_k(\Hom_{kQ_2}(k,k))=1$, the multiplicity of $S_1=k$ in $\Hd(M_5)$ is equal to one. Hence, $\Hd(M_5)\cong S_1=k$, since $M_5$ belongs to $B_0(k{\AlternatingGroup_{5}})$. As $\TS({\AlternatingGroup_{5}};Q_2) = \{M_5\}$, the $k{\AlternatingGroup_{5}}$-module $M_5$ is self-dual by \cref{Omnibus_properties}(g). 
    Using \cref{lem:Duals_and_Socles} we conclude that $\Soc(M_5)\cong \Hd(M_5)^{*}\cong k^{*}\cong k$.
\end{proof}

Now, let $B:=B_0(kG)$, and let $B^\prime\in\Bl(kH)$ be a block of another group algebra $kH$. Assume there exists a splendid Morita equivalence between $B$ and $B'$ defined by a functor~$F$. We keep the notation from \cref{Notation_A5}. In particular, we set $\Irr_K(B):=\{\chi_1,\chi_2,\chi_3,\chi_5\}$ and $D(B):= \langle a, b\rangle$. Moreover, we set $\Irr_K(B'):=\{\chi_\alpha, \chi_\beta, \chi_\gamma, \chi_\delta\}$ and $D(B^\prime):= \langle a^\prime, b^\prime\rangle$. Furthermore, we endow the images of the $B$-modules under $F$ with the symbol ${}^\prime$.

\begin{thm}\label{Proposition_ts_Puig_A5}
The following assertions about the trivial source $B^\prime$-modules hold.
	\begin{enumerate}[label=\textup{(\alph*)}]
	\item Up to isomorphism, there exist precisely the following $7$ trivial source $B'$-modules.
		\begin{enumerate}[label=\textup{(\roman*)}]
			\item We have $\TS(B^\prime;D(B^\prime))=\{S_1^\prime, M_7^\prime, M_8^\prime\}$.
			\item We have $\TS(B^\prime;\TrivialGroup)=\{P(S_1^\prime), P(M_7^\prime), P(M_8^\prime)\}$.
			\item Up to isomorphism, there is exactly one trivial source $B'$-module $M_5^\prime$ with non-trivial cyclic vertices isomorphic to $C_2$.
		\end{enumerate}
		\item After a suitable relabelling of the elements of $\Irr_K(B^\prime)$ we have:
		\begin{align*}
			\chi_{\widehat{S_1^\prime}}&=\chi_\alpha,\quad \chi_{\widehat{M_7^\prime}}=\chi_\delta,\quad \chi_{\widehat{M_8^\prime}}=\chi_\delta,\quad \chi_{\widehat{M_5^\prime}}= \chi_\alpha + \chi_\delta,\\
			\chi_{\widehat{P(S_1^\prime)}}&= \chi_\alpha + \chi_\beta + \chi_\gamma + \chi_\delta,\quad \chi_{\widehat{P(S_2^\prime)}}=\chi_\gamma + \chi_\delta,\quad \chi_{\widehat{P(S_3^\prime)}}=\chi_\beta + \chi_\delta.
		\end{align*}
		\item The set $\Irr_K(B^\prime)$ determines both $\chi_\alpha$ and $\chi_\delta$ uniquely and in a purely character-theoretic way.
	\end{enumerate}
\end{thm}

\begin{proof}
By \cref{Omnibus_properties}(f) both all trivial source modules and their respective vertices are preserved under splendid Morita equivalences. Moreover, every Morita equivalence preserves projective indecomposable modules and simple modules. This proves (a). Decomposition matrices are preserved by Morita equivalences. Hence, the matrix 
\begin{table}[H]
	\centering
    	{\scalebox{0.8}{
	\(\begin{array}{c|cccc}
		\hline
		&  \varphi_{S_1^\prime} & \varphi_{S_2^\prime} & \varphi_{S_3^\prime}\\
		\hline
		\chi_\alpha & 1 & 0 & 0\\
		\chi_\beta & 1 & 0 & 1\\
		\chi_\gamma & 1 & 1 & 0\\
		\chi_\delta & 1 & 1 & 1\\
		\hline
\end{array}\)
}}
\end{table}
\noindent is the decomposition matrix of $B'$. As the composition factors of the trivial source $k\AlternatingGroup_5$-module~$M_5$ are given by
$$[M_5] = 2\cdot S_1 + S_2 + S_3,$$
\noindent we deduce that the composition factors of $M_5^\prime$ are given by
$$[M_5^\prime] = 2\cdot S_1^\prime + S_2^\prime + S_3^\prime,$$
following the order of the decomposition matrix. Moreover, the respective composition factors of the projective indecomposable $B^\prime$-modules are given as follows:
\begin{align*}
	[P(S_1^\prime)]&=4\cdot S_1^\prime + 2\cdot S_2^\prime + 2\cdot S_3^\prime,\\
	[P(S_2^\prime)]&=2\cdot S_1^\prime + 2\cdot S_2^\prime + 1\cdot S_3^\prime,\\
	[P(S_3^\prime)]&=2\cdot S_1^\prime + 1\cdot S_2^\prime + 2\cdot S_3^\prime.
\end{align*}

\noindent It follows from the decomposition matrix of $B^\prime$ that $\chi_{S_1^\prime} = \chi_\alpha$. By \cref{Lemma_A5_M5_Trivial_Socle}, we have $\Soc(M_5^\prime)\cong S_1^\prime$ and, therefore, $\dim_k(\Hom_{kH}(S_1^\prime, M_5^\prime))=1$.
\noindent Hence, $\langle \chi_\alpha, \chi_{\widehat{M_5^\prime}} \rangle =1$. It follows from the composition factors of $M_5^\prime$ and from the decomposition matrix of $B'$ that~$\chi_{\widehat{M_5^\prime}} = \chi_\alpha + \chi_\delta$. We have $[M_7^\prime] = [M_8^\prime] = S_1^\prime + S_2^\prime + S_3^\prime$. The socles of the two trivial source $B'$-modules $M_7^\prime$ and $M_8^\prime$ are $S_2^\prime$ and $S_3^\prime$, respectively, since the socles of their preimages under $F$ are $S_2$ and $S_3$, respectively. Hence, $\dim_k(\Hom_{kH}(S_1^\prime, M_7^\prime)) = 0 = \dim_k(\Hom_{kH}(S_1^\prime, M_8^\prime))$. Therefore, $\chi_{\widehat{M_7^\prime}} = \chi_\delta = \chi_{\widehat{M_8^\prime}}$. The assertions in $(c)$ follow from the decomposition matrix of $B^\prime$.
\end{proof}

\begin{rem}\label{CharDegreesDeterminePuigClass}

    Note that the (splendid) Morita equivalence class of a block $b$ with Klein-four defect groups is already determined by the character degrees: if $\Irr_K(b)=\{\chi_\alpha, \chi_\beta, \chi_\gamma, \chi_\delta\}$ then, by the decomposition matrices and by the previous discussion, precisely the following cases can occur.
    \begin{enumerate}
        \item Suppose that there exists a unique largest character degree, say $\chi_\delta(1)$. Then we have either $\chi_\delta(1) = \chi_\beta(1) + \chi_\gamma(1) + \chi_\alpha(1)$ --- in which case $b$ is splendidly Morita equivalent to~$k\AlternatingGroup_4$ --- or $\chi_\delta(1) = \chi_\beta(1) + \chi_\gamma(1) - \chi_\alpha(1)$ --- in which case $b$ is splendidly Morita equivalent to $B_0(k\AlternatingGroup_5)$; 
        \item Suppose that there does not exist a unique largest character degree in $\Irr_K(b)$. Then $b$ is splendidly Morita equivalent to $kV_4$. 
    \end{enumerate}
 
\end{rem}

\section{Two examples}
Consider the group algebra $kG$, where $k$ is an algebraically closed field of characteristic $2$. The main aim of this final section is to illustrate, by means of two examples, how the results from \cref{sec:TS_characters_in_blocks_with_Klein_four_defect_groups} can be applied. Firstly, a certain family of finite group algebras which are of domestic representation type: we let $G=D_{4v}$, the dihedral group of order $4v$, where $v$ is an odd integer larger than~$1$. It follows that the principal block of the group algebra~$kG$ is splendidly Morita equivalent to the group algebra $kV_4$ in this case (see \cref{PIMsOfD4v}). Secondly, we consider a group algebra $kG$ of domestic representation type where $G$ is a concrete group of order $972$. It turns out that $B_0(kG)$ is splendidly Morita equivalent to the algebra $k\AlternatingGroup_4$ in this case (see \cref{Cor:B_0(G)_is_splendidly_Morita_equivalent_to_kA_4}).

\subsection{Example 1: the dihedral groups \texorpdfstring{$D_{4v}\ (v\in\mathbb{Z}_{\geq 3}\ \textup{odd})$}{D4v (vEZ>3 odd)}}
In this subsection we assume that $G=D_{4v} = \langle r,s \ |\ r^{2v}=s^2={(sr)^2}=1\rangle$, the dihedral group of order $4v$, where $v\geq 3$ is an odd integer. Set $\omega_m := \exp{(- \frac{2\pi i m}{2v})}$. The $v+3$ conjugacy classes of $D_{4v}$ are given by

\begin{table}[H]
	\centering
	{\resizebox{\textwidth}{!}{
	\begin{tabular}{@{}l@{}l@{}l@{}l@{}l@{}l@{}l@{}l@{}l@{}l@{}l@{}l@{}l@{}l@{}l@{}l@{}l@{}l@{}}
		\(\begin{array}{ccccc}
   [1] = \{1\};\ & [r] = \{r,r^{-1}\};\ & [r^2] = \{r^2,r^{-2}\};\ & \cdots & [r^{v-1}] = \{r^{v-1},r^{-(v-1)}\}; \\
   \\
   \phantom{D}[r^{v}] = \{r^{v}\};\phantom{D} & 
   \multicolumn{2}{l}{[s] =\{sr^{2a}\ |\ 1\leq a\leq v\};} & \multicolumn{2}{r}{[sr] =\{sr^{2a-1}\ |\ 1\leq a\leq v\}.}\\
		\end{array}\)
	\end{tabular}
}}
\end{table}

\noindent The $\frac{v+1}{2}$ distinct $2'-$conjugacy classes of $D_{4v}$ are given by $[1], [r^2], [r^4], \dots, [r^{v-1}]$.\par\bigskip
\noindent The ordinary character table of $G$ is as given in \cref{table:Ordinary_ct_D_4v} (cf. \cite[§18.3]{JamesLiebeck}).

\begin{table}[ht]
	\centering
	{\resizebox{12cm}{!}{\renewcommand{\arraystretch}{1.1}		
			\begin{tabular}{| c | c c c c c c c c c|}
				\hline
				$g$ & $1$ & $r$ & $r^2$ & $r^3$ & $\cdots$ & $r^{v-1}$ & $r^v$ & $s$ & $sr$\\
				\hline
				$\chi_1$ & $1$ & $1$ & $1$ & $1$ & $\cdots$ & $1$ & $1$ & $1$ & $1$\\
				$\chi_2$ & $1$ & $-1$ & $1$ & $-1$ & $\cdots$ & $1$ & $-1$ & $1$ & $-1$\\
				$\chi_3$ & $1$ & $1$ & $1$ & $1$ & $\cdots$ & $1$ & $1$ & $-1$ & $-1$\\
				$\chi_4$ & $1$ & $-1$ & $1$ & $-1$ & $\cdots$ & $1$ & $-1$ & $-1$ & $1$\\
				$\chi_5$ & $2$ & $\omega_1 + \overline{\omega}_1$ & $\omega_1^2 + \overline{\omega}_1^2$ & $\omega_1^3 + \overline{\omega}_1^3$ & $\cdots$ & $\omega_1^{v-1} + \overline{\omega}_1^{v-1}$ & $\omega_1^v + \overline{\omega}_1^v$ & $0$ & $0$\\
				$\chi_6$ & $2$ & $\omega_2 + \overline{\omega}_2$ & $\omega_2^2 + \overline{\omega}_2^2$ & $\omega_2^3 + \overline{\omega}_2^3$ & $\cdots$ & $\omega_2^{v-1} + \overline{\omega}_2^{v-1}$ & $\omega_2^v + \overline{\omega}_2^v$ & $0$ & $0$\\
				$\vdots$ & $\vdots$ & $\vdots$ & $\vdots$ & $\vdots$ & $\cdots$ & $\vdots$ & $\vdots$ & $\vdots$ & $\vdots$\\
				$\chi_{v+3}$ & $2$ & $\omega_{v-1} + \overline{\omega}_{v-1}$ & $\omega_{v-1}^2 + \overline{\omega}_{v-1}^2$ & $\omega_{v-1}^3 + \overline{\omega}_{v-1}^3$ & $\cdots$ & $\omega_{v-1}^{v-1} + \overline{\omega}_{v-1}^{v-1}$ & $\omega_{v-1}^v + \overline{\omega}_{v-1}^v$ & $0$ & $0$\\
				\hline
	\end{tabular}	}}
	\caption{Ordinary character table of $D_{4v}$ for odd $v\geq 3$}
	\label{table:Ordinary_ct_D_4v}

\end{table}

\noindent We set $Q_1 := \TrivialGroup, Q_2 := \langle sr^v \rangle, Q_3 := \langle r^v \rangle, Q_4 := \langle s \rangle$, and $Q_5 := \langle s, r^v \rangle\in \textup{Syl}_2(D_{4v})$. Furthermore, we choose $\mathscr{S}_2(D_{4v}) = \{Q_1, Q_2, Q_3, Q_4, Q_5\}$. The lattice of subgroups in~$\mathscr{S}_2(D_{4v})$~is:

$$
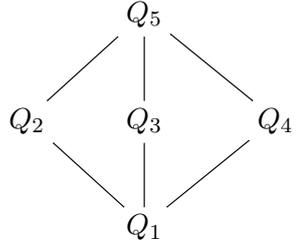
\resizebox{0.2\textwidth}{!}{
\xymatrix{ & Q_5 & \\
	Q_2 \ar@{-}[ur] & Q_3 \ar@{-}[u] & Q_4 \ar@{-}[ul] \\
	& Q_1 \ar@{-}[ul] \ar@{-}[u] \ar@{-}[ur]	& \ \phantom{Q_4}\ . }
    }
    $$

\begin{lem}
The following assertions hold:
\begin{table}[H]
	\centering
	{\resizebox{\textwidth}{!}{
	\begin{tabular}{@{}l@{}l@{}l@{}l@{}l@{}l@{}l@{}l@{}l@{}l@{}l@{}l@{}l@{}l@{}}
		\(\begin{array}{ccccc}
   N_{D_{4v}}(Q_1) = {D_{4v}};\ & N_{D_{4v}}(Q_2) = {Q_5};\ & N_{D_{4v}}(Q_3) = {D_{4v}};\ & N_{D_{4v}}(Q_4) = {Q_5};\ & N_{D_{4v}}(Q_5) = {Q_5};\\ [0,5ex]
   \overline{N}_{D_{4v}}(Q_1) \cong {D_{4v}}; & \overline{N}_{D_{4v}}(Q_2) \cong C_2; & \overline{N}_{D_{4v}}(Q_3) \cong D_{2v}; & \overline{N}_{D_{4v}}(Q_4) \cong C_2; & \overline{N}_{D_{4v}}(Q_5) \cong Q_1.\\
		\end{array}\)
	\end{tabular}
}}
\end{table}
\end{lem}

\begin{proof}
    Direct computation.
\end{proof}

\begin{nota}
We denote the $\frac{v+1}{2}$ simple $kG$-modules by $S_1:=k, S_2, \ldots , S_{\frac{v+1}{2}}$. Here,
$$\chi_1^\circ =  \varphi_k,\ \chi_5^\circ = \varphi_{S_2},\ \chi_6^\circ = \varphi_{S_3},\ \chi_7^\circ = \varphi_{S_4},\ \ldots,\ \chi_{\frac{v+7}{2}}^\circ = \varphi_{S_{\frac{v+1}{2}}}.$$ 
\end{nota}
\noindent We set $\widehat{C}:=\sum_{g\in C}^{}{g}$ for every conjugacy class $C$ of $G$. Recall that for each $\chi\in\Irr(G)$ there exists (e.g. by \cite[Theorem 2.3.2]{LuxPah}) for any conjugacy class $C$ of $G$ a central character 
$$\omega_\chi: Z(KG)\rightarrow K\quad\textup{with}\quad \widehat{C} \mapsto \omega_\chi(\widehat{C})=\frac{|C|}{\chi(1)}\chi(g).$$

\begin{lem}
The following assertions hold for the central characters of $D_{4v}$.\par\smallskip
\begin{tabular}{@{}p{0.495\textwidth}@{\hspace{0,93em}}p{0.458\textwidth}@{}}
\parbox[t]{\linewidth}{
\hspace{1em}\textup{(a)} For $1\le t \le 4$ we have:\small
\begin{enumerate}[label=\textup{(\roman*)}, leftmargin=1em, labelsep=0.3em, topsep=0pt, itemsep=0.2em]
		\item $\omega_{\chi_t}(\widehat{[1_G]}) = 1$ and $\omega_{\chi_t}(\widehat{[r^v]}) = \frac{1}{2} ({(-1)}^{t+1})$;
  		\item $\omega_{\chi_t}(\widehat{[r^u]})=2\cdot {(-1)}^{u(t+1)}\quad (1\leq u\leq v-1)$;
		\item $\omega_{\chi_t}(\widehat{[s]}) = {(-1)}^{1+ \left \lceil{{\frac{t}{2}}}\right \rceil}v$ and $\omega_{\chi_t}(\widehat{[sr]})={(-1)}^{\left \lfloor{{\frac{t}{2}}}\right \rfloor}v$. 
\end{enumerate}
}
&
\parbox[t]{\linewidth}{
\hspace{1em}\textup{(b)} For $5\leq t\leq v+3$ we have:\small
\begin{enumerate}[label=\textup{(\roman*)}, leftmargin=1em, labelsep=0.3em, topsep=0pt, itemsep=0.2em]
		\item $\omega_{\chi_t}(\widehat{[1_G]}) = 1$ and $\omega_{\chi_t}(\widehat{[r^v]}) = \frac{1}{2} (\omega_{t-4}^{u} + {\overline{\omega}_{t-4}^{u}})$;
  		\item $\omega_{\chi_t}(\widehat{[r^u]}) = \omega_{t-4}^{u} + {\overline{\omega}_{t-4}^{u}}\quad (1\leq u\leq v-1)$;
		\item $\omega_{\chi_t}(\widehat{[s]}) = \omega_{\chi_t}(\widehat{[sr]}) = 0$. 
\end{enumerate}
}
\end{tabular}

\end{lem}

\begin{proof}
    Direct computation using \cref{table:Ordinary_ct_D_4v}.
\end{proof}

\begin{prop}\label{PIMsOfD4v}
\begin{enumerate}[label=\textup{(\alph*)}]
        \item We have $|\Bl_2(kG)|=\frac{v+1}{2}$.
	\item The block algebra $B_0(kG)$ is splendidly Morita equivalent to the algebra $kV_4$, we have $\Irr(B_0(kG))=\{\chi_1,\chi_2, \chi_3, \chi_4\}$, and the decomposition matrix $\mathfrak{D}(B_0(kG))$ is
\begin{table}[H]
	\centering
    	{\scalebox{0.8}{
	\(\begin{array}{c|c}
		\hline
		&  1_{G_{p^\prime}}\\
		\hline
		\chi_1 & 1\\
		\chi_2 & 1\\
		\chi_3 & 1\\
		\chi_4 & 1\\
		\hline
\end{array}\)
}}
\end{table}
	\item For each $i\in \{2,\ldots, \frac{v+1}{2}\}$, we have a block $B_{i-1}$ such that $\Irr(B_{i-1}) = \{\chi_{i+3}, \chi_{v+5-i}\}$ and the decomposition matrix $\mathfrak{D}(B_{i-1}(kG))$ is
\begin{table}[H]
	\centering
    	{\scalebox{0.8}{
	\(\begin{array}{c|c}
		\hline
	&  \varphi_{S_i}\\
		\hline
		\chi_{i+3} & 1\\
		\chi_{v+5-i} & 1\\
		\hline
	\end{array}\)
    }}
\end{table}
	\item The projective indecomposable characters of $kG$ are:
	$$\Phi_{\varphi_k} = \chi_1 + \chi_2 + \chi_3 + \chi_4,\quad \Phi_{\varphi_{S_i}}= \chi_{i+3} + \chi_{v+5-i},\ \textup{for}\ 2\leq i\leq \frac{v+1}{2}.$$	
\end{enumerate}
\end{prop}

\begin{proof}
\begin{enumerate}[label=\textup{(\alph*)}]
        \item Note first that 
        the group $G$ is $2$-nilpotent, as $G\cong C_{2v}\rtimes C_2$. The group $G$ has a normal~$2$-complement if and only if for every simple $kG$-module $S$, the composition factors of the projective cover $P(S)$ of $S$ are all isomorphic to~$S$. See \cite[Theorem 8.4.1]{Webb}. Hence, the Cartan matrix of $kG$ is diagonal and we have $|\Bl_2(kG)|=\frac{v+1}{2}$.
        \item As $B_0(kG)$ has the Sylow $2$-subgroups of $G$ as defect groups, we deduce that a defect group of $B_0(kG)$ is $Q_5\cong V_4$. Now, by 
\cref{KleinFourDefectGroupsCravenEatonLinckelmannK}, there exists a splendid Morita equivalence between $B_0(kG)$ and either $kV_4$ or $k\AlternatingGroup_4$ or $B_0(k\AlternatingGroup_5)$. Since the Cartan matrix of $kG$ is diagonal, the only possible shape of the decomposition matrix of $kG$ is as claimed. Hence, $B_0(kG)$ is splendidly Morita equivalent to the algebra $kV_4$. Next, we determine the elements of $\Irr(B_0(kG))$. Since $1_G\in~B_0(kG)$, it follows from \cite[Theorem 4.4.8]{LuxPah} that $\Irr(B_0(kG))=\{\chi_1, \chi_2, \chi_3, \chi_4\}$, as $\chi_1^\circ = \chi_2^\circ = \chi_3^\circ = \chi_4^\circ$.
        \item \noindent Fix $m_1\in \{1,\ldots , \frac{v-1}{2}\}$. Define $m_2:=v-m_1$. Then, $\omega_{m_1}^{j}+ \overline{\omega}_{m_1}^{j} = \omega_{m_2}^{j}+ \overline{\omega}_{m_2}^{j}$ for every even number~$j\in \{1,\ldots , v\}$, since

\begin{table}[H]
    \centering
     {\resizebox{13.8cm}{!}{
    \begin{tabular}{lrclc}
         & $\exp\left(\frac{-2m_1\pi ij}{2v}\right) + \exp\left(\frac{2m_1\pi ij}{2v}\right)$ & $=$ & $\exp\left(\frac{-2(v-m_1)\pi ij}{2v}\right) + \exp\left(\frac{2(v-m_1)\pi ij}{2v}\right)$\\[1,5ex]
        $\Leftrightarrow$ & $2\cdot \cos\left(\frac{2m_1\pi j}{2v}\right)$ & $=$ & $2\cdot \cos\left(\frac{2(v-m_1)\pi j}{2v}\right)$\\[1,5ex]
        $\Leftrightarrow$ & $\cos\left(\frac{m_1\pi j}{v}\right)$ & $=$ & $\cos\left(\frac{(v-m_1)\pi j}{v}\right)$\\[1,5ex]
        $\Leftrightarrow$ & $\cos\left(\frac{m_1\pi j}{v}\right)$ & $=$ & $\cos\left(\pi j - \frac{m_1\pi j}{v}\right)$.        
    \end{tabular}
 }}
\end{table}
\noindent Hence, $\chi_{4+m_1}^{\circ} = \chi_{4+m_2}^{\circ}$ for every choice of $m_1$. As $G$ has precisely $\frac{v+1}{2}$ different $2$-blocks, the claim follows now from \cite[Theorem 4.4.8]{LuxPah}.
\item The assertion follows from Part (b) and Part (c).
\end{enumerate}
\end{proof}

We set $H:=N_G(Q_3)/Q_3 = D_{4v}/Q_3\cong D_{2v}$ and denote the image of an element $g\in D_{4v}$ under the canonical epimorphism $D_{4v}\twoheadrightarrow H$ by $\overline{g}$. We see that $|\Lin(H)|=|H/H^\prime|=2$ and we denote the two linear characters of the group $H$ by $\lambda_1 := 1_H$ and $\lambda_2$.

\begin{prop}\label{TSCT_data_of_D4v}
\begin{enumerate}[label=\textup{(\alph*)}]
	\item The $\frac{v+1}{2}$ distinct $2^\prime$-conjugacy classes of $H$ are:
 $$[\, \overline{1}\,], [\overline{r^2}], [\overline{r^4}],\ldots , [\overline{r^{v-1}}].$$ 
	\item The ordinary characters of the projective indecomposable $kH$-modules are:
	\begin{enumerate}[label=\textup{(\roman*)}]
		\item the projective character $\Phi_{1_{H_{p^\prime}}}:=\lambda_1 + \lambda_2$ where $\lambda_1, \lambda_2\in \textup{Lin}(H)$;
		\item all $\frac{v-1}{2}$ non-linear $\chi_i\in \Irr_K(H)$.
	\end{enumerate}
	\item The ordinary characters of the trivial source $kG$-modules with vertex $Q_3$ are:
	\begin{enumerate}[label=\textup{(\roman*)}]
		\item the character $\chi_1 + \chi_3$;
		\item the $\frac{v-1}{2}$ ordinary irreducible characters $\chi_{6}, \chi_{8},\ldots , \chi_{v+1}, \chi_{v+3}$.
	\end{enumerate}
	\item We have $\TS(G;Q_2)=\{M_1\}, \TS(G;Q_4)=\{M_2\}$, and $\TS(G;Q_5)=\{M_3\}$ for suitable trivial source $kG$-modules $M_1, M_2$, and $M_3$. Moreover, the following assertions hold:

\begin{tabular}[t]{@{}l l l@{}}
\hspace{0,5em}\textup{(i)} $\chi_{\widehat{M_1}} = \chi_1 + \chi_4$;\quad\quad\quad &
\textup{(ii)} $\chi_{\widehat{M_2}} = \chi_1 + \chi_2$;\quad\quad\quad &
\textup{(iii)} $\chi_{\widehat{M_3}} = \chi_1$.
\end{tabular}

\end{enumerate}
\end{prop}

\begin{proof}
In the group $H$, the following conjugacy classes of the group $G$ fuse:

\begin{table}[H]
	\centering
	{\resizebox{\textwidth}{!}{
	\begin{tabular}{@{}l@{}l@{}l@{}l@{}l@{}l@{}l@{}l@{}l@{}l@{}l@{}l@{}l@{}l@{}}
		\(\begin{array}{cccccc}
   [1]\ \textup{and}\ [r^v];\ \ & [r]\ \textup{and}\ [r^{v-1}];\ \ & [r^2]\ \textup{and}\ [r^{v-2}];\ \ & \ldots\ & [r^{\frac{v-1}{2}}]\ \textup{and}\ [r^{\frac{v+1}{2}}];\ \ & [s]\ \textup{and}\ [sr].
		\end{array}\)
	\end{tabular}
}}
\end{table}

\noindent Hence, the group $H$ has $2+\frac{v-1}{2}$ conjugacy classes. Since $s$ is a $2$-element, this implies Part (a). Note that a defect group of $B_0(kH)$ is given by a Sylow $2$-subgroup of $H$. Hence, $\mathfrak{D}(B_0(kH))$ is given as follows:
\begin{table}[H]
	\centering
    	{\scalebox{0.8}{
	\(\begin{array}{c|c}
		\hline
	&  1_{H_{p^\prime}}\\
	\hline
	\lambda_1 & 1\\
	\lambda_2 & 1\\
		\hline
\end{array}\)
}}
\end{table}
\noindent Moreover, since $H$ is $p$-solvable, the decomposition matrix of~$kH$ contains an identity matrix of maximum possible size. Hence, (b) follows. The trivial source $kG$-modules with vertices isomorphic to $Q_3$ are obtained by inflation of the projective indecomposable $kH$-modules. The same is true for their ordinary characters. We examine which characters of $\Irr_K(G)$ have $Q_3$ in their kernel. Since
$$\omega_{m}^{v}+ \overline{\omega}_{m}^{v} =\exp\left( \frac{-2m\pi iv}{2v}\right) + \exp\left( \frac{2m\pi iv}{2v}\right) = 2 \cos \left( \frac{2m\pi v}{2v}\right) = 2\cos(m\pi)=2\cdot {(-1)}^m,$$
we deduce that $Q_3$ is in the kernel of $\chi_t$ if and only if~$t\in \{1,3\}\cup \{6, 8, 10, \ldots , v+3\}$. This proves~(c). Recall from \cref{PIMsOfD4v}(b) that the algebra $B_0(kG)$ is splendidly Morita equivalent to the algebra $kV_4$. Hence, all the ordinary characters stated in Part (d) have to occur as ordinary characters of trivial source $B_0(kG)$-modules due to \cref{TS_characters_Blocks_Puig_equivalent_to_kV4}. Note that
$$(\chi_1 + \chi_4)(sr^v) = (\chi_1 + \chi_4)(sr)=2>0\ \textup{and}$$
$$(\chi_1 + \chi_2)(s) = 2>0.$$
\noindent As $|[\overline{N}_G(Q_i)]_{p^\prime}|=1$ for $i\in\{2,4,5\}$, the assertions follow now from \cref{Landrock_and_Landrock-Scott}(a)\&(b).
\end{proof}

\begin{nota}\label{2_Prime_Classes_Of_D_4v}
    For the $2^\prime$-conjugacy classes of the normaliser quotients we make the following choices using \cref{{TSCT_data_of_D4v}}(a):
$${[\overline{N}_{D_{4v}}(Q_1)]}_{2^\prime} = \{1,r^2, r^4, \ldots, r^{v-1}\},\quad {[\overline{N}_{D_{4v}}(Q_2)]}_{2^\prime} = \{1Q_2\},$$
$${[\overline{N}_{D_{4v}}(Q_3)]}_{2^\prime} = \{1Q_3,r^2Q_3,r^4Q_3, \ldots, r^{v-1}Q_3\},\  {[\overline{N}_{D_{4v}}(Q_4)]}_{2^\prime} = \{1Q_4\},\  {[\overline{N}_{D_{4v}}(Q_5)]}_{2^\prime} = \{1Q_5\}.$$
\end{nota}

\begin{thm}\label{Triv_2_D_4v}
Labelling the columns of the first and third block column of $\Triv_2(D_{4v})$ by the elements $1, r^2, r^4, \ldots, r^{v-1}$ and labelling the ordinary characters of $D_{4v}$ as in \cref{table:Ordinary_ct_D_4v}, $\Triv_2(D_{4v})$ is as given in \cref{table:TSCT_D_4v_when_p_is_2}. In particular, we have $T_{3,3} = T_{3,1}$ and $T_{3,1} = \frac{1}{2} \cdot T_{1,1}$ after a suitable permutation of the rows of $\Triv_2(D_{4v})$.
\end{thm}

\begin{proof}
By \cref{2_Prime_Classes_Of_D_4v} we may choose our labels of the columns of the block columns of $\Triv_2(D_{4v})$ as asserted. The ordinary characters of the elements of $\TS(G;Q_i)$ are given by \cref{PIMsOfD4v}(d) when $i=1$, by \cref{TSCT_data_of_D4v}(e) when $i\in \{2,4,5\}$, and by \cref{TSCT_data_of_D4v}(c) when~$i=3$. This justifies our labelling of the rows of $\Triv_2(D_{4v})$.\par
\indent Next, it is immediate from \cref{rem:tsctbl} that $T_{i,j} = {\mathbf{0}}$ for $1\leq i<j\leq 5$. By \cref{LemmaRickard} and the choice of our labelling of the columns of the block columns of $\Triv_2(D_{4v})$, all remaining entries of $\Triv_2(D_{4v})$ except for the entry of $T_{5,5}$ are obtained by evaluating the trivial source characters at certain elements of the group $D_{4v}$, which is easily done using \cref{table:Ordinary_ct_D_4v}. The claim that ~$T_{5,5}=[1]$ follows from \cref{rem:tsctbl}(b).\par
\indent  If we label the rows of the block rows of $\Triv_2(D_{4v})$ as in \cref{table:TSCT_D_4v_when_p_is_2}, then we obtain that $T_{3,3} = T_{3,1}$ and $T_{3,1} = \frac{1}{2} \cdot T_{1,1}$, due to the following arguments: by \cref{TSCT_data_of_D4v}(d) we have $T_{3,3} = T_{3,1}$, as inflation of characters from $H$ to $G$ does not change their values at conjugacy classes; and by \cref{PIMsOfD4v}, for each projective character of $kG$ the number of constituents with odd indices is equal to the number of constituents with even indices. Therefore, also the equation $T_{3,1} = \frac{1}{2}\cdot T_{1,1}$ holds, as was to be verified.
\end{proof}

\begin{table}[ht]
	\centering
	{\resizebox{\linewidth}{!}{\renewcommand{\arraystretch}{1.1}	
    \begin{tabular}{@{}l@{}l@{}l@{}l@{}l@{}l@{}l@{}l@{}l@{}l@{}l@{}l@{}l@{}l@{}l@{}l@{}l@{}l@{}}
		\(\begin{array}{|l V{4} c c c c c c | c | c c c c c|c|c|}
	\hline
	Q_v\ (1\leq v\leq 5) & \multicolumn{6}{c|}{Q_{1}\cong C_1} & \multicolumn{1}{c|}{Q_{2}\cong C_2} & \multicolumn{5}{c|}{Q_{3}\cong C_2} & \multicolumn{1}{c|}{Q_{4}\cong C_2} & \multicolumn{1}{c|}{Q_{5}\cong V_4}\\ \hline
	N_v\ (1\leq v\leq 5) & \multicolumn{6}{c|}{N_{1}\cong D_{4v}} & \multicolumn{1}{c|}{N_{2}\cong V_4} & \multicolumn{5}{c|}{N_{3}\cong D_{4v}} & \multicolumn{1}{c|}{N_{4}\cong V_4} & \multicolumn{1}{c|}{N_{5}\cong V_4}\\ \hline
	n_j\ \in\ N_v & 1 & r^2 & r^4 & r^6 & \cdots & r^{v-1} & {1} & {1} &  {r^2} & {r^4} & \cdots & {r^{v-1}} &  {1}  &  {1}\\ \Xhline{4\arrayrulewidth}
	\chi_1 + \chi_2 + \chi_3 + \chi_4  &  4  &  4  &  4  &  4  &  \cdots  &  4  &  0  &  0  &  0  &  0  &  \cdots  &  0  &  0  &  0 \\
	\chi_{4+1} + \chi_{4+(v-1)}  &  4  &  2\left(\omega_1^2 + \overline{\omega}_1^2\right)  &  2\left(\omega_1^4 + \overline{\omega}_1^4\right)  &  2\left(\omega_1^6 + \overline{\omega}_1^6\right)  &  \cdots  &  2\left(\omega_1^{v-1} + \overline{\omega}_1^{v-1}\right)  &  0  &  0  &  0  &  0  &  \cdots  &  0  &  0  &  0 \\
	\chi_{4+2} + \chi_{4+(v-2)}  &  4  &  2\left(\omega_2^2 + \overline{\omega}_2^2\right)  &  2\left(\omega_2^4 + \overline{\omega}_2^4\right)  &  2\left(\omega_2^6 + \overline{\omega}_2^6\right)  &  \cdots  &  2\left(\omega_2^{v-1} + \overline{\omega}_2^{v-1}\right)  &  0  &  0  &  0  &  0  &  \cdots  &  0  &  0  &  0 \\
	\vdots  &  \vdots  &  \vdots  &  \vdots  &  \vdots  &  \cdots  &  \vdots  &  \vdots  &  \vdots  &  \vdots  &  \vdots  &  \cdots  &  \vdots  &  \vdots  &  \vdots \\
	\chi_{4+\frac{v-1}{2}} + \chi_{4+\frac{v+1}{2}}  &  4  &  2\left(\omega_{\frac{v-1}{2}}^2 + \overline{\omega}_{\frac{v-1}{2}}^2\right)  &  2\left(\omega_{\frac{v-1}{2}}^4 + \overline{\omega}_{\frac{v-1}{2}}^4\right)  &  2\left(\omega_{\frac{v-1}{2}}^6 + \overline{\omega}_{\frac{v-1}{2}}^6\right)  &  \cdots  &  2\left(\omega_{\frac{v-1}{2}}^{v-1} + \overline{\omega}_{\frac{v-1}{2}}^{v-1}\right)  &  0  &  0  &  0  &  0  &  \cdots  &  0  &  0  &  0 \\
	\hline
	\chi_1 + \chi_4  &  2  &  2  &  2  &  2  &  \cdots  &  2  &  2  &  0  &  0  &  0  &  \cdots  &  0  &  0  &  0 \\
	\hline
	\chi_1 + \chi_3  &  2  &  2  &  2  &  2  &  \cdots  &  2  &  0  &  2  &  2  &  2  &  \cdots  &  2  &  0  &  0 \\
	\chi_{4+2}  &  2  &  \omega_2^2 + \overline{\omega}_2^2  &  \omega_2^4 + \overline{\omega}_2^4  &  \omega_2^6 + \overline{\omega}_2^6  &  \cdots  &  \omega_2^{v-1} + \overline{\omega}_2^{v-1}  &  0  &  2  &  \omega_2^2 + \overline{\omega}_2^2  &  \omega_2^4 + \overline{\omega}_2^4  &  \cdots  &  \omega_2^{v-1} + \overline{\omega}_2^{v-1}  &  0  &  0 \\
	\chi_{4+(v-1)}  &  2  &  \omega_{v-1}^2 + \overline{\omega}_{v-1}^2  &  \omega_{v-1}^4 + \overline{\omega}_{v-1}^4  &  \omega_{v-1}^6 + \overline{\omega}_{v-1}^6  &  \cdots  &  \omega_{v-1}^{v-1} + \overline{\omega}_{v-1}^{v-1}  &  0  &  2  &  \omega_{v-1}^2 + \overline{\omega}_{v-1}^2  &  \omega_{v-1}^4 + \overline{\omega}_{v-1}^4  &  \cdots  &  \omega_{v-1}^{v-1} + \overline{\omega}_{v-1}^{v-1}  &  0  &  0 \\
	\chi_{4+4}  &  2  &  \omega_4^2 + \overline{\omega}_4^2  &  \omega_4^4 + \overline{\omega}_4^4  &  \omega_4^6 + \overline{\omega}_4^6  &  \cdots  &  \omega_4^{v-1} + \overline{\omega}_4^{v-1}  &  0  &  2  &  \omega_4^2 + \overline{\omega}_4^2  &  \omega_4^4 + \overline{\omega}_4^4  &  \cdots  &  \omega_4^{v-1} + \overline{\omega}_4^{v-1}  &  0  &  0 \\
	\chi_{4+(v-3)}  &  2  &  \omega_{v-3}^2 + \overline{\omega}_{v-3}^2  &  \omega_{v-3}^4 + \overline{\omega}_{v-3}^4  &  \omega_{v-3}^6 + \overline{\omega}_{v-3}^6  &  \cdots  &  \omega_{v-3}^{v-1} + \overline{\omega}_{v-3}^{v-1}  &  0  &  2  &  \omega_{v-3}^2 + \overline{\omega}_{v-3}^2  &  \omega_{v-3}^4 + \overline{\omega}_{v-3}^4  &  \cdots  &  \omega_{v-3}^{v-1} + \overline{\omega}_{v-3}^{v-1}  &  0  &  0 \\
	\vdots  &  \vdots  &  \vdots  &  \vdots  &  \vdots  &  \cdots  &  \vdots  &  \vdots  &  \vdots  &  \vdots  &  \vdots  &  \cdots  &  \vdots  &  \vdots  &  \vdots \\
	\hline
	\chi_1 + \chi_2  &  2  &  2  &  2  &  2  &  \cdots  &  2  &  2  &  0  &  0  &  0  &  \cdots  &  0  &  2  &  0 \\
	\hline
	\chi_1  &  1  &  1  &  1  &  1  &  \cdots  &  1  &  1  &  1  &  1  &  1  &  \cdots  &  1  &  1  &  1 \\
	\hline
\end{array}\)
\end{tabular}
	}}
	\caption{Trivial source character table of $D_{4v}$ at $p=2$ for odd $v$}
	\label{table:TSCT_D_4v_when_p_is_2}
\title{}
\end{table}

\subsection{Example 2: a computational example involving a concrete group algebra of domestic representation type}
In this computational subsection, we suppose that our $p$-modular system $(K, \mathcal{O}, k)$ is a standard
$p$-modular system for $G$ as in \cite[Definition 4.2.10]{LuxPah}. Moreover, we assume that the group~$G$ is isomorphic to $N\rtimes H$, where $N= \langle a \rangle\times \langle b \rangle\times \langle c \rangle\times \langle d \rangle$ with $a,b,c$, and $d$ commuting elements of order~$3$ and $H = \langle {\bf{(1,2)(3,4)}}, {\bf{(2,3,4)}} \rangle \cong \AlternatingGroup_4$ acting on $N$ via permutations of the letters~$a,b,c,d$.\\
\indent The aim of this example is twofold: first, using\ \GAP,\ we compute the ordinary character table~$X(G)$ of~$G$ and, by the results obtained in the previous section, together with the entries of~$X(G)$, show a possible way to determine the trivial source characters of those modules which belong to blocks with Klein-four defect groups. Second, we determine the trivial source character table of $G$ at the prime two and see that --- although our main results determine a large portion of $\Triv_2(G)$ ---  this is not an easy consequence of the knowledge of the trivial source characters of the $kG$-modules alone, mainly because we have blocks with various defects and therefore many different Brauer constructions to compute.\\
\indent We can obtain the group $G$ in \GAP\ 4.12.2 as follows:\par\medskip

\begin{footnotesize}
\begin{verbatim}
    H1 := CyclicGroup(3);;  D := DirectProduct(H1,H1,H1,H1);;  gensD := GeneratorsOfGroup(D);;
    d1 := gensD[1];; d2 := gensD[2];; d3 := gensD[3];; d4 := gensD[4];;
    phi1 := GroupHomomorphismByImages(D,D,[d1,d2,d3,d4],[d2,d1,d4,d3]);;
    phi2 := GroupHomomorphismByImages(D,D,[d1,d2,d3,d4],[d1,d4,d2,d3]);;
    AutD := AutomorphismGroup(D);;  H := Subgroup(AutD,[phi1,phi2]);;
    G := SemidirectProduct(H,D); # returns: <pc group with 7 generators>
    IdSmallGroup(G);; # returns: [972,877]
\end{verbatim}
\end{footnotesize}
\par\medskip
\noindent Hence, the group $G$ is isomorphic to $\texttt{SmallGroup}(972,877)$ of \GAP's SmallGroups database. Note that \GAP\ operates from the right. Hence, \GAP\ sees $G$ as $G= H \ltimes N$. Our code code continues:
\par\medskip
\begin{footnotesize}
\begin{verbatim}
    incl1:=Embedding(G,1);; # the embedding of H into G
    incl2:=Embedding(G,2);; # the embedding of N into G
\end{verbatim}
\end{footnotesize}
\par\medskip
\noindent Next, we compute the ordinary character table of $G$ as follows, where we set $w:=\exp(2\pi i/3)$:\par\medskip
\begin{footnotesize}
\begin{verbatim}
    ctG := CharacterTable(G);; Display(ctG); # This yields the character table of the group G.
\end{verbatim}
\end{footnotesize}
\noindent See \cref{table:ct_G_is_N_rtimes_H}. Next, we determine the conjugacy classes used by \GAP's character table \texttt{ctG}:\par\medskip
\begin{footnotesize}
\begin{verbatim}
    ConjugacyClasses(ctG); # returns:
    # [ <identity> of ...^G, f1^G, f2^G, f4^G, f1^2^G, f1*f4^G, f1*f5^G, f2*f4^G, f4^2^G, 
    # f4*f5^G, f1^2*f4^G, f1^2*f5^G, f1*f4^2^G, f1*f4*f5^G, f1*f5^2^G, 
    # f2*f4^2^G, f2*f4*f5^G, f4^2*f5^G, f4*f5*f6^G, f1^2*f4^2^G, f1^2*f4*f5^G, 
    # f1^2*f5^2^G, f1*f4^2*f5^G, f1*f4*f5^2^G, f2*f4^2*f5^G, f4^2*f5^2^G, 
    # f4^2*f5*f6^G, f4*f5*f6*f7^G, f1^2*f4^2*f5^G, f1^2*f4*f5^2^G, f1*f4^2*f5^2^G, 
    # f2*f4^2*f5^2^G, f4^2*f5^2*f6^G, f4^2*f5*f6*f7^G, f1^2*f4^2*f5^2^G, 
    # f4^2*f5^2*f6^2^G, f4^2*f5^2*f6*f7^G, f4^2*f5^2*f6^2*f7^G, f4^2*f5^2*f6^2*f7^2^G ]

    ClassNames(ctG); # returns:
    # [ "1a", "3a", "2a", "3b", "3c", "3d", "9a", "6a", "3e", "3f", "3g", "9b", 
    # "3h", "9c", "9d", "6b", "6c", "3i", "3j", "3k", "9e", "9f", "9g", "9h", "6d",  
    # "3l", "3m", "3n", "9i", "9j", "9k", "6e", "3o", "3p", "9l", "3q", "3r", "3s", "3t" ]
\end{verbatim}
\end{footnotesize}
\noindent Next, we identify the generators of the group $G$ used by \GAP:
\par\medskip
\begin{footnotesize}
\begin{verbatim}
    incl2(d1); # returns: f4
    incl2(d2); # returns: f5
    incl2(d3); # returns: f6
    incl2(d4); # returns: f7
    gensH := GeneratorsOfGroup(H); # returns:
    # [ [ f1, f2, f3, f4 ] -> [ f2, f1, f4, f3 ], [ f1, f2, f3, f4 ] -> [ f1, f4, f2, f3 ] ]
    incl1(gensH[1]); # returns: f3
    incl1(gensH[2]); # returns: f1
    A := incl1(gensH[1]);; B := incl1(gensH[2]);;
    A*B*A*A*B*A*B*A; # returns: f2
    gensH[1]*gensH[2]*gensH[1]*gensH[1]*gensH[2]*gensH[1]*gensH[2]*gensH[1];
    # corresponds to ABAABABA and returns: [ f1, f2, f3, f4 ] -> [ f4, f3, f2, f1 ]
\end{verbatim}
\end{footnotesize}
\par\medskip
\noindent Hence, \texttt{f2} corresponds to the group element ${\bf{(1,4)(2,3)}}\in H$.

\begin{properties}\label{2blocks_of_G}
Now, we determine the distribution of the ordinary irreducible characters of~$G$ into~$p$-blocks using \GAP\ and the result is as follows.
	\begin{enumerate}[label=\textup{(\alph*)}]
		\item $|\textup{Bl}_2(G)|=27$;
		\item $\textup{Irr}_K(B_0(G))=\{\chi_1, \chi_4, \chi_5, \chi_{10}\}$, $\textup{Irr}_K(B_1(G))=\{\chi_2, \chi_6, \chi_9, \chi_{11}\}$,\par\noindent $\textup{Irr}_K(B_2(G))=\{\chi_3, \chi_7, \chi_8, \chi_{12}\}$, $\textup{Irr}_K(B_{21}(G))=\{\chi_{31}, \chi_{32}\}$, $\textup{Irr}_K(B_{22}(G))=\{\chi_{33}, \chi_{35}\}$, $\textup{Irr}_K(B_{23}(G))=\{\chi_{34}, \chi_{36}\}$, all other elements of $\Irr_K(G)$ belong to a block of defect zero.
	\end{enumerate}
\end{properties}

\noindent The assertions follow from the following \GAP\ code. Note that the indices of the $p$-blocks of $G$ are shifted by one as \GAP\ denotes the principal block of $G$ by $B_1(G)$.\par\medskip
\begin{footnotesize}
\begin{verbatim} 
    pbs := PrimeBlocks(ctG,2);;  pbs.block; # returns:
    # [ 1, 2, 3, 1, 1, 2, 3, 3, 2, 1, 2, 3, 4, 5, 6, 7, 8, 9, 10, 11, 12, 13, 14, 
    # 15, 16, 17, 18, 19, 20, 21, 22, 22, 23, 24, 23, 24, 25, 26, 27 ]
    pbs.defect; # returns:
    # [ 2, 2, 2, 0, 0, 0, 0, 0, 0, 0, 0, 0, 0, 0, 0, 0, 0, 0, 0, 0, 0, 1, 1, 1, 0, 0, 0 ]
\end{verbatim}
\end{footnotesize}
\par\medskip
\noindent Next, we set $Q_1:= \TrivialGroup, Q_2:= \langle {\bf{(1,2)(3,4)}} \rangle$, and $Q_3:= \langle {\bf{(1,2)(3,4), (1,3)(2,4)}} \rangle \cong V_4$. Furthermore, we choose $\mathscr{S}_2(G)=\{Q_1, Q_2, Q_3\}$. Then, the chain of subgroups $Q_1\leq Q_2\leq Q_3$ is the lattice of subgroups in $\mathscr{S}_2(G)$.
\newpage
\begin{landscape}
\vspace*{\fill}
\begin{center}
\begin{table}[h]
	{\resizebox{22.2cm}{!}{
\begin{tabular}{@{}l@{}l@{}l@{}}
\hline
\(\begin{array}{|l|ccccccccccccccccccccccccccccccccccccccc|}
  & 1a & 3a & 2a & 3b & 3c & 3d & 9a & 6a & 3e & 3f & 3g & 9b & 3h & 9c & 9d & 6b & 6c & 3i & 3j & 3k & 9e & 9f & 9g & 9h & 6d & 3l & 3m & 3n & 9i & 9j & 9k & 6e & 3o & 3p & 9l & 3q & 3r & 3s & 3t\\ \hline
\chi_{1} & 1 & 1 & 1 & 1 & 1 & 1 & 1 & 1 & 1 & 1 & 1 & 1 & 1 & 1 & 1 & 1 & 1 & 1 & 1 & 1 & 1 & 1 & 1 & 1 & 1 & 1 & 1 & 1 & 1 & 1 & 1 & 1 & 1 & 1 & 1 & 1 & 1 & 1 & 1\\
\chi_{2} & 1 & 1 & 1 & {\omega}^{2} & 1 & {\omega}^{2} & {\omega}^{2} & {\omega}^{2} & {\omega} & {\omega} & {\omega}^{2} & {\omega}^{2} & {\omega} & {\omega} & {\omega} & {\omega} & {\omega} & 1 & 1 & {\omega} & {\omega} & {\omega} & 1 & 1 & 1 & {\omega}^{2} & {\omega}^{2} & {\omega}^{2} & 1 & 1 & {\omega}^{2} & {\omega}^{2} & {\omega} & {\omega} & {\omega}^{2} & 1 & 1 & {\omega}^{2} & {\omega}\\
\chi_{3} & 1 & 1 & 1 & {\omega} & 1 & {\omega} & {\omega} & {\omega} & {\omega}^{2} & {\omega}^{2} & {\omega} & {\omega} & {\omega}^{2} & {\omega}^{2} & {\omega}^{2} & {\omega}^{2} & {\omega}^{2} & 1 & 1 & {\omega}^{2} & {\omega}^{2} & {\omega}^{2} & 1 & 1 & 1 & {\omega} & {\omega} & {\omega} & 1 & 1 & {\omega} & {\omega} & {\omega}^{2} & {\omega}^{2} & {\omega} & 1 & 1 & {\omega} & {\omega}^{2}\\
\chi_{4} & 1 & {\omega}^{2} & 1 & 1 & {\omega} & {\omega}^{2} & {\omega}^{2} & 1 & 1 & 1 & {\omega} & {\omega} & {\omega}^{2} & {\omega}^{2} & {\omega}^{2} & 1 & 1 & 1 & 1 & {\omega} & {\omega} & {\omega} & {\omega}^{2} & {\omega}^{2} & 1 & 1 & 1 & 1 & {\omega} & {\omega} & {\omega}^{2} & 1 & 1 & 1 & {\omega} & 1 & 1 & 1 & 1\\
\chi_{5} & 1 & {\omega} & 1 & 1 & {\omega}^{2} & {\omega} & {\omega} & 1 & 1 & 1 & {\omega}^{2} & {\omega}^{2} & {\omega} & {\omega} & {\omega} & 1 & 1 & 1 & 1 & {\omega}^{2} & {\omega}^{2} & {\omega}^{2} & {\omega} & {\omega} & 1 & 1 & 1 & 1 & {\omega}^{2} & {\omega}^{2} & {\omega} & 1 & 1 & 1 & {\omega}^{2} & 1 & 1 & 1 & 1\\
\chi_{6} & 1 & {\omega}^{2} & 1 & {\omega}^{2} & {\omega} & {\omega} & {\omega} & {\omega}^{2} & {\omega} & {\omega} & 1 & 1 & 1 & 1 & 1 & {\omega} & {\omega} & 1 & 1 & {\omega}^{2} & {\omega}^{2} & {\omega}^{2} & {\omega}^{2} & {\omega}^{2} & 1 & {\omega}^{2} & {\omega}^{2} & {\omega}^{2} & {\omega} & {\omega} & {\omega} & {\omega}^{2} & {\omega} & {\omega} & 1 & 1 & 1 & {\omega}^{2} & {\omega}\\
\chi_{7} & 1 & {\omega} & 1 & {\omega} & {\omega}^{2} & {\omega}^{2} & {\omega}^{2} & {\omega} & {\omega}^{2} & {\omega}^{2} & 1 & 1 & 1 & 1 & 1 & {\omega}^{2} & {\omega}^{2} & 1 & 1 & {\omega} & {\omega} & {\omega} & {\omega} & {\omega} & 1 & {\omega} & {\omega} & {\omega} & {\omega}^{2} & {\omega}^{2} & {\omega}^{2} & {\omega} & {\omega}^{2} & {\omega}^{2} & 1 & 1 & 1 & {\omega} & {\omega}^{2}\\
\chi_{8} & 1 & {\omega}^{2} & 1 & {\omega} & {\omega} & 1 & 1 & {\omega} & {\omega}^{2} & {\omega}^{2} & {\omega}^{2} & {\omega}^{2} & {\omega} & {\omega} & {\omega} & {\omega}^{2} & {\omega}^{2} & 1 & 1 & 1 & 1 & 1 & {\omega}^{2} & {\omega}^{2} & 1 & {\omega} & {\omega} & {\omega} & {\omega} & {\omega} & 1 & {\omega} & {\omega}^{2} & {\omega}^{2} & {\omega}^{2} & 1 & 1 & {\omega} & {\omega}^{2}\\
\chi_{9} & 1 & {\omega} & 1 & {\omega}^{2} & {\omega}^{2} & 1 & 1 & {\omega}^{2} & {\omega} & {\omega} & {\omega} & {\omega} & {\omega}^{2} & {\omega}^{2} & {\omega}^{2} & {\omega} & {\omega} & 1 & 1 & 1 & 1 & 1 & {\omega} & {\omega} & 1 & {\omega}^{2} & {\omega}^{2} & {\omega}^{2} & {\omega}^{2} & {\omega}^{2} & 1 & {\omega}^{2} & {\omega} & {\omega} & {\omega} & 1 & 1 & {\omega}^{2} & {\omega}\\
\chi_{10} & 3 & 0 & -1 & 3 & 0 & 0 & 0 & -1 & 3 & 3 & 0 & 0 & 0 & 0 & 0 & -1 & -1 & 3 & 3 & 0 & 0 & 0 & 0 & 0 & -1 & 3 & 3 & 3 & 0 & 0 & 0 & -1 & 3 & 3 & 0 & 3 & 3 & 3 & 3\\
\chi_{11} & 3 & 0 & -1 & 3 {\omega}^{2} & 0 & 0 & 0 & -{\omega}^{2} & 3 {\omega} & 3 {\omega} & 0 & 0 & 0 & 0 & 0 & -{\omega} & -{\omega} & 3 & 3 & 0 & 0 & 0 & 0 & 0 & -1 & 3 {\omega}^{2} & 3 {\omega}^{2} & 3 {\omega}^{2} & 0 & 0 & 0 & -{\omega}^{2} & 3 {\omega} & 3 {\omega} & 0 & 3 & 3 & 3 {\omega}^{2} & 3 {\omega}\\
\chi_{12} & 3 & 0 & -1 & 3 {\omega} & 0 & 0 & 0 & -{\omega} & 3 {\omega}^{2} & 3 {\omega}^{2} & 0 & 0 & 0 & 0 & 0 & -{\omega}^{2} & -{\omega}^{2} & 3 & 3 & 0 & 0 & 0 & 0 & 0 & -1 & 3 {\omega} & 3 {\omega} & 3 {\omega} & 0 & 0 & 0 & -{\omega} & 3 {\omega}^{2} & 3 {\omega}^{2} & 0 & 3 & 3 & 3 {\omega} & 3 {\omega}^{2}\\
\chi_{13} & 4 & 1 & 0 & -{\omega}+2 {\omega}^{2} & 1 & 1 & {\omega}^{2} & 0 & 2 {\omega}-{\omega}^{2} & -2 & 1 & {\omega}^{2} & 1 & {\omega}^{2} & {\omega} & 0 & 0 & 1 & 2 {\omega}-{\omega}^{2} & 1 & {\omega}^{2} & {\omega} & {\omega}^{2} & {\omega} & 0 & -2 & 1 & 4 & {\omega}^{2} & {\omega} & {\omega} & 0 & 1 & -{\omega}+2 {\omega}^{2} & {\omega} & -{\omega}+2 {\omega}^{2} & -2 & 2 {\omega}-{\omega}^{2} & 4\\
\chi_{14} & 4 & 1 & 0 & 2 {\omega}-{\omega}^{2} & 1 & 1 & {\omega} & 0 & -{\omega}+2 {\omega}^{2} & -2 & 1 & {\omega} & 1 & {\omega} & {\omega}^{2} & 0 & 0 & 1 & -{\omega}+2 {\omega}^{2} & 1 & {\omega} & {\omega}^{2} & {\omega} & {\omega}^{2} & 0 & -2 & 1 & 4 & {\omega} & {\omega}^{2} & {\omega}^{2} & 0 & 1 & 2 {\omega}-{\omega}^{2} & {\omega}^{2} & 2 {\omega}-{\omega}^{2} & -2 & -{\omega}+2 {\omega}^{2} & 4\\
\chi_{15} & 4 & {\omega} & 0 & -{\omega}+2 {\omega}^{2} & {\omega}^{2} & {\omega} & 1 & 0 & 2 {\omega}-{\omega}^{2} & -2 & {\omega}^{2} & {\omega} & {\omega} & 1 & {\omega}^{2} & 0 & 0 & 1 & 2 {\omega}-{\omega}^{2} & {\omega}^{2} & {\omega} & 1 & 1 & {\omega}^{2} & 0 & -2 & 1 & 4 & {\omega} & 1 & {\omega}^{2} & 0 & 1 & -{\omega}+2 {\omega}^{2} & 1 & -{\omega}+2 {\omega}^{2} & -2 & 2 {\omega}-{\omega}^{2} & 4\\
\chi_{16} & 4 & {\omega}^{2} & 0 & 2 {\omega}-{\omega}^{2} & {\omega} & {\omega}^{2} & 1 & 0 & -{\omega}+2 {\omega}^{2} & -2 & {\omega} & {\omega}^{2} & {\omega}^{2} & 1 & {\omega} & 0 & 0 & 1 & -{\omega}+2 {\omega}^{2} & {\omega} & {\omega}^{2} & 1 & 1 & {\omega} & 0 & -2 & 1 & 4 & {\omega}^{2} & 1 & {\omega} & 0 & 1 & 2 {\omega}-{\omega}^{2} & 1 & 2 {\omega}-{\omega}^{2} & -2 & -{\omega}+2 {\omega}^{2} & 4\\
\chi_{17} & 4 & {\omega}^{2} & 0 & -{\omega}+2 {\omega}^{2} & {\omega} & {\omega}^{2} & {\omega} & 0 & 2 {\omega}-{\omega}^{2} & -2 & {\omega} & 1 & {\omega}^{2} & {\omega} & 1 & 0 & 0 & 1 & 2 {\omega}-{\omega}^{2} & {\omega} & 1 & {\omega}^{2} & {\omega} & 1 & 0 & -2 & 1 & 4 & 1 & {\omega}^{2} & 1 & 0 & 1 & -{\omega}+2 {\omega}^{2} & {\omega}^{2} & -{\omega}+2 {\omega}^{2} & -2 & 2 {\omega}-{\omega}^{2} & 4\\
\chi_{18} & 4 & {\omega} & 0 & 2 {\omega}-{\omega}^{2} & {\omega}^{2} & {\omega} & {\omega}^{2} & 0 & -{\omega}+2 {\omega}^{2} & -2 & {\omega}^{2} & 1 & {\omega} & {\omega}^{2} & 1 & 0 & 0 & 1 & -{\omega}+2 {\omega}^{2} & {\omega}^{2} & 1 & {\omega} & {\omega}^{2} & 1 & 0 & -2 & 1 & 4 & 1 & {\omega} & 1 & 0 & 1 & 2 {\omega}-{\omega}^{2} & {\omega} & 2 {\omega}-{\omega}^{2} & -2 & -{\omega}+2 {\omega}^{2} & 4\\
\chi_{19} & 4 & 1 & 0 & -3 {\omega}-2 {\omega}^{2} & 1 & {\omega}^{2} & 1 & 0 & -2 {\omega}-3 {\omega}^{2} & -2 {\omega} & {\omega}^{2} & 1 & {\omega} & {\omega}^{2} & 1 & 0 & 0 & 1 & -{\omega}+2 {\omega}^{2} & {\omega} & {\omega}^{2} & 1 & {\omega} & {\omega}^{2} & 0 & -2 {\omega}^{2} & {\omega}^{2} & 4 {\omega}^{2} & {\omega} & {\omega}^{2} & {\omega} & 0 & {\omega} & {\omega}+3 {\omega}^{2} & {\omega} & 2 {\omega}-{\omega}^{2} & -2 & 3 {\omega}+{\omega}^{2} & 4 {\omega}\\
\chi_{20} & 4 & 1 & 0 & -2 {\omega}-3 {\omega}^{2} & 1 & {\omega} & 1 & 0 & -3 {\omega}-2 {\omega}^{2} & -2 {\omega}^{2} & {\omega} & 1 & {\omega}^{2} & {\omega} & 1 & 0 & 0 & 1 & 2 {\omega}-{\omega}^{2} & {\omega}^{2} & {\omega} & 1 & {\omega}^{2} & {\omega} & 0 & -2 {\omega} & {\omega} & 4 {\omega} & {\omega}^{2} & {\omega} & {\omega}^{2} & 0 & {\omega}^{2} & 3 {\omega}+{\omega}^{2} & {\omega}^{2} & -{\omega}+2 {\omega}^{2} & -2 & {\omega}+3 {\omega}^{2} & 4 {\omega}^{2}\\
\chi_{21} & 4 & {\omega} & 0 & -3 {\omega}-2 {\omega}^{2} & {\omega}^{2} & 1 & {\omega} & 0 & -2 {\omega}-3 {\omega}^{2} & -2 {\omega} & {\omega} & {\omega}^{2} & {\omega}^{2} & 1 & {\omega} & 0 & 0 & 1 & -{\omega}+2 {\omega}^{2} & 1 & {\omega} & {\omega}^{2} & {\omega}^{2} & 1 & 0 & -2 {\omega}^{2} & {\omega}^{2} & 4 {\omega}^{2} & 1 & {\omega} & {\omega}^{2} & 0 & {\omega} & {\omega}+3 {\omega}^{2} & 1 & 2 {\omega}-{\omega}^{2} & -2 & 3 {\omega}+{\omega}^{2} & 4 {\omega}\\
\chi_{22} & 4 & {\omega}^{2} & 0 & -2 {\omega}-3 {\omega}^{2} & {\omega} & 1 & {\omega}^{2} & 0 & -3 {\omega}-2 {\omega}^{2} & -2 {\omega}^{2} & {\omega}^{2} & {\omega} & {\omega} & 1 & {\omega}^{2} & 0 & 0 & 1 & 2 {\omega}-{\omega}^{2} & 1 & {\omega}^{2} & {\omega} & {\omega} & 1 & 0 & -2 {\omega} & {\omega} & 4 {\omega} & 1 & {\omega}^{2} & {\omega} & 0 & {\omega}^{2} & 3 {\omega}+{\omega}^{2} & 1 & -{\omega}+2 {\omega}^{2} & -2 & {\omega}+3 {\omega}^{2} & 4 {\omega}^{2}\\
\chi_{23} & 4 & {\omega}^{2} & 0 & -3 {\omega}-2 {\omega}^{2} & {\omega} & {\omega} & {\omega}^{2} & 0 & -2 {\omega}-3 {\omega}^{2} & -2 {\omega} & 1 & {\omega} & 1 & {\omega} & {\omega}^{2} & 0 & 0 & 1 & -{\omega}+2 {\omega}^{2} & {\omega}^{2} & 1 & {\omega} & 1 & {\omega} & 0 & -2 {\omega}^{2} & {\omega}^{2} & 4 {\omega}^{2} & {\omega}^{2} & 1 & 1 & 0 & {\omega} & {\omega}+3 {\omega}^{2} & {\omega}^{2} & 2 {\omega}-{\omega}^{2} & -2 & 3 {\omega}+{\omega}^{2} & 4 {\omega}\\
\chi_{24} & 4 & {\omega} & 0 & -2 {\omega}-3 {\omega}^{2} & {\omega}^{2} & {\omega}^{2} & {\omega} & 0 & -3 {\omega}-2 {\omega}^{2} & -2 {\omega}^{2} & 1 & {\omega}^{2} & 1 & {\omega}^{2} & {\omega} & 0 & 0 & 1 & 2 {\omega}-{\omega}^{2} & {\omega} & 1 & {\omega}^{2} & 1 & {\omega}^{2} & 0 & -2 {\omega} & {\omega} & 4 {\omega} & {\omega} & 1 & 1 & 0 & {\omega}^{2} & 3 {\omega}+{\omega}^{2} & {\omega} & -{\omega}+2 {\omega}^{2} & -2 & {\omega}+3 {\omega}^{2} & 4 {\omega}^{2}\\
\chi_{25} & 4 & {\omega} & 0 & 3 {\omega}+{\omega}^{2} & {\omega}^{2} & 1 & {\omega}^{2} & 0 & {\omega}+3 {\omega}^{2} & -2 {\omega} & {\omega} & 1 & {\omega}^{2} & {\omega} & 1 & 0 & 0 & 1 & 2 {\omega}-{\omega}^{2} & 1 & {\omega}^{2} & {\omega} & 1 & {\omega}^{2} & 0 & -2 {\omega}^{2} & {\omega}^{2} & 4 {\omega}^{2} & {\omega} & 1 & {\omega} & 0 & {\omega} & -2 {\omega}-3 {\omega}^{2} & {\omega}^{2} & -{\omega}+2 {\omega}^{2} & -2 & -3 {\omega}-2 {\omega}^{2} & 4 {\omega}\\
\chi_{26} & 4 & {\omega}^{2} & 0 & {\omega}+3 {\omega}^{2} & {\omega} & 1 & {\omega} & 0 & 3 {\omega}+{\omega}^{2} & -2 {\omega}^{2} & {\omega}^{2} & 1 & {\omega} & {\omega}^{2} & 1 & 0 & 0 & 1 & -{\omega}+2 {\omega}^{2} & 1 & {\omega} & {\omega}^{2} & 1 & {\omega} & 0 & -2 {\omega} & {\omega} & 4 {\omega} & {\omega}^{2} & 1 & {\omega}^{2} & 0 & {\omega}^{2} & -3 {\omega}-2 {\omega}^{2} & {\omega} & 2 {\omega}-{\omega}^{2} & -2 & -2 {\omega}-3 {\omega}^{2} & 4 {\omega}^{2}\\
\chi_{27} & 4 & {\omega}^{2} & 0 & 3 {\omega}+{\omega}^{2} & {\omega} & {\omega} & 1 & 0 & {\omega}+3 {\omega}^{2} & -2 {\omega} & 1 & {\omega}^{2} & 1 & {\omega}^{2} & {\omega} & 0 & 0 & 1 & 2 {\omega}-{\omega}^{2} & {\omega}^{2} & {\omega} & 1 & {\omega} & 1 & 0 & -2 {\omega}^{2} & {\omega}^{2} & 4 {\omega}^{2} & 1 & {\omega}^{2} & {\omega}^{2} & 0 & {\omega} & -2 {\omega}-3 {\omega}^{2} & {\omega} & -{\omega}+2 {\omega}^{2} & -2 & -3 {\omega}-2 {\omega}^{2} & 4 {\omega}\\
\chi_{28} & 4 & {\omega} & 0 & {\omega}+3 {\omega}^{2} & {\omega}^{2} & {\omega}^{2} & 1 & 0 & 3 {\omega}+{\omega}^{2} & -2 {\omega}^{2} & 1 & {\omega} & 1 & {\omega} & {\omega}^{2} & 0 & 0 & 1 & -{\omega}+2 {\omega}^{2} & {\omega} & {\omega}^{2} & 1 & {\omega}^{2} & 1 & 0 & -2 {\omega} & {\omega} & 4 {\omega} & 1 & {\omega} & {\omega} & 0 & {\omega}^{2} & -3 {\omega}-2 {\omega}^{2} & {\omega}^{2} & 2 {\omega}-{\omega}^{2} & -2 & -2 {\omega}-3 {\omega}^{2} & 4 {\omega}^{2}\\
\chi_{29} & 4 & 1 & 0 & 3 {\omega}+{\omega}^{2} & 1 & {\omega}^{2} & {\omega} & 0 & {\omega}+3 {\omega}^{2} & -2 {\omega} & {\omega}^{2} & {\omega} & {\omega} & 1 & {\omega}^{2} & 0 & 0 & 1 & 2 {\omega}-{\omega}^{2} & {\omega} & 1 & {\omega}^{2} & {\omega}^{2} & {\omega} & 0 & -2 {\omega}^{2} & {\omega}^{2} & 4 {\omega}^{2} & {\omega}^{2} & {\omega} & 1 & 0 & {\omega} & -2 {\omega}-3 {\omega}^{2} & 1 & -{\omega}+2 {\omega}^{2} & -2 & -3 {\omega}-2 {\omega}^{2} & 4 {\omega}\\
\chi_{30} & 4 & 1 & 0 & {\omega}+3 {\omega}^{2} & 1 & {\omega} & {\omega}^{2} & 0 & 3 {\omega}+{\omega}^{2} & -2 {\omega}^{2} & {\omega} & {\omega}^{2} & {\omega}^{2} & 1 & {\omega} & 0 & 0 & 1 & -{\omega}+2 {\omega}^{2} & {\omega}^{2} & 1 & {\omega} & {\omega} & {\omega}^{2} & 0 & -2 {\omega} & {\omega} & 4 {\omega} & {\omega} & {\omega}^{2} & 1 & 0 & {\omega}^{2} & -3 {\omega}-2 {\omega}^{2} & 1 & 2 {\omega}-{\omega}^{2} & -2 & -2 {\omega}-3 {\omega}^{2} & 4 {\omega}^{2}\\
\chi_{31} & 6 & 0 & -2 & -3 & 0 & 0 & 0 & 1 & -3 & 3 & 0 & 0 & 0 & 0 & 0 & 1 & -2 & 0 & -3 & 0 & 0 & 0 & 0 & 0 & 1 & 3 & 0 & 6 & 0 & 0 & 0 & -2 & 0 & -3 & 0 & -3 & 3 & -3 & 6\\
\chi_{32} & 6 & 0 & 2 & -3 & 0 & 0 & 0 & -1 & -3 & 3 & 0 & 0 & 0 & 0 & 0 & -1 & 2 & 0 & -3 & 0 & 0 & 0 & 0 & 0 & -1 & 3 & 0 & 6 & 0 & 0 & 0 & 2 & 0 & -3 & 0 & -3 & 3 & -3 & 6\\
\chi_{33} & 6 & 0 & -2 & -3 {\omega} & 0 & 0 & 0 & {\omega} & -3 {\omega}^{2} & 3 {\omega}^{2} & 0 & 0 & 0 & 0 & 0 & {\omega}^{2} & -2 {\omega}^{2} & 0 & -3 & 0 & 0 & 0 & 0 & 0 & 1 & 3 {\omega} & 0 & 6 {\omega} & 0 & 0 & 0 & -2 {\omega} & 0 & -3 {\omega}^{2} & 0 & -3 & 3 & -3 {\omega} & 6 {\omega}^{2}\\
\chi_{34} & 6 & 0 & -2 & -3 {\omega}^{2} & 0 & 0 & 0 & {\omega}^{2} & -3 {\omega} & 3 {\omega} & 0 & 0 & 0 & 0 & 0 & {\omega} & -2 {\omega} & 0 & -3 & 0 & 0 & 0 & 0 & 0 & 1 & 3 {\omega}^{2} & 0 & 6 {\omega}^{2} & 0 & 0 & 0 & -2 {\omega}^{2} & 0 & -3 {\omega} & 0 & -3 & 3 & -3 {\omega}^{2} & 6 {\omega}\\
\chi_{35} & 6 & 0 & 2 & -3 {\omega} & 0 & 0 & 0 & -{\omega} & -3 {\omega}^{2} & 3 {\omega}^{2} & 0 & 0 & 0 & 0 & 0 & -{\omega}^{2} & 2 {\omega}^{2} & 0 & -3 & 0 & 0 & 0 & 0 & 0 & -1 & 3 {\omega} & 0 & 6 {\omega} & 0 & 0 & 0 & 2 {\omega} & 0 & -3 {\omega}^{2} & 0 & -3 & 3 & -3 {\omega} & 6 {\omega}^{2}\\
\chi_{36} & 6 & 0 & 2 & -3 {\omega}^{2} & 0 & 0 & 0 & -{\omega}^{2} & -3 {\omega} & 3 {\omega} & 0 & 0 & 0 & 0 & 0 & -{\omega} & 2 {\omega} & 0 & -3 & 0 & 0 & 0 & 0 & 0 & -1 & 3 {\omega}^{2} & 0 & 6 {\omega}^{2} & 0 & 0 & 0 & 2 {\omega}^{2} & 0 & -3 {\omega} & 0 & -3 & 3 & -3 {\omega}^{2} & 6 {\omega}\\
\chi_{37} & 12 & 0 & 0 & 3 & 0 & 0 & 0 & 0 & 3 & 0 & 0 & 0 & 0 & 0 & 0 & 0 & 0 & -3 & 3 & 0 & 0 & 0 & 0 & 0 & 0 & 0 & -3 & 12 & 0 & 0 & 0 & 0 & -3 & 3 & 0 & 3 & 0 & 3 & 12\\
\chi_{38} & 12 & 0 & 0 & 3 {\omega}^{2} & 0 & 0 & 0 & 0 & 3 {\omega} & 0 & 0 & 0 & 0 & 0 & 0 & 0 & 0 & -3 & 3 & 0 & 0 & 0 & 0 & 0 & 0 & 0 & -3 {\omega}^{2} & 12 {\omega}^{2} & 0 & 0 & 0 & 0 & -3 {\omega} & 3 {\omega} & 0 & 3 & 0 & 3 {\omega}^{2} & 12 {\omega}\\
\chi_{39} & 12 & 0 & 0 & 3 {\omega} & 0 & 0 & 0 & 0 & 3 {\omega}^{2} & 0 & 0 & 0 & 0 & 0 & 0 & 0 & 0 & -3 & 3 & 0 & 0 & 0 & 0 & 0 & 0 & 0 & -3 {\omega} & 12 {\omega} & 0 & 0 & 0 & 0 & -3 {\omega}^{2} & 3 {\omega}^{2} & 0 & 3 & 0 & 3 {\omega} & 12 {\omega}^{2}\\
\hline
\end{array}\)\\
\end{tabular}
	}}
	\caption{Ordinary character table of $N\rtimes H$}
	\label{table:ct_G_is_N_rtimes_H}
\title{}
\end{table}
\end{center}
\vspace*{\fill}
\end{landscape}
\newpage

\begin{properties}\label{Normalisers_GAP_Example}
Now, we calculate the normalisers of the $2$-subgroups of $G$ using \GAP\ and the result is as follows.

    \begin{table}[H]
	\centering
	{\scalebox{1}{
	\begin{tabular}{@{}l@{}l@{}l@{}l@{}l@{}l@{}l@{}l@{}l@{}l@{}l@{}l@{}l@{}l@{}}
		\(\begin{array}{lcr}
   N_{G}(Q_1) = {G};\ & N_{G}(Q_2) \cong C_6\times S_3\cong C_3\times D_{12};\ & N_{G}(Q_3) \cong {C_3\times \AlternatingGroup_4};\\[0,5ex]
   \overline{N}_{G}(Q_1) \cong {G}; & \overline{N}_{G}(Q_2) \cong C_3 \times S_3; & \overline{N}_{G}(Q_3) \cong C_3\times C_3.\\
		\end{array}\)
	\end{tabular}
}}
\end{table}
\end{properties}

\noindent These assertions follow, for example, from the following \GAP\ code:\par\medskip
\begin{footnotesize}
\begin{verbatim}
    Q2 := Group([A]);;  Q3 := Group([A,A*B*A*A*B*A*B*A]);;
    N2 := Normaliser(G,Q2); # returns: Group([ f2, f3, f4*f5, f6*f7 ])
    N3 := Normaliser(G,Q3); # returns: Group([ f1, f2, f3, f4*f5*f6*f7 ])
    StructureDescription(N2); # returns: "C6 x S3"
    StructureDescription(N3); # returns: "C3 x A4"
    nat := NaturalHomomorphismByNormalSubgroup(N2,Q2); # returns:
    # [ f2, f3, f4*f5, f6*f7 ] -> [ f1, <identity> of ..., f2, f3 ]
    facN2 := FactorGroup(N2,Q2); # returns: Group([ f1, <identity> of ..., f2, f3 ])
    StructureDescription(facN2); # returns: "C3 x S3"
    facN3 := FactorGroup(N3,Q3); # returns: 
    # Group([ f1, <identity> of ..., <identity> of ..., f2 ])
    StructureDescription(facN3); # returns: "C3 x C3"
\end{verbatim}
\end{footnotesize}
\vspace{0.2cm}

\noindent We understand the results of \cref{2blocks_of_G} theoretically:
\begin{enumerate}[label=\textup{(\alph*)}]
    \item It follows from \cref{Normalisers_GAP_Example} that $kN_G(Q_3)\cong k[C_3\times \AlternatingGroup_4]$ has exactly three $2$-blocks with defect groups isomorphic to $V_4$. Hence, by Brauer's first main theorem, the group $G$ has exactly three~$2$-blocks with defect groups isomorphic to $V_4$. We have already denoted these three blocks by $B_0(G)$, $B_1(G)$, and $B_2(G)$, respectively.
    \item It follows from \cref{Normalisers_GAP_Example} that $kN_G(Q_2)\cong C_3\times D_{12}$ has exactly three $2$-blocks with defect groups isomorphic to $C_2$. Hence, by Brauer's first main theorem, the group $G$ has exactly three~$2$-blocks with defect groups isomorphic to $C_2$. We have already denoted these three blocks by $B_{21}(G)$, $B_{22}(G)$, and $B_{23}(G)$, respectively.
\end{enumerate}

\begin{cor}\label{Cor:B_0(G)_is_splendidly_Morita_equivalent_to_kA_4}
    The block algebras $B_0(G)$, $B_1(G)$, and $B_2(G)$ are splendidly Morita equivalent to the group algebra $k\AlternatingGroup_4$.
\end{cor}

\begin{proof}
By inspecting $X(G)$, we see that $\Lin(G)=\{\chi_1, \ldots, \chi_9\}$. Hence, $\chi_i^\circ$ is an irreducible Brauer character of $G$ for every $i\in\{1,\ldots, 9\}$. 
Next, we observe that the nine simple $kG$-modules~$S_1,\ldots, S_9$ affording $\chi_1^\circ, \ldots , \chi_9^\circ$ must have maximal vertices, as their dimensions are coprime to $p$. Thus, they belong to blocks whose defect groups are isomorphic to a Klein four-group. Therefore, by \cref{2blocks_of_G}, the ordinary character degrees of the elements of $\Irr_K(B_0(G))$, $\Irr_K(B_1(G))$, and $\Irr_K(B_2(G))$ are given by $\{1,1,1,3\}$. It follows from \cref{CharDegreesDeterminePuigClass}
that each of the blocks $B_0(G)$, $B_1(G)$, and $B_2(G)$ is splendidly Morita equivalent to the block algebra~$k\AlternatingGroup_4$.
\end{proof}

\begin{rem}
    It also follows from Part 5 of \cite[Theorem]{KoshitaniGlasgow1981} combined with \cite[Theorem 4.1]{HidaKoshitani} that we must have splendid Morita equivalences between $B_0(G), B_1(G), B_2(G)$, and $k\AlternatingGroup_4$. However, we wanted to illustrate that it is possible to exploit the machinery developed in the previous section in order to prove the assertion.
\end{rem}

\begin{thm}
Set $w:=\exp(2\pi i/3)$. Then $\textup{Triv}_2(G)$ is given as follows.
	\begin{enumerate}[label=\rm(\alph{enumi}),itemsep=1ex]
        \item The labelling of the columns of the block columns of $\Triv_2(G)$ is as asserted in \cref{table:TSCT_N_rtimes_H_when_p_is_2} and \cref{table:TSCT_N_rtimes_H_when_p_is_2_T22_and_T32_and_T33}.
        \item We have $T_{1,2}=T_{1,3}=T_{2,3}={\mathbf{0}}$.
   			\item The matrices $T_{i,1}$ with $1\leq i\leq 3$ are as given in \cref{table:TSCT_N_rtimes_H_when_p_is_2}.
            \item The matrices $T_{2,2}$, $T_{3,2}$, and $T_{3,3}$ are as given in \cref{table:TSCT_N_rtimes_H_when_p_is_2_T22_and_T32_and_T33}.
    \end{enumerate}
\end{thm}

\begin{proof}
	\begin{enumerate}[label=\rm(\alph{enumi}), itemsep=0.5em, topsep=0.5em]
\item We consider the labelling of the $2^\prime$-conjugacy classes of $\Triv_2(G)$ which is given as follows: for the first block column, it coincides with the labelling of the $2^\prime$-conjugacy classes of $X(G)$ in \cref{table:ct_G_is_N_rtimes_H}; for the second and third block column, it is computed automatically by\ \GAP\ via the author's algorithm to compute trivial source character tables~(see \cite[Chapter 5 \& Chapter 7]{BBthesis}). Concretely, in our example, we have
\begin{small}
$$	1a^\prime {\sim_G} 1a,\quad 
            3a^\prime {\sim_G} 3f,\quad
            3b^\prime {\sim_G} 3l,\quad
		3c^\prime {\sim_G} 3n,\quad 
            3d^\prime {\sim_G} 3r,\quad
            3e^\prime {\sim_G} 3t,$$
$$ 	1a^{\prime\prime} \sim_G 1a,
            3a^{\prime\prime} \sim_G 3f,
            3b^{\prime\prime} \sim_G 3l,
		3c^{\prime\prime} \sim_G 3n, 
            3d^{\prime\prime} \sim_G 3r,
            3b^{\prime\prime} \sim_G 3l,
		3c^{\prime\prime} \sim_G 3n, 
            3d^{\prime\prime} \sim_G 3r,
            3e^{\prime\prime} \sim_G 3t.
$$
\end{small}

			\item This is immediate from \cref{rem:tsctbl}.
   \item \begin{itemize}[labelsep=0.2em, itemsep=0em, topsep=0em, left=0em]
		\item[$\cdot$] \underline{the matrix $T_{3,1}$}: by \cref{Omnibus_properties}, splendid Morita equivalences preserve trivial source modules and their vertices. Hence, as $B_0(G), B_1(G)$, and $B_2(G)$ are splendidly Morita equivalent to the group algebra $k\AlternatingGroup_4$, it follows from \cref{Proposition_ts_Puig_A4}
        that that the ordinary characters $\chi_1,\ldots, \chi_9$ have to occur as characters of trivial source modules with maximal vertices. But  \cref{Normalisers_GAP_Example} implies that $|\TS(G;Q_3)|=|{[N_G(Q_3)/Q_3]}_{2^\prime}|=9$. Therefore, we have already found all trivial source characters in this case. By \cref{LemmaRickard}, the entries of the matrix $T_{3,1}$ are now obtained by evaluating the ordinary characters at~$2^\prime$-conjugacy classes and can be read off from~$X(G)$ in \cref{table:ct_G_is_N_rtimes_H}. 
\vspace{0.2cm}
		\item[$\cdot$] \underline{the matrix $T_{2,1}$}: as the block algebras $B_0(G), B_1(G)$, and $B_2(G)$ are splendidly Morita equivalent to the group algebra $k\AlternatingGroup_4$, it follows from \cref{Proposition_ts_Puig_A4}
        that the ordinary characters $\chi_1 + \chi_4 + \chi_5 + \chi_{10}$, $\chi_3 + \chi_7 + \chi_8 + \chi_{12}$, and $\chi_2 + \chi_6 + \chi_9 + \chi_{11}$ have to occur as characters of trivial source modules with vertices of order two. Furthermore, by \cref{BrauerTreeTrivialSourceModules}, also the ordinary characters $\chi_{32}$, $\chi_{35}$, and $\chi_{36}$ have to occur as characters of trivial source modules with vertices of order two. But by \cref{Normalisers_GAP_Example}, we have $|\TS(G;Q_2)|=|{[N_G(Q_2)/Q_2]}_{2^\prime}|=6$. Therefore, we have already found all trivial source characters in this case.  By \cref{LemmaRickard}, the entries of the matrix $T_{2,1}$ are now obtained by evaluating the ordinary characters at $2^\prime$-conjugacy classes and can be read off from~$X(G)$ in \cref{table:ct_G_is_N_rtimes_H}.\vspace{0.2cm}
		\item[$\cdot$] \underline{the matrix $T_{1,1}$}: as the block algebras $B_0(G), B_1(G)$, and $B_2(G)$ are splendidly Morita equivalent to the group algebra~$k\AlternatingGroup_4$, it follows from \cref{Proposition_ts_Puig_A4}
        that ordinary characters labelling the first nine rows of the matrix $T_{1,1}$ have to occur as characters of trivial source modules. Furthermore, by \cref{BrauerTreeTrivialSourceModules}, also the ordinary characters~$\chi_{31} + \chi_{32}$, $\chi_{34} + \chi_{36}$, and $\chi_{33} + \chi_{35}$ have to occur as characters of PIMs of $kG$. All other characters belong to blocks of $kG$ of defect zero and are hence ordinary characters of trivial source modules which are simple and projective. By \cref{LemmaRickard}, the entries of the matrix $T_{1,1}$ are now obtained by evaluating the ordinary characters at $2^\prime$-conjugacy classes and can be read off from $X(G)$ in \cref{table:ct_G_is_N_rtimes_H}.
	\end{itemize}
   \item 
   	\begin{itemize}[labelsep=0.2em, itemsep=0em, topsep=0em, left=0em]
		\item[$\cdot$] \underline{the matrix $T_{2,2}$}: by the proof of Part (c), we already know the ordinary characters~$\chi_{\widehat{M}}$ of every trivial source module $M\in\TS(G;Q_2)$. Since we can write any~$g\in G$ as $g=g_pg_{p^\prime} = g_{p^\prime}g_p$, where $g_p$ denotes the $p$-part of $g$, it follows from \cref{LemmaRickard} that we have 
                $$
        \chi_{\widehat{M}}(g) 
            = \tau_{\langle g_p\rangle, g_{p^{\prime}}}^G([M])=\tau_{Q_2, g_{p^{\prime}}}^G([M])
        $$  
        
\begin{landscape}
\vspace*{\fill}
\begin{center}
\begin{table}[h]
	{\resizebox{22cm}{!}{
\begin{tabular}{@{}l@{}l@{}l@{}l@{}l@{}l@{}l@{}l@{}l@{}l@{}}
\(\begin{array}{|l V{4}ccccccccccccccccccccccccccccccccc|}
\hline
Q_v\ (1\leq v\leq 1) & \multicolumn{33}{c|}{Q_{1}\cong \TrivialGroup}\\ \hline
N_v\ (1\leq v\leq 1) & \multicolumn{33}{c|}{N_{1}\cong G}\\ \hline
n_j\ \in\ N_v & 1a & 3b & 3e & 3f & 3i & 3l & 3j & 3m & 3o & 3q & 3n & 3p & 3r & 3s & 3t & 3a & 9a & 9d & 3d & 9c & 9h & 3h & 9g & 9k & 3c & 9b & 9f & 3g & 9e & 9j & 3k & 9i & 9l\\ \Xhline{4\arrayrulewidth}
\chi_{1}+ \chi_{10} & 4 & 4 & 4 & 4 & 4 & 4 & 4 & 4 & 4 & 4 & 4 & 4 & 4 & 4 & 4 & 1 & 1 & 1 & 1 & 1 & 1 & 1 & 1 & 1 & 1 & 1 & 1 & 1 & 1 & 1 & 1 & 1 & 1\\
 \chi_{5}+ \chi_{10} & 4 & 4 & 4 & 4 & 4 & 4 & 4 & 4 & 4 & 4 & 4 & 4 & 4 & 4 & 4 & {\omega} & {\omega} & {\omega} & {\omega} & {\omega} & {\omega} & {\omega} & {\omega} & {\omega} & {\omega}^{2} & {\omega}^{2} & {\omega}^{2} & {\omega}^{2} & {\omega}^{2} & {\omega}^{2} & {\omega}^{2} & {\omega}^{2} & {\omega}^{2}\\
 \chi_{4}+ \chi_{10} & 4 & 4 & 4 & 4 & 4 & 4 & 4 & 4 & 4 & 4 & 4 & 4 & 4 & 4 & 4 & {\omega}^{2} & {\omega}^{2} & {\omega}^{2} & {\omega}^{2} & {\omega}^{2} & {\omega}^{2} & {\omega}^{2} & {\omega}^{2} & {\omega}^{2} & {\omega} & {\omega} & {\omega} & {\omega} & {\omega} & {\omega} & {\omega} & {\omega} & {\omega}\\
 \chi_{2}+ \chi_{11} & 4 & 4 {\omega}^{2} & 4 {\omega} & 4 {\omega} & 4 & 4 {\omega}^{2} & 4 & 4 {\omega}^{2} & 4 {\omega} & 4 & 4 {\omega}^{2} & 4 {\omega} & 4 & 4 {\omega}^{2} & 4 {\omega} & 1 & {\omega}^{2} & {\omega} & {\omega}^{2} & {\omega} & 1 & {\omega} & 1 & {\omega}^{2} & 1 & {\omega}^{2} & {\omega} & {\omega}^{2} & {\omega} & 1 & {\omega} & 1 & {\omega}^{2}\\
 \chi_{3}+ \chi_{12} & 4 & 4 {\omega} & 4 {\omega}^{2} & 4 {\omega}^{2} & 4 & 4 {\omega} & 4 & 4 {\omega} & 4 {\omega}^{2} & 4 & 4 {\omega} & 4 {\omega}^{2} & 4 & 4 {\omega} & 4 {\omega}^{2} & 1 & {\omega} & {\omega}^{2} & {\omega} & {\omega}^{2} & 1 & {\omega}^{2} & 1 & {\omega} & 1 & {\omega} & {\omega}^{2} & {\omega} & {\omega}^{2} & 1 & {\omega}^{2} & 1 & {\omega}\\
 \chi_{7}+ \chi_{12} & 4 & 4 {\omega} & 4 {\omega}^{2} & 4 {\omega}^{2} & 4 & 4 {\omega} & 4 & 4 {\omega} & 4 {\omega}^{2} & 4 & 4 {\omega} & 4 {\omega}^{2} & 4 & 4 {\omega} & 4 {\omega}^{2} & {\omega} & {\omega}^{2} & 1 & {\omega}^{2} & 1 & {\omega} & 1 & {\omega} & {\omega}^{2} & {\omega}^{2} & 1 & {\omega} & 1 & {\omega} & {\omega}^{2} & {\omega} & {\omega}^{2} & 1\\
 \chi_{6}+ \chi_{11} & 4 & 4 {\omega}^{2} & 4 {\omega} & 4 {\omega} & 4 & 4 {\omega}^{2} & 4 & 4 {\omega}^{2} & 4 {\omega} & 4 & 4 {\omega}^{2} & 4 {\omega} & 4 & 4 {\omega}^{2} & 4 {\omega} & {\omega}^{2} & {\omega} & 1 & {\omega} & 1 & {\omega}^{2} & 1 & {\omega}^{2} & {\omega} & {\omega} & 1 & {\omega}^{2} & 1 & {\omega}^{2} & {\omega} & {\omega}^{2} & {\omega} & 1\\
 \chi_{9}+ \chi_{11} & 4 & 4 {\omega}^{2} & 4 {\omega} & 4 {\omega} & 4 & 4 {\omega}^{2} & 4 & 4 {\omega}^{2} & 4 {\omega} & 4 & 4 {\omega}^{2} & 4 {\omega} & 4 & 4 {\omega}^{2} & 4 {\omega} & {\omega} & 1 & {\omega}^{2} & 1 & {\omega}^{2} & {\omega} & {\omega}^{2} & {\omega} & 1 & {\omega}^{2} & {\omega} & 1 & {\omega} & 1 & {\omega}^{2} & 1 & {\omega}^{2} & {\omega}\\
 \chi_{8}+ \chi_{12} & 4 & 4 {\omega} & 4 {\omega}^{2} & 4 {\omega}^{2} & 4 & 4 {\omega} & 4 & 4 {\omega} & 4 {\omega}^{2} & 4 & 4 {\omega} & 4 {\omega}^{2} & 4 & 4 {\omega} & 4 {\omega}^{2} & {\omega}^{2} & 1 & {\omega} & 1 & {\omega} & {\omega}^{2} & {\omega} & {\omega}^{2} & 1 & {\omega} & {\omega}^{2} & 1 & {\omega}^{2} & 1 & {\omega} & 1 & {\omega} & {\omega}^{2}\\
 \chi_{15} & 4 & -{\omega}+2 {\omega}^{2} & 2 {\omega}-{\omega}^{2} & -2 & 1 & -2 & 2 {\omega}-{\omega}^{2} & 1 & 1 & -{\omega}+2 {\omega}^{2} & 4 & -{\omega}+2 {\omega}^{2} & -2 & 2 {\omega}-{\omega}^{2} & 4 & {\omega} & 1 & {\omega}^{2} & {\omega} & 1 & {\omega}^{2} & {\omega} & 1 & {\omega}^{2} & {\omega}^{2} & {\omega} & 1 & {\omega}^{2} & {\omega} & 1 & {\omega}^{2} & {\omega} & 1\\
 \chi_{16} & 4 & 2 {\omega}-{\omega}^{2} & -{\omega}+2 {\omega}^{2} & -2 & 1 & -2 & -{\omega}+2 {\omega}^{2} & 1 & 1 & 2 {\omega}-{\omega}^{2} & 4 & 2 {\omega}-{\omega}^{2} & -2 & -{\omega}+2 {\omega}^{2} & 4 & {\omega}^{2} & 1 & {\omega} & {\omega}^{2} & 1 & {\omega} & {\omega}^{2} & 1 & {\omega} & {\omega} & {\omega}^{2} & 1 & {\omega} & {\omega}^{2} & 1 & {\omega} & {\omega}^{2} & 1\\
 \chi_{17} & 4 & -{\omega}+2 {\omega}^{2} & 2 {\omega}-{\omega}^{2} & -2 & 1 & -2 & 2 {\omega}-{\omega}^{2} & 1 & 1 & -{\omega}+2 {\omega}^{2} & 4 & -{\omega}+2 {\omega}^{2} & -2 & 2 {\omega}-{\omega}^{2} & 4 & {\omega}^{2} & {\omega} & 1 & {\omega}^{2} & {\omega} & 1 & {\omega}^{2} & {\omega} & 1 & {\omega} & 1 & {\omega}^{2} & {\omega} & 1 & {\omega}^{2} & {\omega} & 1 & {\omega}^{2}\\
 \chi_{18} & 4 & 2 {\omega}-{\omega}^{2} & -{\omega}+2 {\omega}^{2} & -2 & 1 & -2 & -{\omega}+2 {\omega}^{2} & 1 & 1 & 2 {\omega}-{\omega}^{2} & 4 & 2 {\omega}-{\omega}^{2} & -2 & -{\omega}+2 {\omega}^{2} & 4 & {\omega} & {\omega}^{2} & 1 & {\omega} & {\omega}^{2} & 1 & {\omega} & {\omega}^{2} & 1 & {\omega}^{2} & 1 & {\omega} & {\omega}^{2} & 1 & {\omega} & {\omega}^{2} & 1 & {\omega}\\
 \chi_{14} & 4 & 2 {\omega}-{\omega}^{2} & -{\omega}+2 {\omega}^{2} & -2 & 1 & -2 & -{\omega}+2 {\omega}^{2} & 1 & 1 & 2 {\omega}-{\omega}^{2} & 4 & 2 {\omega}-{\omega}^{2} & -2 & -{\omega}+2 {\omega}^{2} & 4 & 1 & {\omega} & {\omega}^{2} & 1 & {\omega} & {\omega}^{2} & 1 & {\omega} & {\omega}^{2} & 1 & {\omega} & {\omega}^{2} & 1 & {\omega} & {\omega}^{2} & 1 & {\omega} & {\omega}^{2}\\
 \chi_{13} & 4 & -{\omega}+2 {\omega}^{2} & 2 {\omega}-{\omega}^{2} & -2 & 1 & -2 & 2 {\omega}-{\omega}^{2} & 1 & 1 & -{\omega}+2 {\omega}^{2} & 4 & -{\omega}+2 {\omega}^{2} & -2 & 2 {\omega}-{\omega}^{2} & 4 & 1 & {\omega}^{2} & {\omega} & 1 & {\omega}^{2} & {\omega} & 1 & {\omega}^{2} & {\omega} & 1 & {\omega}^{2} & {\omega} & 1 & {\omega}^{2} & {\omega} & 1 & {\omega}^{2} & {\omega}\\
 \chi_{19} & 4 & -3 {\omega}-2 {\omega}^{2} & -2 {\omega}-3 {\omega}^{2} & -2 {\omega} & 1 & -2 {\omega}^{2} & -{\omega}+2 {\omega}^{2} & {\omega}^{2} & {\omega} & 2 {\omega}-{\omega}^{2} & 4 {\omega}^{2} & {\omega}+3 {\omega}^{2} & -2 & 3 {\omega}+{\omega}^{2} & 4 {\omega} & 1 & 1 & 1 & {\omega}^{2} & {\omega}^{2} & {\omega}^{2} & {\omega} & {\omega} & {\omega} & 1 & 1 & 1 & {\omega}^{2} & {\omega}^{2} & {\omega}^{2} & {\omega} & {\omega} & {\omega}\\
 \chi_{20} & 4 & -2 {\omega}-3 {\omega}^{2} & -3 {\omega}-2 {\omega}^{2} & -2 {\omega}^{2} & 1 & -2 {\omega} & 2 {\omega}-{\omega}^{2} & {\omega} & {\omega}^{2} & -{\omega}+2 {\omega}^{2} & 4 {\omega} & 3 {\omega}+{\omega}^{2} & -2 & {\omega}+3 {\omega}^{2} & 4 {\omega}^{2} & 1 & 1 & 1 & {\omega} & {\omega} & {\omega} & {\omega}^{2} & {\omega}^{2} & {\omega}^{2} & 1 & 1 & 1 & {\omega} & {\omega} & {\omega} & {\omega}^{2} & {\omega}^{2} & {\omega}^{2}\\
 \chi_{24} & 4 & -2 {\omega}-3 {\omega}^{2} & -3 {\omega}-2 {\omega}^{2} & -2 {\omega}^{2} & 1 & -2 {\omega} & 2 {\omega}-{\omega}^{2} & {\omega} & {\omega}^{2} & -{\omega}+2 {\omega}^{2} & 4 {\omega} & 3 {\omega}+{\omega}^{2} & -2 & {\omega}+3 {\omega}^{2} & 4 {\omega}^{2} & {\omega} & {\omega} & {\omega} & {\omega}^{2} & {\omega}^{2} & {\omega}^{2} & 1 & 1 & 1 & {\omega}^{2} & {\omega}^{2} & {\omega}^{2} & 1 & 1 & 1 & {\omega} & {\omega} & {\omega}\\
 \chi_{23} & 4 & -3 {\omega}-2 {\omega}^{2} & -2 {\omega}-3 {\omega}^{2} & -2 {\omega} & 1 & -2 {\omega}^{2} & -{\omega}+2 {\omega}^{2} & {\omega}^{2} & {\omega} & 2 {\omega}-{\omega}^{2} & 4 {\omega}^{2} & {\omega}+3 {\omega}^{2} & -2 & 3 {\omega}+{\omega}^{2} & 4 {\omega} & {\omega}^{2} & {\omega}^{2} & {\omega}^{2} & {\omega} & {\omega} & {\omega} & 1 & 1 & 1 & {\omega} & {\omega} & {\omega} & 1 & 1 & 1 & {\omega}^{2} & {\omega}^{2} & {\omega}^{2}\\
 \chi_{22} & 4 & -2 {\omega}-3 {\omega}^{2} & -3 {\omega}-2 {\omega}^{2} & -2 {\omega}^{2} & 1 & -2 {\omega} & 2 {\omega}-{\omega}^{2} & {\omega} & {\omega}^{2} & -{\omega}+2 {\omega}^{2} & 4 {\omega} & 3 {\omega}+{\omega}^{2} & -2 & {\omega}+3 {\omega}^{2} & 4 {\omega}^{2} & {\omega}^{2} & {\omega}^{2} & {\omega}^{2} & 1 & 1 & 1 & {\omega} & {\omega} & {\omega} & {\omega} & {\omega} & {\omega} & {\omega}^{2} & {\omega}^{2} & {\omega}^{2} & 1 & 1 & 1\\
 \chi_{21} & 4 & -3 {\omega}-2 {\omega}^{2} & -2 {\omega}-3 {\omega}^{2} & -2 {\omega} & 1 & -2 {\omega}^{2} & -{\omega}+2 {\omega}^{2} & {\omega}^{2} & {\omega} & 2 {\omega}-{\omega}^{2} & 4 {\omega}^{2} & {\omega}+3 {\omega}^{2} & -2 & 3 {\omega}+{\omega}^{2} & 4 {\omega} & {\omega} & {\omega} & {\omega} & 1 & 1 & 1 & {\omega}^{2} & {\omega}^{2} & {\omega}^{2} & {\omega}^{2} & {\omega}^{2} & {\omega}^{2} & {\omega} & {\omega} & {\omega} & 1 & 1 & 1\\
 \chi_{28} & 4 & {\omega}+3 {\omega}^{2} & 3 {\omega}+{\omega}^{2} & -2 {\omega}^{2} & 1 & -2 {\omega} & -{\omega}+2 {\omega}^{2} & {\omega} & {\omega}^{2} & 2 {\omega}-{\omega}^{2} & 4 {\omega} & -3 {\omega}-2 {\omega}^{2} & -2 & -2 {\omega}-3 {\omega}^{2} & 4 {\omega}^{2} & {\omega} & 1 & {\omega}^{2} & {\omega}^{2} & {\omega} & 1 & 1 & {\omega}^{2} & {\omega} & {\omega}^{2} & {\omega} & 1 & 1 & {\omega}^{2} & {\omega} & {\omega} & 1 & {\omega}^{2}\\
 \chi_{27} & 4 & 3 {\omega}+{\omega}^{2} & {\omega}+3 {\omega}^{2} & -2 {\omega} & 1 & -2 {\omega}^{2} & 2 {\omega}-{\omega}^{2} & {\omega}^{2} & {\omega} & -{\omega}+2 {\omega}^{2} & 4 {\omega}^{2} & -2 {\omega}-3 {\omega}^{2} & -2 & -3 {\omega}-2 {\omega}^{2} & 4 {\omega} & {\omega}^{2} & 1 & {\omega} & {\omega} & {\omega}^{2} & 1 & 1 & {\omega} & {\omega}^{2} & {\omega} & {\omega}^{2} & 1 & 1 & {\omega} & {\omega}^{2} & {\omega}^{2} & 1 & {\omega}\\
 \chi_{29} & 4 & 3 {\omega}+{\omega}^{2} & {\omega}+3 {\omega}^{2} & -2 {\omega} & 1 & -2 {\omega}^{2} & 2 {\omega}-{\omega}^{2} & {\omega}^{2} & {\omega} & -{\omega}+2 {\omega}^{2} & 4 {\omega}^{2} & -2 {\omega}-3 {\omega}^{2} & -2 & -3 {\omega}-2 {\omega}^{2} & 4 {\omega} & 1 & {\omega} & {\omega}^{2} & {\omega}^{2} & 1 & {\omega} & {\omega} & {\omega}^{2} & 1 & 1 & {\omega} & {\omega}^{2} & {\omega}^{2} & 1 & {\omega} & {\omega} & {\omega}^{2} & 1\\
 \chi_{30} & 4 & {\omega}+3 {\omega}^{2} & 3 {\omega}+{\omega}^{2} & -2 {\omega}^{2} & 1 & -2 {\omega} & -{\omega}+2 {\omega}^{2} & {\omega} & {\omega}^{2} & 2 {\omega}-{\omega}^{2} & 4 {\omega} & -3 {\omega}-2 {\omega}^{2} & -2 & -2 {\omega}-3 {\omega}^{2} & 4 {\omega}^{2} & 1 & {\omega}^{2} & {\omega} & {\omega} & 1 & {\omega}^{2} & {\omega}^{2} & {\omega} & 1 & 1 & {\omega}^{2} & {\omega} & {\omega} & 1 & {\omega}^{2} & {\omega}^{2} & {\omega} & 1\\
 \chi_{26} & 4 & {\omega}+3 {\omega}^{2} & 3 {\omega}+{\omega}^{2} & -2 {\omega}^{2} & 1 & -2 {\omega} & -{\omega}+2 {\omega}^{2} & {\omega} & {\omega}^{2} & 2 {\omega}-{\omega}^{2} & 4 {\omega} & -3 {\omega}-2 {\omega}^{2} & -2 & -2 {\omega}-3 {\omega}^{2} & 4 {\omega}^{2} & {\omega}^{2} & {\omega} & 1 & 1 & {\omega}^{2} & {\omega} & {\omega} & 1 & {\omega}^{2} & {\omega} & 1 & {\omega}^{2} & {\omega}^{2} & {\omega} & 1 & 1 & {\omega}^{2} & {\omega}\\
 \chi_{25} & 4 & 3 {\omega}+{\omega}^{2} & {\omega}+3 {\omega}^{2} & -2 {\omega} & 1 & -2 {\omega}^{2} & 2 {\omega}-{\omega}^{2} & {\omega}^{2} & {\omega} & -{\omega}+2 {\omega}^{2} & 4 {\omega}^{2} & -2 {\omega}-3 {\omega}^{2} & -2 & -3 {\omega}-2 {\omega}^{2} & 4 {\omega} & {\omega} & {\omega}^{2} & 1 & 1 & {\omega} & {\omega}^{2} & {\omega}^{2} & 1 & {\omega} & {\omega}^{2} & 1 & {\omega} & {\omega} & {\omega}^{2} & 1 & 1 & {\omega} & {\omega}^{2}\\
\chi_{31}+ \chi_{32} & 12 & -6 & -6 & 6 & 0 & 6 & -6 & 0 & 0 & -6 & 12 & -6 & 6 & -6 & 12 & 0 & 0 & 0 & 0 & 0 & 0 & 0 & 0 & 0 & 0 & 0 & 0 & 0 & 0 & 0 & 0 & 0 & 0\\
 \chi_{34}+ \chi_{36} & 12 & -6 {\omega}^{2} & -6 {\omega} & 6 {\omega} & 0 & 6 {\omega}^{2} & -6 & 0 & 0 & -6 & 12 {\omega}^{2} & -6 {\omega} & 6 & -6 {\omega}^{2} & 12 {\omega} & 0 & 0 & 0 & 0 & 0 & 0 & 0 & 0 & 0 & 0 & 0 & 0 & 0 & 0 & 0 & 0 & 0 & 0\\
 \chi_{33}+ \chi_{35} & 12 & -6 {\omega} & -6 {\omega}^{2} & 6 {\omega}^{2} & 0 & 6 {\omega} & -6 & 0 & 0 & -6 & 12 {\omega} & -6 {\omega}^{2} & 6 & -6 {\omega} & 12 {\omega}^{2} & 0 & 0 & 0 & 0 & 0 & 0 & 0 & 0 & 0 & 0 & 0 & 0 & 0 & 0 & 0 & 0 & 0 & 0\\
 \chi_{37} & 12 & 3 & 3 & 0 & -3 & 0 & 3 & -3 & -3 & 3 & 12 & 3 & 0 & 3 & 12 & 0 & 0 & 0 & 0 & 0 & 0 & 0 & 0 & 0 & 0 & 0 & 0 & 0 & 0 & 0 & 0 & 0 & 0\\
 \chi_{38} & 12 & 3 {\omega}^{2} & 3 {\omega} & 0 & -3 & 0 & 3 & -3 {\omega}^{2} & -3 {\omega} & 3 & 12 {\omega}^{2} & 3 {\omega} & 0 & 3 {\omega}^{2} & 12 {\omega} & 0 & 0 & 0 & 0 & 0 & 0 & 0 & 0 & 0 & 0 & 0 & 0 & 0 & 0 & 0 & 0 & 0 & 0\\
 \chi_{39} & 12 & 3 {\omega} & 3 {\omega}^{2} & 0 & -3 & 0 & 3 & -3 {\omega} & -3 {\omega}^{2} & 3 & 12 {\omega} & 3 {\omega}^{2} & 0 & 3 {\omega} & 12 {\omega}^{2} & 0 & 0 & 0 & 0 & 0 & 0 & 0 & 0 & 0 & 0 & 0 & 0 & 0 & 0 & 0 & 0 & 0 & 0\\
\hline
 \chi_{1}+ \chi_{4}+ \chi_{5}+ \chi_{10} & 6 & 6 & 6 & 6 & 6 & 6 & 6 & 6 & 6 & 6 & 6 & 6 & 6 & 6 & 6 & 0 & 0 & 0 & 0 & 0 & 0 & 0 & 0 & 0 & 0 & 0 & 0 & 0 & 0 & 0 & 0 & 0 & 0\\
 \chi_{3}+ \chi_{7}+ \chi_{8}+ \chi_{12} & 6 & 6 {\omega} & 6 {\omega}^{2} & 6 {\omega}^{2} & 6 & 6 {\omega} & 6 & 6 {\omega} & 6 {\omega}^{2} & 6 & 6 {\omega} & 6 {\omega}^{2} & 6 & 6 {\omega} & 6 {\omega}^{2} & 0 & 0 & 0 & 0 & 0 & 0 & 0 & 0 & 0 & 0 & 0 & 0 & 0 & 0 & 0 & 0 & 0 & 0\\
 \chi_{2}+ \chi_{6}+ \chi_{9}+ \chi_{11} & 6 & 6 {\omega}^{2} & 6 {\omega} & 6 {\omega} & 6 & 6 {\omega}^{2} & 6 & 6 {\omega}^{2} & 6 {\omega} & 6 & 6 {\omega}^{2} & 6 {\omega} & 6 & 6 {\omega}^{2} & 6 {\omega} & 0 & 0 & 0 & 0 & 0 & 0 & 0 & 0 & 0 & 0 & 0 & 0 & 0 & 0 & 0 & 0 & 0 & 0\\
 \chi_{32} & 6 & -3 & -3 & 3 & 0 & 3 & -3 & 0 & 0 & -3 & 6 & -3 & 3 & -3 & 6 & 0 & 0 & 0 & 0 & 0 & 0 & 0 & 0 & 0 & 0 & 0 & 0 & 0 & 0 & 0 & 0 & 0 & 0\\
 \chi_{36} & 6 & -3 {\omega}^{2} & -3 {\omega} & 3 {\omega} & 0 & 3 {\omega}^{2} & -3 & 0 & 0 & -3 & 6 {\omega}^{2} & -3 {\omega} & 3 & -3 {\omega}^{2} & 6 {\omega} & 0 & 0 & 0 & 0 & 0 & 0 & 0 & 0 & 0 & 0 & 0 & 0 & 0 & 0 & 0 & 0 & 0 & 0\\
 \chi_{35} & 6 & -3 {\omega} & -3 {\omega}^{2} & 3 {\omega}^{2} & 0 & 3 {\omega} & -3 & 0 & 0 & -3 & 6 {\omega} & -3 {\omega}^{2} & 3 & -3 {\omega} & 6 {\omega}^{2} & 0 & 0 & 0 & 0 & 0 & 0 & 0 & 0 & 0 & 0 & 0 & 0 & 0 & 0 & 0 & 0 & 0 & 0\\
\hline
 \chi_{1} & 1 & 1 & 1 & 1 & 1 & 1 & 1 & 1 & 1 & 1 & 1 & 1 & 1 & 1 & 1 & 1 & 1 & 1 & 1 & 1 & 1 & 1 & 1 & 1 & 1 & 1 & 1 & 1 & 1 & 1 & 1 & 1 & 1\\
 \chi_{5} & 1 & 1 & 1 & 1 & 1 & 1 & 1 & 1 & 1 & 1 & 1 & 1 & 1 & 1 & 1 & {\omega} & {\omega} & {\omega} & {\omega} & {\omega} & {\omega} & {\omega} & {\omega} & {\omega} & {\omega}^{2} & {\omega}^{2} & {\omega}^{2} & {\omega}^{2} & {\omega}^{2} & {\omega}^{2} & {\omega}^{2} & {\omega}^{2} & {\omega}^{2}\\
 \chi_{4} & 1 & 1 & 1 & 1 & 1 & 1 & 1 & 1 & 1 & 1 & 1 & 1 & 1 & 1 & 1 & {\omega}^{2} & {\omega}^{2} & {\omega}^{2} & {\omega}^{2} & {\omega}^{2} & {\omega}^{2} & {\omega}^{2} & {\omega}^{2} & {\omega}^{2} & {\omega} & {\omega} & {\omega} & {\omega} & {\omega} & {\omega} & {\omega} & {\omega} & {\omega}\\
 \chi_{3} & 1 & {\omega} & {\omega}^{2} & {\omega}^{2} & 1 & {\omega} & 1 & {\omega} & {\omega}^{2} & 1 & {\omega} & {\omega}^{2} & 1 & {\omega} & {\omega}^{2} & 1 & {\omega} & {\omega}^{2} & {\omega} & {\omega}^{2} & 1 & {\omega}^{2} & 1 & {\omega} & 1 & {\omega} & {\omega}^{2} & {\omega} & {\omega}^{2} & 1 & {\omega}^{2} & 1 & {\omega}\\
 \chi_{2} & 1 & {\omega}^{2} & {\omega} & {\omega} & 1 & {\omega}^{2} & 1 & {\omega}^{2} & {\omega} & 1 & {\omega}^{2} & {\omega} & 1 & {\omega}^{2} & {\omega} & 1 & {\omega}^{2} & {\omega} & {\omega}^{2} & {\omega} & 1 & {\omega} & 1 & {\omega}^{2} & 1 & {\omega}^{2} & {\omega} & {\omega}^{2} & {\omega} & 1 & {\omega} & 1 & {\omega}^{2}\\
 \chi_{7} & 1 & {\omega} & {\omega}^{2} & {\omega}^{2} & 1 & {\omega} & 1 & {\omega} & {\omega}^{2} & 1 & {\omega} & {\omega}^{2} & 1 & {\omega} & {\omega}^{2} & {\omega} & {\omega}^{2} & 1 & {\omega}^{2} & 1 & {\omega} & 1 & {\omega} & {\omega}^{2} & {\omega}^{2} & 1 & {\omega} & 1 & {\omega} & {\omega}^{2} & {\omega} & {\omega}^{2} & 1\\
 \chi_{6} & 1 & {\omega}^{2} & {\omega} & {\omega} & 1 & {\omega}^{2} & 1 & {\omega}^{2} & {\omega} & 1 & {\omega}^{2} & {\omega} & 1 & {\omega}^{2} & {\omega} & {\omega}^{2} & {\omega} & 1 & {\omega} & 1 & {\omega}^{2} & 1 & {\omega}^{2} & {\omega} & {\omega} & 1 & {\omega}^{2} & 1 & {\omega}^{2} & {\omega} & {\omega}^{2} & {\omega} & 1\\
 \chi_{8} & 1 & {\omega} & {\omega}^{2} & {\omega}^{2} & 1 & {\omega} & 1 & {\omega} & {\omega}^{2} & 1 & {\omega} & {\omega}^{2} & 1 & {\omega} & {\omega}^{2} & {\omega}^{2} & 1 & {\omega} & 1 & {\omega} & {\omega}^{2} & {\omega} & {\omega}^{2} & 1 & {\omega} & {\omega}^{2} & 1 & {\omega}^{2} & 1 & {\omega} & 1 & {\omega} & {\omega}^{2}\\
 \chi_{9} & 1 & {\omega}^{2} & {\omega} & {\omega} & 1 & {\omega}^{2} & 1 & {\omega}^{2} & {\omega} & 1 & {\omega}^{2} & {\omega} & 1 & {\omega}^{2} & {\omega} & {\omega} & 1 & {\omega}^{2} & 1 & {\omega}^{2} & {\omega} & {\omega}^{2} & {\omega} & 1 & {\omega}^{2} & {\omega} & 1 & {\omega} & 1 & {\omega}^{2} & 1 & {\omega}^{2} & {\omega}\\
\hline

\end{array}\)\\
\end{tabular}
	}}
	\caption{$T_{i,1}$ for $1\leq i\leq 3$
 }
	\label{table:TSCT_N_rtimes_H_when_p_is_2}
\title{}
\end{table}
\end{center}
\vspace*{\fill}
\end{landscape}
\newpage
        for each $M\in\TS(G;Q_2)$, as $Q_2$ has order two and is normal in $N_2$. Thus, the entries of the matrix $T_{2,2}$ are obtained by evaluating the ordinary characters of $G$ at certain conjugacy classes of $G$ and can therefore be read off from $X(G)$ in \cref{table:ct_G_is_N_rtimes_H}. 
\vspace{0.2cm}
		\item[$\cdot$] \underline{the matrix $T_{3,2}$}: by \cref{LemmaRickard}, the entries of the matrix $T_{3,2}$ are obtained by evaluating the ordinary characters of $G$ at certain conjugacy classes of $G$ and can therefore be read off from $X(G)$ in \cref{table:ct_G_is_N_rtimes_H}.
\vspace{0.2cm}
		\item[$\cdot$] \underline{the matrix $T_{3,3}$}: it follows from 
        \cref{conv:tsctbl} and 
        \cite[Theorem 4.14]{LASpPermSuryey}
        that the entries of $T_{3,3}$ are given by evaluating the restriction of the ordinary characters of the trivial source $kG$-modules with maximal vertices from $G$ to $N_G(Q_3)$ at the~$2^\prime$-conjugacy classes of $N_G(Q_3)/Q_3$. Using Part~(a), we see that these evaluations are as claimed.\qedhere
	\end{itemize}
	\end{enumerate}
\end{proof}

\begin{center}
\begin{table}[h]
	{\resizebox{10.2cm}{!}{
\begin{tabular}{@{}l@{}l@{}l@{}l@{}l@{}l@{}l@{}l@{}l@{}l@{}}
\(\begin{array}{|l V{4}cccccc|ccccccccc|}
\hline
Q_v\ (2\leq v\leq 3) & \multicolumn{6}{c|}{Q_{2}\cong C_2} & \multicolumn{9}{c|}{Q_{3}\cong V_4}\\ \hline
N_v\ (2\leq v\leq 3) & \multicolumn{6}{c|}{N_{2}\cong C_3\times D_{12}} & \multicolumn{9}{c|}{N_{3}\cong C_3\times \AlternatingGroup_4}\\ \hline
n_j\ \in\ N_v & 1a' & 3a' & 3b' & 3c' & 3d' & 3e' & 1a'' & 3a'' & 3b'' & 3c'' & 3d'' & 3h'' & 3e'' & 3f'' & 3g''\\ \Xhline{4\arrayrulewidth}
\chi_{1}+ \chi_{4}+ \chi_{5}+ \chi_{10} & 2 & 2 & 2 & 2 & 2 & 2 & 0 & 0 & 0 & 0 & 0 & 0 & 0 & 0 & 0\\
 \chi_{3}+ \chi_{7}+ \chi_{8}+ \chi_{12} & 2 & 2 {\omega}^{2} & 2 {\omega} & 2 {\omega} & 2 & 2 {\omega}^{2} & 0 & 0 & 0 & 0 & 0 & 0 & 0 & 0 & 0\\
 \chi_{2}+ \chi_{6}+ \chi_{9}+ \chi_{11} & 2 & 2 {\omega} & 2 {\omega}^{2} & 2 {\omega}^{2} & 2 & 2 {\omega} & 0 & 0 & 0 & 0 & 0 & 0 & 0 & 0 & 0\\
 \chi_{32} & 2 & -1 & -1 & 2 & -1 & 2 & 0 & 0 & 0 & 0 & 0 & 0 & 0 & 0 & 0\\
 \chi_{36} & 2 & -{\omega} & -{\omega}^{2} & 2 {\omega}^{2} & -1 & 2 {\omega} & 0 & 0 & 0 & 0 & 0 & 0 & 0 & 0 & 0\\
 \chi_{35} & 2 & -{\omega}^{2} & -{\omega} & 2 {\omega} & -1 & 2 {\omega}^{2} & 0 & 0 & 0 & 0 & 0 & 0 & 0 & 0 & 0\\
\hline
 \chi_{1} & 1 & 1 & 1 & 1 & 1 & 1 & 1 & 1 & 1 & 1 & 1 & 1 & 1 & 1 & 1\\
 \chi_{5} & 1 & 1 & 1 & 1 & 1 & 1 & 1 & {\omega} & 1 & {\omega}^{2} & {\omega} & 1 & {\omega}^{2} & {\omega} & {\omega}^{2}\\
 \chi_{4} & 1 & 1 & 1 & 1 & 1 & 1 & 1 & {\omega}^{2} & 1 & {\omega} & {\omega}^{2} & 1 & {\omega} & {\omega}^{2} & {\omega}\\
 \chi_{3} & 1 & {\omega}^{2} & {\omega} & {\omega} & 1 & {\omega}^{2} & 1 & 1 & {\omega} & 1 & {\omega} & {\omega}^{2} & {\omega} & {\omega}^{2} & {\omega}^{2}\\
 \chi_{2} & 1 & {\omega} & {\omega}^{2} & {\omega}^{2} & 1 & {\omega} & 1 & 1 & {\omega}^{2} & 1 & {\omega}^{2} & {\omega} & {\omega}^{2} & {\omega} & {\omega}\\
 \chi_{7} & 1 & {\omega}^{2} & {\omega} & {\omega} & 1 & {\omega}^{2} & 1 & {\omega} & {\omega} & {\omega}^{2} & {\omega}^{2} & {\omega}^{2} & 1 & 1 & {\omega}\\
 \chi_{6} & 1 & {\omega} & {\omega}^{2} & {\omega}^{2} & 1 & {\omega} & 1 & {\omega}^{2} & {\omega}^{2} & {\omega} & {\omega} & {\omega} & 1 & 1 & {\omega}^{2}\\
 \chi_{8} & 1 & {\omega}^{2} & {\omega} & {\omega} & 1 & {\omega}^{2} & 1 & {\omega}^{2} & {\omega} & {\omega} & 1 & {\omega}^{2} & {\omega}^{2} & {\omega} & 1\\
 \chi_{9} & 1 & {\omega} & {\omega}^{2} & {\omega}^{2} & 1 & {\omega} & 1 & {\omega} & {\omega}^{2} & {\omega}^{2} & 1 & {\omega} & {\omega} & {\omega}^{2} & 1\\
\hline

\end{array}\)\\
\end{tabular}
	}}
	\caption{$T_{i,j}$ for $2\leq i,j\leq 3$}
	\label{table:TSCT_N_rtimes_H_when_p_is_2_T22_and_T32_and_T33}
\title{}
\end{table}
\end{center}

\noindent \textbf{Acknowledgments.}
The author is thankful to Caroline Lassueur for valuable comments on an earlier draft of this note and gratefully acknowledges financial support by the DFG-SFB/TRR195.


	\nocite{}
	\bibliographystyle{aomalpha}
	\bibliography{Biblio.bib}
	\bigskip
	


\end{document}